\documentclass[letterpaper,
               10pt
              ]{amsart}
              
\usepackage[margin=1.2in]{geometry}
\usepackage{amsmath,amsthm,amssymb}
\usepackage{xspace,xcolor}
\usepackage[breaklinks,colorlinks,citecolor=teal,linkcolor=teal,urlcolor=teal,pagebackref,hyperindex]{hyperref}
\usepackage[alphabetic]{amsrefs}
\usepackage[all]{xy}
\usepackage{enumitem}
\usepackage{tikz,xcolor}
\usepackage{tikz-cd}
\usepackage{color}

\usepackage{thmtools}
\usepackage{thm-restate}
\usepackage{mathrsfs}

\theoremstyle{plain}
\newtheorem{thm}{Theorem}[section]

\newtheorem{lem}[thm]{Lemma}
\newtheorem{prop}[thm]{Proposition}
\newtheorem{cor}[thm]{Corollary}

\theoremstyle{definition}
\newtheorem{defi}[thm]{Definition}

\newtheorem{egs}[thm]{Examples}

\theoremstyle{remark}
\newtheorem{rmk}[thm]{Remark}

\def\Mustata{Mus\-ta\-\c{t}\u{a}\xspace}

\def\N{{\mathbf N}}
\def\Z{{\mathbf Z}}
\def\Q{{\mathbf Q}}

\def\C{{\mathbf C}}

\def\A{{\mathbf A}}
\def\P{{\mathbf P}}

\def\cA{\mathcal{A}}
\def\cB{\mathcal{B}}
\def\cC{\mathcal{C}}
\def\cD{\mathcal{D}}
\def\cE{\mathcal{E}}
\def\cF{\mathcal{F}}

\def\cH{\mathcal{H}}
\def\cI{\mathcal{I}}
\def\cJ{\mathcal{J}}
\def\cK{\mathcal{K}}

\def\cM{\mathcal{M}}
\def\cN{\mathcal{N}}
\def\cO{\mathcal{O}}

\def\cS{\mathcal{S}}
\def\cT{\mathcal{T}}
\def\cU{\mathcal{U}}

\def\bD{{\bf D}}
\def\bR{{\bf R}}

\def\.{\cdot}
\def\^{\widehat}

\def\de{\partial}

\def\({\left(}
\def\){\right)}

\def\ad{\mathrm{a}\mathrm{d}\,}
\def\DR{\mathrm{D}\mathrm{R}}
\def\Sp{\mathrm{S}\mathrm{p}}
\def\gr{\mathrm{gr}}
\def\fsl{\mathfrak{sl}}
\def\ker{\mathrm{ker}\,}
\def\FL{\mathsf{F}\mathsf{L}}
\def\MHM{\mathrm{M}\mathrm{H}\mathrm{M}}

\renewcommand{\and}{ \ \ \text{ and } \ \ }

\DeclareMathOperator{\coker} {coker}

\DeclareMathOperator{\Proj} {Proj}

\DeclareMathOperator{\Hom} {Hom}

\DeclareMathOperator{\pr} {pr}

\begin{document}

\title{On V-filtration, Hodge filtration and Fourier transform}

\author{Qianyu Chen}
\author{Bradley Dirks}

\address{Department of Mathematics, University of Michigan, 530 Church Street, Ann Arbor, MI 48109, USA}
\email{qyc@umich.edu}

\address{Department of Mathematics, University of Michigan, 530 Church Street, Ann Arbor, MI 48109, USA}

\email{bdirks@umich.edu}

\date{\today}
\thanks{BD was partially supported by NSF grant DMS-1840234.}


\begin{abstract} For $i:Z\hookrightarrow X$ a closed immersion of smooth varieties, we study how the $V$-filtration along $Z$ and the Hodge filtration on a mixed Hodge module $\cM$ on $X$ interact with each other. We also give a formula for the functors $i^*$, $i^!$ in terms of this $V$-filtration. As applications, we obtain results on the Hodge filtration of monodromic mixed Hodge modules and we give a Hodge theoretic proof of Skoda's theorem on multiplier ideals. Finally, we use the results to study the Fourier-Laplace transform of a monodromic mixed Hodge module.
\end{abstract}

\maketitle

\section{Introduction}

\subsection{Motivation} 
For $X$ a smooth complex algebraic variety, Saito's theory of mixed Hodge modules on $X$ ~\cites{Saito-MHP,Saito-MHM} provides a vast generalization of the theory of variations of Hodge structure on $X$ (also see the survey ~\cite{Schnell-MHM}).

The main objects of study are holonomic $\cD_X$-modules equipped with a good filtration, called the Hodge filtration. Very roughly, mixed Hodge modules are defined inductively by forcing their restriction to hypersurfaces to be mixed Hodge modules. This restriction is defined using the $V$-filtration of Kashiwara and Malgrange (see~\cite{Kas83} and~\cite{Mal83} and Section \ref{subsec:VFilt} below). To have a satisfactory theory, then, it is important to require some sort of compatibility between the Hodge filtration $F$ and the $V$-filtration (see ~\cite{Saito-MHP}*{Section 3.2} or ~\cite{Schnell-MHM}*{Section 11}). For $Z \subseteq X$ a smooth subvariety of higher codimension, there is still a notion of $V$-filtration along $Z$ for holonomic $\cD_X$-modules. One of our main theorems proves an analogous compatibility condition between the Hodge filtration of a filtered $\cD_X$-module underlying a mixed Hodge module and this $V$-filtration along $Z$.

If one wants to restrict mixed Hodge modules to a smooth subvariety $Z\subseteq X$ of higher codimension, this is done by writing $Z$ locally as an intersection of smooth hypersurfaces, and then step-by-step restricting from one hypersurface to the next. Another main theorem gives a way to restrict mixed Hodge modules on $X$ to mixed Hodge modules on $Z$ in a single step (see Theorem ~\ref{thm:main}), using the $V$-filtration along $Z$. 

\subsection{Main Results} Let $f=(f_1,\dots,f_r):X\rightarrow \A^r$ be a smooth morphism between smooth algebraic varieties over $\C$, where $\A^r$ is the affine $r$-space with coordinates $(t_1,\dots, t_r)$. Let $Z\subseteq X$ be the fiber over the origin. Assume there exist global vector fields $\partial_1,\partial_2,\dots,\partial_r$ on $X$ dual to the one-forms $df_1,df_2,\dots,df_r$. Let $\cD_X$ be the sheaf of differential operators on $X$. 

When $r=1$, we have a smooth function $t$ and a global vector field $\partial_t$ such that $[\partial_t, t]=1$. We have already mentioned the $V$-filtration along the hypersurface defined by $t$ on a holonomic $\cD$-module $\cM$. Slightly more precisely (see Section \ref{subsec:VFilt} for a more detailed description), this is a decreasing filtration $V^\bullet \cM$, indexed by $\Q$, such that
\begin{enumerate}
    \item $t V^\alpha \cM \subseteq V^{\alpha+1}\cM$, with equality if $\alpha >0$,
    \item $\de_t V^\alpha \cM \subseteq V^{\alpha-1}\cM$,
    \item $t\de_t - \alpha +1$ is nilpotent on $\gr^{\alpha}_V \cM  = V^{\alpha}\cM/V^{>\alpha}\cM$.
\end{enumerate} 

In this case, the graded quotients $\gr_V^\alpha\cM$ are holonomic $\cD_Z$ modules that are used to define nearby cycles and vanishing cycles of $\cM$.  If a filtered $\cD_X$-module $(\cM,F)$ underlies a mixed Hodge module, then it is \emph{quasi-unipotent and regular} along a hypersurface (\cite{Saito-MHP}*{3.2} and see also \cite{Schnell-MHM}*{11.4}). This is a compatibility condition between the Hodge filtration $F$ and the $V$-filtration. By definition, it requires 
\begin{enumerate}
    \item $t: F_p V^{\alpha} \cM \rightarrow F_p V^{\alpha+1}\cM$ is an isomorphism for $\alpha >0$,
    \item $\partial_t: F_p\gr_V^{\alpha+1}\cM\rightarrow F_{p+1}\gr_V^{\alpha}\cM$ is an isomorphism for $\alpha< 0$.
\end{enumerate} 

In fact, all filtered $\cD_X$-modules underlying a mixed Hodge module on $X$ satisfy this property for \emph{any} locally defined function $g$. Also by the theory of Hodge modules, we have two distinguished triangles in the derived category of mixed Hodge modules on $Z$
\begin{equation}\label{eq:dis}
i^*\cM[-1] \to  \gr_V^{1}\cM \xrightarrow[]{\partial_t} \gr_V^0\cM  \to  i^*\cM \quad \text{and} \quad i^!\cM \to  \gr_V^0\cM \xrightarrow[]{t} \gr_V^{1}\cM  \to  i^!\cM[1],
\end{equation}
where $i:Z\to X$ is the closed immersion. This relates the $V$-filtration to the restriction functors; see the nice survey by Schnell~\cite{Schnell-MHM} . 

When $r\geq 2$, $\gr_V^\alpha\cM$ need not even be coherent as a $\cD_Z$-module in general (see Section ~\ref{subsec:nct}). This is a major difference of the theory in higher codimension. In this paper, we generalize the above properties concerning the $V$-filtration along hypersurfaces on mixed Hodge modules to higher codimension. The statement is formulated using certain Koszul-type complexes. For any \emph{left} filtered regular holonomic and quasi-unipotent $\cD_X$-module $\cM$ and rational number $\alpha$, define filtered complexes
\[
A^{\alpha}(\cM) =\left\{ (V^\alpha \cM,F[-r]) \xrightarrow[]{t} \bigoplus_{i=1}^r (V^{\alpha+1}\cM,F[-r])\xrightarrow[]{t} \cdots \xrightarrow[]{t} (V^{\alpha+r}\cM,F[-r])\right\}
\]
placed in degrees $0,1,\dots,r$,
\[
B^\alpha(\cM) =\left\{ (\gr^\alpha_V \cM,F[-r]) \xrightarrow[]{t} \bigoplus_{i=1}^r (\gr^{\alpha +1}_V\cM,F[-r])\xrightarrow[]{t} \cdots \xrightarrow[]{t} (\gr^{\alpha+r}_V\cM,F[-r])\right\}
\] 
as the quotient $A^{\alpha}(\cM)/A^{>\alpha}(\cM)$ and 
\[
C^\alpha(\cM)=\left\{ (\gr^{\alpha+r}_V \cM,F) \xrightarrow[]{\de_t} \bigoplus_{i=1}^r (\gr^{\alpha+r-1}_V \cM,F[-1]) \xrightarrow[]{\de_t} \cdots \xrightarrow[]{\de_t} (\gr^{\alpha}_V\cM,F[-r])\right\}
\] 
in degrees $-r,-r+1,\dots,0$, where $V^\bullet\cM$ is the $V$-filtration along $Z$, and $F[i]_k=F_{k-i}$. 

The first main result of this paper is a generalization to higher codimension of Saito's condition of a $\cD$-module being ``quasi-unipotent and regular" along a hypersurface.

\begin{thm}\label{thm:acyclic}
If the filtered $\cD_X$-module $(\cM,F)$ underlies a mixed Hodge module, then the Koszul-like complexes $A^{\chi}(\cM)$ (resp. $C^\chi(\cM)$) are filtered acyclic for $\chi >0$ (resp. $\chi<0$).
\end{thm}

Our next theorem is a generalization of the distinguished triangles~\eqref{eq:dis}. To simplify the notation, denote $B(\cM):=B^0(\cM)$ and $C(\cM):=C^0(\cM)$. We give a comparison between $B(\cM)$ (resp. $C(\cM)$) and $i^!\cM$ (resp. $i^*\cM$) if $\cM$ underlies a mixed Hodge module. Here $i: Z \hookrightarrow X$ is the closed embedding.

\begin{thm}\label{thm:main} 
    Let $M=(\cM,F,L,\cK)$ be a mixed Hodge module where $F$ is the Hodge filtration, $L$ is the weight filtration and $\cK$ is the $\Q$-structure of the $\cD_X$-module $\cM$ i.e. $\DR_X\cM\simeq \cK\otimes_\Q \C$. Let $\theta=\sum^r_{i=1}t_i\partial_i$ be the Euler vector field. Then we have:
    \begin{enumerate}
        \item\label{item:a} the complexes $B(\cM)$ and $C(\cM)$ together with the filtrations $W$ induced by the relative monodromy filtration $W=W(\theta-\alpha+r,\gr_V^\alpha L_\bullet\cM)$ on $\gr_V^\alpha\cM$ are mixed Hodge complexes on $Z$;
        \item \label{item:b} the complex $B(\cM)$ (resp. $C(\cM)$) is isomorphic to $(i^!\cM,F)$ (resp. $(i^*\cM,F)$) in the derived category of filtered $\cD$-modules with $\Q$-structure; 
        \item  \label{item:c} moreover, the isomorphisms in $(b)$ are compatible with weight filtration and induce isomorphism on the cohomologies:
        \[
        \gr^W_k\cH^\ell B(\cM)\simeq \gr^W_{k+\ell}\cH^\ell i^!\cM \quad \text{and} \quad \gr^W_{k}\cH^{-\ell} C(\cM)\simeq \gr^W_{k-\ell}\cH^{-\ell} i^*\cM
        \]
        as polarizable Hodge modules for $\ell\geq 0$.
    \end{enumerate}
\end{thm}

Part of the proof of (b) is to show that $B(\cM)$ and $C(\cM)$ have $\Q$-structure, and then to relate the $\Q$-structure on these complexes with those defined by Saito's theory on $i^!\cM$ and $i^*\cM$, respectively.

See Section~\ref{sec:mhc} and and ~\cite{MHC} for the definition of mixed Hodge complexes. The existence of the relative monodromy filtration $W=W(\theta-\alpha+r,\gr_V^\alpha L_\bullet\cM)$ on $\gr_V^\alpha\cM$ can be achieved easily using the deformation to the normal cone (see Lemma~\ref{lem:relmonext}). 

Forgetting the weight filtration, the two distinguished triangles~\eqref{eq:dis} are recovered by Theorem~\ref{thm:main}(a). The reason why we do not get the distinguished triangles in the derived category of mixed Hodge modules is that we directly use the monodromy filtrations relative to $\gr^0_V(L\cM)$ on $\gr^0_V\cM$ without the shift in Saito's definition of vanishing cycles (see Section~\ref{subsec:admiss}). For further motivation and discussion, see Section~\ref{subsec:nct}.

Theorem~\ref{thm:main} simplifies the calculation of the restriction functors. For example, the usual approach to computation of $i^!$ uses either the Koszul complex induced by the iterated $V$-filtrations along hypersurfaces or the Koszul complex induced by the localization along hypersurfaces. Theorem~\ref{thm:main} says that we can bypass the localization or the iterated $V$-filtration by a one-step calculation on the $V$-filtration along $Z$. 

As a very special case of Theorem~\ref{thm:acyclic}, we give a Hodge-theoretic proof of Skoda's famous theorem concerning multiplier ideals. For the definition of multiplier ideals and their properties (as well as a proof of Skoda's theorem), see ~[Ch. 9]\cite{Laz}. See \cite{BMS} for the relation between the multiplier ideals and the $V$-filtration.

\begin{cor}[Skoda]
Let $\mathfrak a$ be a coherent ideal of $\cO_X$ generated by $r$ elements and $\cJ(X,\mathfrak a^c)$ be the multiplier ideal of exponent $c$. Then we have 
\[
\cJ(X,\mathfrak a^{c})=\mathfrak a \cJ(X,\mathfrak a^{c-1})
\]
for any $c\geq r$.
\end{cor}

One of our main tools in this paper is the process of ``specialization", which is described in ~\cite[Section 2.30]{Saito-MHM} and also used in ~\cite{BMS}. From any mixed Hodge module $M$ on $X$, we obtain a \emph{monodromic} mixed Hodge module $\Sp(M)$ on $T_Z X$, the normal bundle of $Z$ inside $X$.

Our main application of Theorem \ref{thm:main} is to the Fourier transform of monodromic mixed Hodge modules. To define a monodromic $\cD$-module, let $E \to X$ be a vector bundle of rank $r$. A $\cD_{E}$-module $\cM$ is \emph{monodromic} if, for any local trivialization $E \cong X \times \A^r$, $\cM$ decomposes into generalized eigenspaces
\begin{equation} \label{deco1} \cM = \bigoplus_{\chi \in \C} \cM^\chi,\end{equation}
where $\theta - \chi +r = \sum_{i=1}^r \de_{z_i}z_i - \chi$ is nilpotent on $\cM^\chi$. We say that a mixed Hodge module $M$ on $E$ is \emph{monodromic} if the underlying $\cD$-module $\cM$ is a monodromic $\cD_E$-module. We denote the abelian category of monodromic mixed Hodge modules by $\MHM_{mon}(E)$. If $\cM$ underlies a mixed Hodge module, then in fact the only non-zero summands in the decomposition \ref{deco1} are for $\chi \in \Q$.

Let $E^\vee$ be the dual bundle. The Fourier-Laplace transform of any $\cD_{E}$-module is a $\cD_{E^\vee}$-module, denoted $\mathsf{FL}(\cM)$ (or, if we want to stress that we are on $E$, it is denoted $\FL_E(\cM)$), where, if $E \cong X\times \A^r$ and $E^\vee \cong X\times \A^r$ are local trivializations, with coordinates $z_1,\dots, z_r$ and $w_1,\dots, w_r$ on the respective $\A^r$ factors, then $\mathsf{FL}(\cM)$ has the same underlying $\cD_X$-module structure as $\cM$, but
\[ w_i m = -\de_{z_i} m, \quad \de_{w_i} m = z_i m.\]

This functor preserves the property of being monodromic, however, it does not have a lift to the category of mixed Hodge modules. In fact, even if $\cM$ has regular singularities, it is possible for $\mathsf{FL}(\cM)$ to have irregular singularities. Thankfully, Brylinski \cite{Bry} showed that if $\cM$ is \emph{monodromic} with regular singularities, then $\mathsf{FL}(\cM)$ also has regular singularities.

In fact, one can express the Fourier transform for monodromic modules as a composition of functors coming from geometry (called the \emph{Fourier-Sato transform} or \emph{Monodromic Fourier transform}). Let $\cE = E\times_X E^\vee$, with the regular function $g: \cE \to \A^1$ given by the natural evaluation map. Locally, with coordinates $z_1,\dots, z_r$ and $w_1,\dots, w_r$ as above, $g = \sum_{i=1}^r z_iw_i$. Consider the projection $p: \cE \to E$ and the zero-section $\sigma: E^\vee \to \cE$. 

Let $\psi_g = \psi_{g,1} \oplus \psi_{g,\neq 1}: \MHM(\cE\times \A^1) \to \MHM(\cE)$ be the total nearby cycles functor, and $\phi_g = \phi_{g,1} \oplus \psi_{g,\neq 1}: \MHM(\cE\times \A^1) \to \MHM(\cE)$ be the total vanishing cycles functor. For their definition, see Section \ref{subsec:conv}.

Our main theorem concerning the Fourier transform is the following:

\begin{restatable}{thm}{fourier} \label{thm:fourier} Let $\cM$ be the filtered $\cD_E$-module underlying a monodromic mixed Hodge module $M$ on $E$. Then the $\cD_{E^\vee}$-module underlying the mixed Hodge module
\[ \cH^0 \sigma^* \phi_g (p^!(M)[-r]) \in \MHM(E^\vee)\]
is isomorphic to $\mathsf{FL}(\cM)$. We denote this composition of functors by $\FL(M)$.
 
Moreover, for $\lambda \in [0,1)$ and $\ell \in \Z$, we have
\[ F_p \mathsf{FL}(\cM)^{r-(\lambda+\ell)} = F_{p-\ell - \lceil \lambda \rceil} \cM^{\lambda+\ell}.\]

Finally, the weight filtration is given by
\[ W_k \FL(\cM)^{\lambda + \Z} = \FL(W_{k+r+\lceil \lambda \rceil} \cM)^{\lambda+ \Z}.\]

Here $\FL(\cM)^{\lambda +\Z} = \bigoplus_{\ell \in \Z} \FL(\cM)^{\lambda + \ell}$ for any $\lambda \in [0,1)$.
\end{restatable}
The reason why we need Theorem \ref{thm:main} in the proof of this theorem is that the last functor $\cH^0 \sigma^*$ is most easily understood in the context of that theorem. This allows us to understand both $F_\bullet$ and $W_\bullet$ on $\FL(\cM)$. The last statement was already shown for a special class of monodromic mixed Hodge modules in ~\cite[Prop. 4.12]{GKZ}, where the authors of loc. cit. use the definition of Fourier-Sato transform as in \cite{KS}, see Remark \ref{rmk:comp} below for a comparison.

In \cite{GKZ}, they define the monodromic Fourier transform as a functor from $D^b_{mon}(\MHM(E))$ to $D^b_{mon}(\MHM(E^\vee))$, where $D^b_{mon}(-)$ is the full subcategory of $D^b(-)$ of objects with monodromic cohomology modules. Then, using the fact that this computes $\FL$ (which is an exact functor), they conclude that the functor descends to one $\MHM_{mon}(E) \to \MHM_{mon}(E^\vee)$. In the same way, we can extend the functor $\FL$ which we have defined to the entire mondromic derived category, giving
\[ \FL = \sigma^* \phi_g p^![-r] : D^b_{mon}(\MHM(E)) \to D^b_{mon}(\MHM(E^\vee)).\]

Not only does Theorem \ref{thm:fourier} follow from Theorem \ref{thm:main}, but it also recasts Theorem \ref{thm:acyclic} in a more symmetric way. Indeed, it is not hard to see that there are equalities $B^\chi(\cM) = B^\chi(\Sp(\cM)), C^\chi(\cM) = C^\chi(\Sp(\cM))$. Similarly, it is not hard to see that (at least, ignoring the filtration) $C^\chi(\FL(\Sp(\cM)) = B^{-\chi}(\Sp(\cM))$ and $B^\chi(\FL(\Sp(\cM)) = C^{-\chi}(\Sp(\cM))$. Using the theorem, one can check that, actually, these equalities do hold at the filtered level. Hence, filtered acyclicity of the $B^\bullet$ complex follows from filtered acyclicity of the $C^\bullet$ complex, and conversely.

Note that the Hodge filtration is exactly the same as that which is obtained from ~\cite[Prop. 3.25]{monoMHM}. As remarked in ~\cite[Rem. 3.24]{monoMHM}, the Fourier-Laplace transform can be endowed with many different mixed Hodge module structures. So the utility of this theorem is to make explicit the Hodge and weight filtrations for \emph{some} mixed Hodge module structure on $\FL(\cM)$.

In order to get the information about the Hodge filtration for $\mathsf{FL}(\cM)$, it is important to know that the Hodge filtration for monodromic mixed Hodge modules decomposes along the eigenspace decomposition. For the case $r=1$, this was shown in ~\cite[Thm 2.2]{monoMHM}. We use this as a base case for induction to show
\begin{thm} Let $(\cM,F,W)$ be the bifiltered $\cD$-module underlying a monodromic mixed Hodge module on $E$. Then
\[ F_\bullet \cM = \bigoplus_{\chi \in \Q} F_\bullet \cM^\chi,\]
where $F_\bullet \cM^\chi = \cM^\chi \cap F_\bullet \cM$.

Moreover, the weight filtration $W_\bullet \cM$ is its own relative monodromy filtration along the nilpotent operator $N = \bigoplus_{\chi \in \Q} (\theta - \chi+r)$, where $\theta = \sum_{i=1}^r z_i \de_{z_i}$. In particular, if $\cM$ is pure, then the monodromy filtration is trivial, so for all $\chi \in \Q$, we have
\[ \cM^\chi = \ker(\theta - \chi +r).\]
\end{thm}

Let $\overline{\FL}: \MHM(E) \to \MHM(E^\vee)$ be the \emph{inverse Fourier transform}. Up to a Tate twist, it is simply $\FL$. However, due to conventions in the definition of the filtrations for mixed Hodge modules, we have to Tate twist the part which is not unipotent:
\[ \overline{\FL}(M) : = \cH^0\sigma^* \phi_{g,1} p^!(M(-r))[-r] \oplus a^* \cH^0 \sigma^* \psi_{g,\neq 1} p^!(M(1-r))[-r].\]

Here $(\ell)$ denotes a Tate twist by $\ell$, $\psi_{g,\neq 1} = \phi_{g,\neq 1} = \bigoplus_{\lambda \in (0,1)} \psi_{g,\lambda}$, and $a: E \to E$ is the antipodal map.

Concerning this transformation, we have the following \emph{Fourier inversion formula}
\begin{cor} For any monodromic mixed Hodge module $M$ on $E$, we have an isomorphism
\[a^*\overline{\FL_{E^\vee}}(\FL_E(M)) \cong M.\]
\end{cor}

The presence of the antipodal map is forced by the $\cD$-module structure, but does not affect the Hodge or weight filtrations.

Finally, this inverse Fourier transform also comes up when we try to understand how duality behaves with respect to $\FL$. Specifically, we have the following:
\begin{thm} \label{FLDuality} We have an isomorphism
\[ \bD_{E^\vee}\circ \FL_E \cong \overline{\FL}_{E^\vee} \circ \bD_{E}: \MHM_{mon}(E) \to \MHM_{mon}(E^\vee).\]
\end{thm}

In this way, we could have defined $\overline{\FL} = \bD_{E^\vee} \circ \FL \circ \bD_E$. Note that the Tate twists which occur on $\phi_{g,\neq 1}$ but do not happen for $\phi_{g,1}$ are completely analogous to the behavior of duality with respect to nearby and vanishing cycles, as explained in ~\cite[Prop. 2.6]{Saito-MHM} and \cite{Saito-Duality}.

We will also use our results in future work with Mircea \Mustata and Sebasti\'{a}n Olano to study local cohomology.

\subsection{Strategy of the proof}
To prove Theorem~\ref{thm:acyclic}, we first treat the case when $(\cM,F_\bullet)$ underlies a polarizable pure Hodge module. Using the fact that pure Hodge modules decompose by strict support, we have to consider two situations: 
\begin{enumerate}[label=(\alph*)]
    \item  the support of $\cM$ is contained in $Z$; 
    \item  there is no $\cD_X$-submodule of $\cM$ whose support is contained in $Z$.
\end{enumerate} 
The case (a) directly follows from the definition. For case (b), we pass to the blow-up and reduce the problem to the codimension one case. Let $\pi:\widehat X\rightarrow X$ be the blow-up of $Z$ and $E$ be the exceptional divisor. Let $(\widehat\cM,F_\bullet\widehat\cM)$ be the minimal extension of $(\cM,F_\bullet\cM)|_{X\setminus Z}$ along $E$, which also underlies a pure Hodge module by the structure theorem of Hodge modules~\cite{Schnell-MHM}. By the direct image theorem of Hodge modules, $(\cM,F_\bullet\cM)$ is a direct summand of $\pi_+(\widehat\cM, F_\bullet\widehat\cM)$. Therefore, it suffices to prove the statement for $\pi_+(\widehat\cM,F_\bullet\widehat\cM)$. Then we factor $\pi:\widehat X\rightarrow X$ into the graph embedding $i_\pi: \widehat X\rightarrow \widehat X\times X$ and the second projection $ p: \widehat X\times X\rightarrow X$ and study the direct images of $(\widehat\cM,F_\bullet\cM)$ under these two morphisms. The graph embedding case has no homological algebra involved and in the case of the projection, we use the bistrictness proved by Budur, Musta\c{t}\u{a} and Saito \cite{BMS} and Hard Lefschetz \cite[2.14]{Saito-MHM} on the direct images.

The strategy of proof for the pure case does not work for mixed Hodge modules because there is no reason that $(\cM,F_\bullet\cM)$ is a direct summand of $\pi_+(\widehat \cM,F_\bullet\widehat \cM)$. Instead, we use deformation to the normal cone to get the compatibility among the Hodge filtration, $V$-filtration and weight filtration. Using the compatibility we reduce the proof to the pure case. 

We also give an alternative proof of Theorem \ref{thm:acyclic}, using the direct image theorem for \emph{mixed} Hodge modules, as well as an explicit computation using a \v{C}ech complex.

As for the proof of Theorem~\ref{thm:main}, we first deal with the case when $(\cM,F)$ underlies a polarizable Hodge module as we did in the proof of the pure case for Theorem~\ref{thm:acyclic}. In this case we heavily use the semisimplicity of polarizable pure Hodge modules. To do the mixed case we need a theorem of Deligne, which roughly states that there exists a unique functorial splitting of the associated graded of the relative monodromy filtration. The proof reduces to the pure case by Deligne's Theorem. 

Finally, for the proof of the results concerning the Fourier transform, the main difficulty lies in computing the $V$-filtration along the graph embedding of the function $g = \sum_{i=1}^r z_i w_i$. As an example, if $\cM = \cO_X[z]$ is the structure sheaf on $E$, then $\cO_X[z,w]$ is the structure sheaf on $\cE$, and the function $g$ is (quasi)-homogeneous. Hence, the computation of the $V$-filtration for such a module is given in ~\cite[Formula 4.2.1]{Saito-HFilt} and ~\cite[Lemma 3.3]{Mingyi}.

\subsection{Outline} We first review some basic facts about $V$-filtration and mixed Hodge modules in Section~\ref{sec:pre}. Some topological properties of $V$-filtrations along subvarieties are derived in Section~\ref{sec:topo}. We give two different proofs of Theorem~\ref{thm:acyclic} in Section~\ref{sec:pacyclic}. Some applications of Theorem~\ref{thm:acyclic} are derived in Section~\ref{sec:applic}. Theorem~\ref{thm:main} is proved in Section~\ref{sec:pmhc}. We also point out a proof of Theorem~\ref{thm:main}$(b)$ which does not rely on Theorem~\ref{thm:main}$(a)$; see Remark~\ref{rmk:bypass}. Finally we study the Fourier transform of monodromic mixed Hodge modules in Section~\ref{sec:fourier}. 

\subsection{Acknowledgements} The authors are extremely grateful to their respective advisers, Christian Schnell and Mircea \Mustata. Without them this project would not have been possible. We would also like to thank Mihnea Popa for several suggestions and questions. QC thanks Guodu Chen and Mads Villadsen for reading a draft of this paper. BD is thankful to Sebasti\'{a}n Olano, James Hotchkiss and Jack Carlisle for many useful conversations. We would like to thank the anonymous referee for various suggestions and corrections.

\section{Preliminaries} \label{sec:pre}

\subsection{Convention and Notation}
Let $X$ be a smooth complex algebraic variety. We recall that there is an equivalence of categories between filtered left and right $\cD_X$-modules. Given a filtered left $\cD_X$-module $(\cM,F)$, we denote by $(\cM^r,F)$ the corresponding filtered right $\cD_X$-module. In fact, $\cM^r=\omega_X\otimes_{\cO_X} \cM$, while the filtration on $\cM^r$ is given by 
\[
F_{p-n}\cM^r=\omega_X\otimes_{\cO_X} F_p\cM \quad \text{for all} \quad p\in\Z,
\]
where $n=\dim X$.

For right $\cD_X$-modules it is customary to use the increasing $V$-filtration. This is related to the $V$-filtration on the corresponding left $\cD_X$-module by 
\[
V_\alpha\cM^r=\omega_X\otimes_{X} V^{-\alpha}\cM.
\]

To match this convention with the lower indices, for right $\cD$-modules, we will denote the complexes $A,B$ and $C$ from Theorem \ref{thm:acyclic} with lower indices, as
\[
A_{\alpha}(\cM) =\left\{ (V_\alpha \cM,F) \xrightarrow[]{t} \bigoplus_{i=1}^r (V_{\alpha-1}\cM,F)\xrightarrow[]{t} \cdots \xrightarrow[]{t} (V_{\alpha-r}\cM,F)\right\}
\]
placed in degrees $0,1,\dots,r$,
\[
B_{\alpha}(\cM) =\left\{ (\gr_\alpha^V \cM,F) \xrightarrow[]{t} \bigoplus_{i=1}^r (\gr_{\alpha -1}^V\cM,F)\xrightarrow[]{t} \cdots \xrightarrow[]{t} (\gr_{\alpha-r}^V\cM,F)\right\}
\] 
as the quotient $A_{\alpha}/A_{>\alpha}$ and 
\[
C_{\alpha}(\cM)=\left\{ (\gr_{\alpha-r}^V \cM,F[r]) \xrightarrow[]{\de_t} \bigoplus_{i=1}^r (\gr_{\alpha-r+1}^V \cM,F[r-1]) \xrightarrow[]{\de_t} \cdots \xrightarrow[]{\de_t} (\gr_{\alpha}^V\cM,F)\right\}
\] 
in degrees $-r,-r+1,\dots,0$. 

Moreover, for $Z\subseteq X$ a smooth subvariety of the smooth variety $X$, we denote by $T_Z X = \mathcal{S}pec_X(\bigoplus_{k \geq 0} \cI^k/\cI^{k+1}) \to Z$ the normal bundle of $Z$ inside $X$.

\subsection{Kashiwara-Malgrange $V$-filtrations}\label{subsec:VFilt}
We begin with a review of the theory of $V$-filtrations introduced by Kashiwara and Malgrange. For more details, see ~\cite[Section 3.1]{Saito-MHP} and~\cite[Section 9]{Schnell-MHM} for the case of a hypersurface and ~\cite[Section 1.1]{BMS} for the case of higher codimension.

Let $(t_1,\dots, t_r): X \to \A^r$ be a smooth regular function, with fiber $Z$ over the origin.  We define a $\Z$-indexed filtration on $\cD_X$ by
\[V^k\cD_X = \{ P \in \cD_X \mid P \cdot \cI_Z^j \subseteq \cI_Z^{j+k} \text { for all }j\}.\]

A $\Q$-indexed filtration $V^\bullet \cM$ is \emph{discrete and left-continuous} if $\bigcap_{\alpha < \beta} V^\alpha = V^\beta$ for all $\beta \in \Q$, and if there exists some $\ell \in \Z_{>0}$ such that the subspace $V^\alpha$ is constant for all $\alpha \in (\frac{m}{\ell}, \frac{m+1}{\ell}]$, for any $m \in \Z$.

Given a coherent left $\cD_X$-module $\cM$, a Kashiwara-Malgrange $V$-filtration on $\cM$ along $Z$ (see \cite{Kas83}, \cite{Mal83}) is an exhaustive, decreasing $\Q$-indexed filtration which is discrete and left-continuous such that, if $\theta : = \sum_{i=1}^r t_i \de_{t_i}$ is any locally defined Euler vector field along $Z$, the filtration must satisfy:
\begin{enumerate} \item $V^k\cD_X V^\chi \cM \subseteq V^{\chi+k}\cM$ for all $k\in \Z,\chi \in \Q$,
\item $V^k\cD_X V^\chi\cM = V^{\chi+k}\cM$ for all $k \in \Z_{\geq 0}$, $\chi \gg 0$,
\item Each $V^\chi \cM$ is coherent over $V^0\cD_X$,
\item The operator $\theta - \chi +r$ is nilpotent on $\gr_V^\chi \cM = V^\chi\cM/V^{>\chi} \cM$.
\end{enumerate} 

It is an easy exercise to see that there can be at most one $V$-filtration on any coherent $\cD_X$-module $\cM$. We say that a module $\cM$ which has a $\Q$-indexed $V$-filtration is $\Q$-specializable. Any morphism between $\Q$-specializable modules is strict with respect to the $V$-filtration. Moreover, if
\[0  \to \cM' \to \cM \to \cM'' \to 0\]
is a short exact sequence of $\cD_X$-modules, and $\cM$ has a $V$-filtration, then the induced filtrations on $\cM'$ and $\cM''$ satisfy the properties of the $V$-filtration.

\begin{egs} \label{VFiltEg} $(a)$ Let $\cE$ be an $\cO_X$-coherent $\cD_X$-module. Then $V^{k}\cE : = \cI_Z^{k-r} \cdot \cE$ satisfies the properties of the $V$-filtration. For example,
\[ \theta t^\alpha m = t^\alpha (|\alpha| + \theta) m,\]
$(b)$ \label{closedHodge} (Kashiwara's equivalence) Assume $\cM$ is supported on $Z$, so by Kashiwara's equivalence (see ~\cite[Section 1.6]{HTT}), there exists a coherent $\cD_Z$-module $\cN$ such that $\cM = \sum_{\alpha \in \N^r} \cN \de_t^\alpha$. Then
\[ V^{-k}\cM = \sum_{|\alpha|\leq k} \cN \de_t^\alpha.\]

For us, it will also be important to understand the case when $(\cM,F) \cong i_+(\cN,F)$ as a filtered $\cD$-module. For left $\cD$-modules, the pushforward of a filtered module has filtration defined as
\[ F_p i_+(\cN,F) = \sum_{\alpha \in \N^r} F_{p-|\alpha|-r}\cN \de_t^\alpha.\]

From this, we see easily that
\[ F_p V^{-k}i_+(\cN,F) = \sum_{|\alpha| \leq k} F_{p-|\alpha|-r}\cN \de_t^\alpha.\]
\end{egs}

This last example leads to an important property of the $V$-filtration.

\begin{lem} Assume $\varphi: \cN \to \cM$ is a morphism between two specializable modules, such that $\varphi|_U: \cN|_U \to \cM|_U$ is an isomorphism, where $U = X - Z$. Then $\varphi: V^{>0}\cN \to V^{>0}\cM$ is an isomorphism.
\end{lem}
\begin{proof} Let $K = \ker(\varphi), C = \text{coker}(\varphi)$. The assumption implies these are supported on $Z$, so by the previous example, $V^{>0} K =0$ and $V^{>0}C =0$. Hence, taking $V^{>0}$ of the long exact sequence
\[ 0 \to K \to \cN \to \cM \to C \to 0,\]
we get
\[ 0 = V^{>0}K \to V^{>0}\cN \to V^{>0}\cM \to V^{>0}C = 0,\]
proving the claim.
\end{proof}

One can also speak of a $\C$-indexed Kashiwara-Malgrange filtration. The following theorem of Kashiwara implies that the $\cD_X$-modules we care about always carry such a filtration. It is built into the theory of mixed Hodge modules that the filtered $\cD_X$-modules underlying them must actually be $\Q$-specializable, not just $\C$-specializable.

\begin{thm} \cite{Kas83} Any regular holonomic $\cD_X$-module $\cM$ is $\C$-specializable.
\end{thm}

\subsection{Normal crossing type}\label{subsec:nct}
For the codimension one case, it is essentially immediate from the definition that the maps $t: V^\alpha \cM \to V^{\alpha+1}\cM$ (resp. $\de_t: \gr_V^{\alpha+1} \cM \to \gr_V^{\alpha}\cM$) are isomorphisms for all $\alpha \neq 0$. The following example shows that, for codimension larger than one, the correct generalization of this property should concern Koszul-like complexes in the $t_1,\dots, t_r$ (resp. $\de_{t_1},\dots, \de_{t_r}$).

Let $\cM$ be an algebraic regular holonomic \emph{left} $D_2$-module of normal crossing type along the two axes on $\A^2$, where $D_2$ is the Weyl algebra over $\A^2$. For details on normal crossing type modules, see ~\cite[Section 3]{Saito-MHM}. Let $(x,y)$ be the coordinate system on $\A^2$. Define $\cM^{\alpha,\beta}=\ker (\partial_x x-\alpha)^\infty \cap \ker (\partial_y y-\beta)^\infty$ for $(\alpha,\beta)\in \Q^2$. Because of the assumption that $\cM$ is of normal crossing type, we have the identity
\[
\bigoplus_{\alpha,\beta\in \Q^2} \cM^{\alpha,\beta} = \cM
\] 
and each $\cM^{\alpha,\beta}$ is a finite dimensional vector space over $\C$. Then one can easily check the $V$-filtration along the origin is given by 
\[
V^k \cM = \bigoplus_{\alpha+\beta \geq k} \cM^{\alpha,\beta},
\]
and $\gr_{V_x}^\alpha\gr_{V_y}^\beta \cM= \cM^{\alpha,\beta}$ where $V_x \cM$ is the $V$-filtration along $\{x=0\}$ and $V_y \cM$ is the $V$-filtration along $\{y=0\}$. Then the double complex 
\begin{equation}\label{eq:nceg}
\begin{tikzcd}
    \gr_V^k \cM \arrow{r}{x} \arrow{d}{y} & \gr_V^{k+1} \cM \arrow{d}{y} \\
    \gr_V^{k+1} \cM \arrow{r}{x} & \gr_V^{k+2} \cM
\end{tikzcd}
=\bigoplus_{\alpha+\beta = k} \left(
\begin{tikzcd}
    \cM^{\alpha,\beta} \arrow{r}{x} \arrow{d}{y} &  \cM^{\alpha+1,\beta} \arrow{d}{y} \\
    \cM^{\alpha,\beta+1} \arrow{r}{x} &  \cM^{\alpha+1,\beta+1}
\end{tikzcd}
\right)
\end{equation}
is exact if $k\neq 0$ because one of $x$ and $y$ must be bijective in a summand by the properties of $V$-filtration in codimension one. If $k=0$, the above double complex is quasi-isomorphic to the total complex of
\[
\begin{tikzcd}
 \cM^{0,0} \arrow{r}{x} \arrow{d}{y} &  \cM^{1,0} \arrow{d}{y} \\
 \cM^{0,1} \arrow{r}{x} &  \cM^{1,1}
\end{tikzcd}
\]
which is isomorphic to $i^!_Z\cM$. Since the total complex of the double complex is just the Koszul complex 
\[
\begin{tikzcd}
\gr_V^k\cM \arrow{r}{(x,y)}  &  \left(\gr_V^{k+1} \cM \right)^2 \arrow{r}{ \binom{y}{-x}} & \gr_V^{k+2} \cM,
\end{tikzcd}
\]
we proved a version of generalization of the properties of $V$-filtration in codimension one that the above Koszul complex is isomorphic to $i^!_Z\cM$ when $k=0$ and is exact when $k\neq 0$. The similar statement regarding the complex
\[
\begin{tikzcd}
\gr_V^{k+2}\cM \arrow{r}{(\partial_x,\partial_y)}  &  \left(\gr_V^{k+1} \cM \right)^2 \arrow{r}{ \binom{\partial_y}{-\partial_x}} & \gr^V_{k} \cM
\end{tikzcd}
\]
is left to the readers.

If $(\cM,L)$ underlies a mixed Hodge module of normal crossing type where $L$ is the weight filtration then $\cM^{\alpha,\beta}$ carries a relative mondromy filtration $W=W(\partial_x x+\partial_y y-\alpha-\beta,L\cM^{\alpha,\beta})$. In fact, we have the relation $W=W(\partial_x x-\alpha, W(\partial_y y-\beta,L))$ by~\cite[3]{Saito-MHM} since we assume $\cM$ is of normal crossing type. It follows that, if $k=0$, the result of applying $\gr^W$ to the complex ~\eqref{eq:nceg} is quasi-isomorphic to 
\[
\begin{tikzcd}
 \gr^W \cM^{0,0} \arrow{r}{x} \arrow{d}{y} & \gr^W \cM^{1,0} \arrow{d}{y} \\
 \gr^W \cM^{0,1} \arrow{r}{x} & \gr^W \cM^{1,1}
\end{tikzcd}
\]
but the upper-horizontal and left-vertical morphisms are zero by~\cite[1]{Saito-MHM}. This is the motivation for using mixed Hodge complexes in Theorem~\ref{thm:main}.

\subsection{Deformation to the Normal Cone} \label{subsec:dfn}
This subsection is devoted to studying the specialization construction, which goes through the deformation to the normal cone. See for example, Section 2.30 of \cite{Saito-MHM} and Section 1.3 of \cite{BMS}.

Let $Z\subseteq X$ be defined by the ideal sheaf $\cI_Z \subseteq \cO_X$, and consider the variety
\[ 
\widetilde{X} : = \cS pec_X \left(\bigoplus_{\ell \in \Z} \cI^{\ell}_Z \otimes u^\ell \right),
\]
along with the smooth morphism $u: \widetilde{X} \to \A^1 = \text{Spec}(\C[u])$. The fiber $u^{-1}(0)$ is isomorphic to $T_ZX$, the normal cone of $Z$ in $X$, and so we call this a \emph{deformation to the normal cone}. Over the open subset $\mathbf{G}_m:=\A^1 -\{0\}$, the map is isomorphic to the smooth projection $X \times \mathbf{G}_m \to \mathbf{G}_m $. We will also consider the smooth morphism $p: X \times \mathbf{G}_m  \to X$ of relative dimension 1.  Let $j: X \times \mathbf{G}_m  \hookrightarrow \widetilde{X}$ be the open immersion. It is the complement of the smooth divisor $T_Z X = u^{-1}(0)$.
\[
    \begin{tikzcd}
    X\times \mathbf{G}_m =\tilde X^* \arrow{r}{j} \arrow{d} & \tilde X  \arrow{d} & T_ZX \arrow{l}\arrow{d}  & \tilde X^* \arrow{r}{j}\arrow{dr}{p} & \tilde X \arrow{d}{\rho} & T_ZX \arrow{l}\arrow{d} \\
    \mathbf{G}_m  \arrow{r} &  \A^1 & \{0\} \arrow{l} & & X & Z\arrow{l}
    \end{tikzcd}
\]
For any $M \in \MHM(X)$, define $\Sp(M) : = \psi_u j_* (p^*(M)[-1]) \in \MHM(T_Z X)$. Here the shift by $[-1]$ comes from the relative dimension of the morphism $p$. Let $\widetilde{M} = j_*(p^*(M)[-1])$. As explained in \cite[Formula (1.3.1)]{BMS}, the underlying $\cD$-module of $\widetilde{\cM}$ satisfies
\[ 
\rho_*(\widetilde{\cM}) = \bigoplus_{\ell \in \Z} \cM u^\ell, \quad \rho_*(V^{\lambda}\widetilde{\cM}) = \bigoplus_{\ell \in \Z} V^{r + \lambda - \ell -1}\cM u^\ell,
\]
as $\cO_X$-modules, where $V^\bullet \widetilde{\cM}$ is the $V$-filtration along the smooth hypersurface defined by $u$ and $V^\bullet \cM$ is the $V$-filtration along $Z$. In particular, the $\cD$-module underlying $\Sp(M)$ is
\[ \Sp(\cM) = \bigoplus_{\lambda \in (0,1]} \bigoplus_{\ell \in \Z} \gr_V^{r+\lambda - \ell -1}\cM u^\ell.\]

We are able to identify the Hodge filtration on $\Sp(\cM)$ in terms of that on $\cM$. To do this, we first compute the Hodge filtration on $\widetilde{\cM}$.

\begin{lem} \label{SpHFilt} Let $(\cM,F_\bullet)$ be the filtered $\cD_X$-module underlying a mixed Hodge module $M$ on $X$. Then
\[ F_p V^\gamma \widetilde{\cM} = \bigoplus_{\ell \in \Z}  V^{\gamma - \ell + r -1}\cM \cap \left(\sum_{q=0}^{\lfloor -\gamma \rfloor} (\theta + \ell+1)\cdots (\theta+\ell+q)F_{p+1-q}V^{r-1-q-\ell}\cM\right)u^\ell,\]
with the understanding that the sum is simply equal to the term with $q=0$ if $\lfloor -\gamma \rfloor \leq 0$.

In particular, if $\gamma \geq 0$, 
\[ 
F_p V^\gamma \widetilde{\cM} = \bigoplus_{\ell \in \Z} F_{p+1} V^{\gamma - \ell + r -1}\cM u^\ell
.\]
\end{lem}
\begin{proof} To begin, because $j$ is the inclusion of the complement of a Cartier divisor, we have the following formula for the Hodge filtration (see Formula 3.2.3.2 of \cite{Saito-MHP} and Remark \ref{HodgeFiltComplement} below)
\[ F_p \widetilde{\cM} = \sum_{q\geq 0} \de_u^q\left(V^0\widetilde{\cM} \cap j_*\left(F_{p-q}p^*(\cM)\right)[-1]\right).\]
As $p: X \times \mathbf G_m \to X$ is smooth of relative dimension 1, we know $F_\bullet p^*(\cM)[-1] = \bigoplus_{\ell \in \Z} F_{\bullet+ 1}\cM u^\ell$.

Also, we have that $\de_u$ acts on a homogeneous element $m u^\ell$ by $(\theta + \ell)m u^{\ell -1}$, where $\theta = \sum_{i=1}^r t_i \de_{t_i}$. Thus, we get the following description (which was pointed out to the authors by \Mustata)
\[ F_p \widetilde{\cM} = \bigoplus_{\ell \in \Z}\left( \sum_{q\geq0 } (\theta + \ell +1)\cdots (\theta + \ell + q) F_{p+1-q} V^{r-1-q-\ell}\cM\right)u^\ell.\]

Let $mu^\ell \in V^\gamma F_p \widetilde{\cM}$, so we can write $m = \sum_{q=0}^N (\theta + \ell +1)\cdots (\theta + \ell + q) m_q$, where $m_q \in F_{p+1-q} V^{r-1-q-\ell}\cM$. Also, $m \in V^{\gamma - \ell + r -1}\cM$ by assumption.

Now, break up \[m = \sum_{q=0}^{\lfloor -\gamma \rfloor} (\theta + \ell +1)\cdots (\theta + \ell + q) m_q + \sum_{q = \lfloor-\gamma \rfloor +1}^N (\theta + \ell+ 1)\cdots (\theta+\ell+q) m_q.\]

For $q\leq \lfloor -\gamma \rfloor$, we have that $r-1-q-\ell \geq r-1+\gamma-\ell$, and so $V^{r-1-q-\ell}\cM \subseteq V^{r -1+\gamma -\ell}\cM$. As $m\in V^{r - 1 + \gamma -\ell}\cM$ by assumption, this implies that the sum
\[ 
\begin{aligned}
&\sum_{\lfloor -\gamma \rfloor+1}^N (\theta +\ell+1)\cdots (\theta+\ell+q)m_q =\\
& (\theta +\ell+1)\cdots(\theta+\ell+\lfloor -\gamma \rfloor +1) \sum_{q = \lfloor -\gamma \rfloor +1}^N (\theta+\ell +\lfloor -\gamma \rfloor +2)\cdots(\theta + \ell +q) m_q  \in V^{r -1+\gamma-\ell}\cM.
\end{aligned}
\]

Write $m' = \sum_{\lfloor -\gamma \rfloor +1}^N (\theta + \ell + \lfloor -\gamma \rfloor +2)\cdots (\theta + \ell +q) m_q$. As $m_q \in V^{r-1-q-\ell}\cM$, we have that $m' \in V^{r - 1-N-\ell}\cM$. Then $(\theta +\ell +N+1)^d m' \in  V^{> r - 1-N-\ell}\cM$ for some $d \gg 0$. As $N > \lfloor - \gamma \rfloor$ and $(\theta +\ell+1)\cdots (\theta + \ell + \lfloor -\gamma \rfloor +1)m' \in  V^{> r -1-N-\ell}\cM$, B\' ezout's identity implies $m' \in V^{>r - 1-N-\ell}\cM$. By discreteness of the $V$-filtration and repeating this argument, we see that $m' \in V^{r -1 - \lfloor -\gamma \rfloor -\ell}\cM$. As $\theta$ shifts $F$ by one, $m' \in F_{p+1-\lfloor -\gamma \rfloor} V^{r-1-\lfloor -\gamma \rfloor -\ell}\cM$.

Thus, $m = (\theta + \ell +1)\cdots (\theta+\ell + \lfloor -\gamma \rfloor +1)(m_{\lfloor -\gamma \rfloor} + m') + \sum_{q =0}^{\lfloor -\gamma \rfloor -1} (\theta + \ell+1)\cdots (\theta+\ell +q) m_q$, which proves the claim inductively.
\end{proof}

In the next subsection, we compute the weight filtration for $\Sp(\cM)$. Note that $W_\bullet \Sp(\cM)$ is a $\cD_{T_ZX}$-submodule, and so it is monodromic. We show that, on each monodromic graded piece, the weight filtration induces the relative monodromy filtration.

\subsection{Admissiblity}\label{subsec:admiss}
For convenience, we recall the definition of the relative monodromy filtration, see Section 1 of \cite{Saito-MHM} for details. 

Let $L$ be a finite increasing filtration on an object $M \in \cC$, an exact category which we take to be embedded in some abelian category $\cA$. Let $S: \cC \to \cC$ be an additive automorphism of the category, which extends to $\cA$.

Let $N: (M,L) \to S^{-1}(M,L)$ be a filtered morphism such that $N^i =0$ for $i\gg 0$. Here the filtration $L$ on $S^j M$ is defined as $L_k (S^j M) = S^j(L_k M)$ for any $j\in \Z, k \in \Z$. Then there is at most one finite, increasing filtration $W=W(N,L)$ of $(M,L)$, called the \emph{relative monodromy filtration} which satisfies:
\begin{enumerate} 
    \item $N: (M;L,W) \to S^{-1}(M;L,W[2])$ is a filtered morphism,
    \item $N^i: \gr^W_{k+i} \gr^L_k M \to \gr^W_{k-i} \gr^L_k M$ is an isomorphism for all $i>0$.
\end{enumerate}
Here, recall that an increasing filtration is shifted as $W[j]_\bullet = W_{\bullet -j}$. We shall take $\cC$ the category of filtered $\cD$-modules and $S$ the shifting of the filtration. 

In the theory of mixed Hodge modules, the objects are defined to satisfy the \emph{admissible condition}: if $(\cM,W)$ is a mixed Hodge module with its weight filtration and $g\in \cO_X$ is any locally defined regular function, then 
\begin{enumerate}
    \item the relative monodromy filtration for $\psi_g(M,W)$ exists for the nilpotent monodromy operator on this nearby cycle, with $L_i = \psi_g(W_{i+1}M)$. Similarly, one assumes the existence of the relative monodromy filtration on $\phi_{g,1}(M,W)$, with $L_i = \phi_{g,1}(W_i M)$ defined without a shift.
    \item the three filtrations are compatible \[
0\rightarrow F_\ell V_\alpha W_{i-1}\cM\rightarrow F_\ell V_\alpha W_{i}\cM \rightarrow F_\ell V_\alpha\gr^W_{i}\cM\rightarrow 0,
\] 
\end{enumerate}
where $V$ is the $V$-filtration along $g$.

In the setting of higher codimension, say $Z$ is a smooth subvariety defined by $t_1,\dots, t_r$, it is an easy exercise using the specialization construction to see that the $V$-filtration along $Z$ satisfies a similar property. The associated graded modules $\gr^V_{\chi} \cM$ also have nilpotent operators, given by $\theta - \chi = \sum_{i=1}^r t_i \de_{t_i} - \chi$. 

\begin{lem}\label{thm:compatible}
    Suppose that the triple $(\cM,F,W)$ underlies a graded polarizable mixed Hodge module, then the three filtrations $F, V, W$ are compatible, i.e., the following sequence is exact
    \[
    0\rightarrow F_k V_\alpha W_{i-1}\cM\rightarrow F_k V_\alpha W_{i}\cM \rightarrow F_k V_\alpha\gr^W_{i}\cM\rightarrow 0.
    \]
\end{lem}

\begin{proof} 
We first recall the setting in Section~\ref{subsec:dfn}: let $\widetilde X=\mathcal{S}pec_{X} \left( \sum_{\ell \in\Z}\cI^\ell_Z\cdot u^{-\ell}\right)$ be the deformation to the normal cone along $Z$, where $\cI_Z$ is the ideal sheaf of $Z$ and $\cI^{\ell}_Z=0$ for $\ell<0$. Let $\rho: \widetilde X \rightarrow X$, $p\colon \widetilde X^* \rightarrow X$ be the two structure morphisms and $j:\widetilde X^*\rightarrow \widetilde X$ is the open immersion. Abusing the notation, we also denote by $\rho\colon T_ZX\to Z$ as the restriction of $\rho:\widetilde X\to X$.

Let $\widetilde \cM= j_+p^*\cM$. Then by Saito's theory~\cite{Saito-MHM}, there exist filtrations $F_\bullet\widetilde \cM$ and $W_\bullet\widetilde\cM$ on $\widetilde \cM$ such that the triple $(\widetilde \cM, F_\bullet\widetilde\cM, W_\bullet\widetilde\cM)$ underlies a graded polarizable mixed Hodge module and that $j^*F_\bullet\widetilde\cM= \bigoplus_{\ell \in \Z} F_{\bullet+1} \cM u^\ell$ and $j^*W_\bullet \widetilde \cM= \bigoplus_{\ell \in \Z} W_\bullet \cM u^\ell$. It follows from the compatibility for mixed Hodge modules of the codimension-one case that 
\begin{equation}\label{eq:defcom}
0\rightarrow F_k V_\alpha W_{i-1}\widetilde \cM\rightarrow F_k V_\alpha W_{i}\widetilde\cM \rightarrow F_k V_\alpha\gr^W_{i}\widetilde\cM\rightarrow 0,
\end{equation}
where $V_\bullet$ is the $V$-filtration along $T_Z X$. Since $V_{<0}$ only depends on the restriction of a $\cD$-module to $\tilde X^*$,  it follows that $V_\alpha W_i \widetilde\cM =V_\alpha j_+ p^*W_i  \cM$ for $\alpha<0$. On the other hand, the Hodge filtration on $V_{\alpha}$ for $\alpha<0$ can be calculated by
\[
F_k V_\alpha W_i \widetilde \cM = F_k V_\alpha W_i j_+ p^* \widetilde \cM = j_* p^* F_{k+1} W_i \cM \cap V_\alpha j_+ p^* W_i\cM. 
\]
We obtain, for $\alpha<0$,
\[
\rho_* F_k V_\alpha W_{i}\widetilde\cM= \bigoplus_{\ell \in \Z} F_{k+1} V_{\alpha+\ell+1} W_{i} \cM \cdot u^\ell.
\]
Similarly, we have, for $\alpha<0$,
\[
\rho_* F_k V_\alpha \gr^W_{i}\widetilde\cM= \bigoplus_{\ell \in \Z} F_k V_{\alpha+\ell+1} \gr^W_{i}\cM\cdot u^\ell.
\]
Applying $\rho_*$ to the sequence~\eqref{eq:defcom} for $\alpha< 0$ yields an exact sequence on $X$:
\[
0\rightarrow \bigoplus_{\ell \in\Z}F_k V_{\alpha+\ell+1} W_{i-1} \cM \cdot u^\ell \rightarrow \bigoplus_{\ell \in\Z}F_k V_{\alpha+\ell+1} W_{i} \cM \cdot u^\ell \rightarrow \bigoplus_{\ell\in\Z}F_k V_{\alpha+\ell+1} \gr^W_{i} \cM \cdot u^\ell \rightarrow 0 .
\]
Since the morphisms in the above sequence respect the grading, we have 
\[
0\rightarrow F_k V_\alpha W_{i-1} \cM\rightarrow F_k V_\alpha W_{i}\cM \rightarrow F_k V_\alpha\gr^W_{i}\cM\rightarrow 0
\]
for every $\alpha\in\Q$. We conclude the proof.
\end{proof}

\begin{lem}\label{lem:relmonext}
    If $(\cM,F,W)$ is a bifiltered $\cD_X$-module underlying a mixed Hodge module with the weight filtration $W$, then the relative monodromy filtration $W(\theta-\chi,L)$ on $\gr_\chi^V\cM$ exists where $L_\bullet\gr_\chi^V\cM=\gr_\chi^V (W_\bullet\cM)$ is induced by the weight filtration. 
\end{lem}
\begin{proof}
    The relative monodromy filtration $W=W(u\partial_u-\alpha,L)$ exists on $\gr^V_\alpha\widetilde \cM$ for $\alpha\in [-1,0]$ because $\widetilde \cM$ is a mixed Hodge module. Then since $W_k\gr^V_\alpha\widetilde \cM$ is invariant under the $\mathbf G_m$-action $u\de_u$, applying $\rho_*$ gives 
    \[
    \rho_* W_k\gr^V_\alpha\widetilde\cM= \bigoplus_{\ell \in \Z} W_k\gr^V_{\alpha+\ell+1}\cM\cdot u^\ell.
    \] 
    induces a filtration $W$ on each $\gr^V_{\alpha+\ell+1}\cM$. We easily check that $W\gr^V_{\alpha+\ell+1}\cM$ is the relative monodromy filtration $W(\theta-\alpha-\ell-1, L)$ if $\alpha<0$. Indeed, we have seen that, for $\alpha<0$ 
    \[
    \rho_* \gr^W_{k+i}\gr^L_i \gr^V_\alpha \widetilde \cM = \bigoplus_{\ell \in\Z}\gr^W_{k+i}\gr^L_i\gr^V_{\alpha+\ell+1}\cM\cdot u^\ell.
    \]
    The isomorphism $(u\partial_u-\alpha)^k: \gr^W_{k+i}\gr^L_i \gr^V_\alpha \widetilde \cM  \to \gr^W_{-k+i}\gr^L_i \gr^V_\alpha \widetilde \cM$ commutes with the $\mathbf G_m$-action so it induces an isomorphism on each graded piece after we apply $\rho_*$.
\end{proof}

\begin{lem}\label{lem:e2}
Let $(\cM,F)$ be a filtered $\cD$-module underlying a mixed Hodge module on a product of smooth varieties $Y\times X$. Let $\pr_2:Y\times X\rightarrow X$ be the second projection and $V_\bullet\cM$ is the $V$-filtration along $Y\times Z$. Suppose that $\pr_2$ is projective on the support of $\cM$. Then we have: 
\begin{enumerate} 
    \item The spectral sequence associated to the relative monodromy filtration on ${\pr_2}_+(\gr_\alpha^V\cM,F)$ degenerates at the second page $E_2$ in the category of filtered $\cD$-modules. 
    \item If $(\cM,F)$ underlies a polarizable Hodge module, then $E^{p,q}_2$ is a filtered direct summand of $E^{p,q}_1$.
    \item If $(\cM,F)$ underlies a polarizable Hodge module and $W\gr^V_\alpha\cM$ is the monodromy filtration, then the image of $\cH^i {\pr_2}_+W_k \gr_\alpha^V  \cM$ in $\cH^i {\pr_2}_+ \gr_\alpha^V \cM$ is the monodromy filtration of 
    \[\gr_\alpha^V \cH^i {\pr_2}_+\cM=\cH^i {\pr_2}_+ \gr_\alpha^V \cM.\]
    \item We have the decomposition in the filtered derived category of $\cD$-modules
    \[
    {\pr_2}_+ (\gr^W_k\gr_\alpha^V\cM,F)\simeq \bigoplus_i (\cH^i {\pr_2}_+ \gr^W_k \gr_\alpha^V\cM, F)[-i]
    \]
    where $W\gr_\alpha^V\cM$ is the relative monodromy filtration.
\end{enumerate} 
\end{lem}
\begin{proof}
Let $\widetilde \pr_2\colon Y\times T_ZX \to T_ZX$ be the induced morphism on the normal bundles. Then we obtain the following commutative diagram by abusing notation.
\[
\begin{tikzcd}
 Y\times T_Z X \arrow{r}{\widetilde \pr_2}\arrow{d}{\rho} & T_ZX \arrow{d}{\rho}  \\
 Y\times Z \arrow{r}{\pr_2} & Z
\end{tikzcd}
\]
Then by the exactness of $\rho_*$, for $\alpha<0$
\[
\begin{aligned}
\rho_*{\widetilde {\pr_2}}_+(\gr^V_\alpha\widetilde\cM,F)=\rho_*\bR\widetilde{\pr_2}_*(\gr^V_\alpha\widetilde\cM\otimes \bigwedge^{-*}\mathscr T_{Y},F)  = \bR{\pr_2}_*\rho_*(\gr^V_\alpha\widetilde\cM\otimes \bigwedge^{-*}\mathscr T_{Y},F)\\
 = \bR{\pr_2}_*\bigoplus_{\ell\in\Z}(\gr^V_{\alpha+\ell+1}\cM\otimes \bigwedge^{-*}\mathscr T_{Y},F)\cdot u^\ell=\bigoplus_{\ell\in Z}{\pr_2}_+(\gr^V_{\alpha+\ell+1}\cM,F)\cdot u^\ell.
\end{aligned}
\]
Here, we identify $\mathscr T_Y$ with $\mathscr T_{Y\times T_ZX/T_ZX}$ and the Hodge filtration $F_\bullet \gr^V_\alpha\widetilde\cM\otimes \bigwedge^{-*}\mathscr T_{Y}$ is given by
\[
F_{\bullet+*}\gr^V_\alpha\widetilde\cM\otimes \bigwedge^{-*} \mathscr T_Y.
\]
In particular, we see that $\rho_*{\widetilde {\pr_2}}_+=\rho_*{{\pr_2}}_+$ in the derived category of filtered $\cD$-modules. Therefore, by functoriality of the spectral sequences,  we obtain a relation between the spectral sequence $E_r^{p,q}\left({\widetilde{\pr_2}}_+(\gr^V_\alpha\widetilde\cM,F)\right)$ induced by the relative monodromy filtration on ${\pr_2}_+\gr^V_\alpha\widetilde\cM$ and the spectral sequence $E_r^{p,q}({\pr_2}_+(\gr^V_{\alpha}\cM,F))$ on ${\pr_2}_+(\gr^V_{\alpha}\cM,F)$:
\begin{equation}\label{eq:spec}
\rho_*E_r^{p,q}\left({\widetilde {\pr_2}}_+(\gr^V_\alpha\widetilde\cM,F)\right)=\bigoplus_{\ell\in \Z}E_r^{p,q}\left({\pr_2}_+(\gr^V_{\alpha+\ell+1}\cM,F)\right)\cdot u^\ell \quad \text{for } \alpha<0.
\end{equation}
Moreover, the differential $d_r$ is compatible with the direct sum decomposition. Due to the fact that  $(\gr^V_\alpha\widetilde\cM,F,W)$ underlies a mixed Hodge module, the  spectral sequence $E_r^{p,q}\left({\widetilde {\pr_2}}_+(\gr^V_\alpha\widetilde\cM,F)\right)$ induced by the relative mondromy weight filtration on $\widetilde{\pr_2}_+(\gr^V_\alpha\widetilde\cM,F)$ degenerates at the second page. Therefore, the spectral sequence $E_r^{p,q}({\pr_2}_+(\gr^V_\alpha\cM,F))$ also degenerates at the second page.

Since polarizable Hodge modules are semisimple~\cite[5.2.13]{Saito-MHP}, $E^{p,q}_2(\gr_\alpha^V\widetilde \cM,F)$ is a summand of $E^{p,q}_1(\gr_\alpha^V\widetilde \cM,F)$. Thanks to~\eqref{eq:spec} again, $E_2^{p,q}({\pr_2}_+\gr^V_\alpha\cM)$ is a summand of $E_1^{p,q}({\pr_2}_+\gr^V_\alpha\cM)$.

It follows from~\cite[5.3.4.2]{Saito-MHP} that the image of $\cH^i \widetilde{\pr_2}_+W_\bullet \gr_\alpha^V\widetilde \cM$ in $\cH^i \widetilde{\pr_2}_+ \gr_\alpha^V\widetilde\cM$ is the monodromy filtration $W_\bullet \gr_\alpha^V \cH^i \widetilde{\pr_2}_+ \widetilde\cM$. Noticing that $\rho_*$ is exact and $\rho_*\cH^i \widetilde{\pr_2}_+ =\cH^i{\pr_2}_+\rho_*$, we get 
\[
\rho_*W_\bullet \gr_\alpha^V \cH^i \widetilde{\pr_2}_+\widetilde\cM=\bigoplus_{\ell \in \Z} W_\bullet \gr_{\alpha+\ell+1}^V \cH^i {\pr_2}_+\cM \cdot u^\ell
\]
for $\alpha<0$. Then $W_\bullet\gr_{\alpha+\ell+1}^V \cH^i {\pr_2}_+\cM$ is the monodromy filtration $W(\theta-\alpha-\ell-1)$ as pointed out in the proof of Lemma~\ref{lem:relmonext}. Then applying $\rho_*$ for $\alpha<0$ to 
\[
\cH^i \widetilde{\pr_2}_+W_\bullet \gr_\alpha^V\widetilde \cM \twoheadrightarrow W_\bullet \gr_\alpha^V \cH^i \widetilde{\pr_2}_+ \widetilde\cM \hookrightarrow \cH^i \widetilde{\pr_2}_+ \gr_\alpha^V\widetilde\cM
\] 
gives:
\[
\bigoplus_{\ell \in \Z}\cH^i {\pr_2}_+W_\bullet \gr_{\alpha+\ell+1}^V \cM\cdot u^\ell \twoheadrightarrow \bigoplus_{\ell \in \Z}W_\bullet \gr_{\alpha+\ell+1}^V \cH^i {\pr_2}_+ \cM \cdot u^\ell \hookrightarrow \bigoplus_{\ell \in \Z}\cH^i {\pr_2}_+ \gr_{\alpha+\ell+1}^V\cM\cdot u^\ell
\] 
respecting the grading. We have proved $(c)$.

For $(d)$, since $\gr^W_k\gr_\alpha^V\widetilde \cM$ is a polarizable Hodge module, by choosing an ample class on $Y$, it follows from the hard Lefschetz theorem that 
\[
 \cH^{-i} \widetilde{\pr_2}_+ (\gr^W_k \gr_\alpha^V\widetilde \cM, F)\cong \cH^{i}\widetilde{\pr_2}_+ (\gr^W_k \gr_\alpha^V\widetilde \cM, F)(i),
\]
where $(i)$ is the Tate twist. This implies for $\alpha<0$ by applying $\rho_*$;
\[
\bigoplus_{\ell\in \Z}\cH^{-i}{\pr_2}_+ (\gr^W_k\gr_{\alpha+\ell+1}^V  \cM,F)\cdot u^\ell \simeq  \bigoplus_{\ell \in\Z}(\cH^i {\pr_2}_+ \gr^W_k \gr_{\alpha+\ell+1}^V \cM, F)(i)\cdot u^\ell,
\]
respecting the grading. Therefore, as a consequence, we have the decomposition as claimed.
\end{proof}

\begin{lem}\label{lem:relmonexact} 
For any short exact sequence of mixed Hodge modules
\[
0\rightarrow \cM' \rightarrow \cM \rightarrow \cM'' \rightarrow 0,
\]
the induced sequence 
\[
0\rightarrow (\gr^V_\alpha \cM', F,W) \rightarrow (\gr^V_\alpha \cM, F,W) \rightarrow (\gr^V_\alpha \cM'', F,W) \rightarrow 0
\]
is bifiltered exact, where $W$ is the relative monodromy filtration.
\end{lem}
\begin{proof}
    By the assumption and~\cite[2.5]{Saito-MHM}, we have
    \[
    0\rightarrow (\gr^V_\alpha \widetilde \cM', F,W) \rightarrow (\gr^V_\alpha \widetilde \cM, F,W) \rightarrow (\gr^V_\alpha \widetilde \cM'', F,W) \rightarrow 0
    \]
    is exact for $\alpha\in [-1,0)$. Then the rest of the proof goes like the proof of the above two Lemmas.
\end{proof}

Recall the definition of a \emph{bistrict} morphism between two bifiltered objects $\phi: (\cM,F,F') \to (\cN,F,F')$. Such a morphism is one which is strict with respect to both $F$ and $F'$ and which satisfies
\[
\begin{aligned}
 (F_p \cN + \mathrm{im} \phi) \cap (F_q' \cN + \mathrm{im} \phi) = F_pF_q'\cN + \mathrm{im} \phi 
 \quad \text{and} \quad (F_pF_q' \cN)\cap \mathrm{im} \phi  = \phi(F_pF_q' \cM).
\end{aligned}
\]
A complex is \emph{bistrict} if all morphisms are bistrict. Bistrict complexes have the property that
\[ 
F_pF_q' \cH^k C^\bullet = \cH^k(F_p F_q' C^\bullet).
\]
As a corollary of Lemma \ref{SpHFilt}, we show that any morphism between mixed Hodge modules is bistrict along $(F,V)$ for $V$ the $V$-filtration along a smooth subvariety $Z$ of codimension $r$.

\begin{cor} \label{bistrict} Let $\varphi: M \to N$ be a morphism of mixed Hodge modules on $X$. Then the corresponding map on filtered $\cD_X$-modules $\varphi: \cM \to \cN$ is bistrict with respect to the Hodge filtration $F_\bullet \cM$ and $V^\bullet\cM$.
\end{cor}
\begin{proof} By the theory of mixed Hodge modules, $\varphi$ is strict with respect to the Hodge filtration. Similarly, by uniqueness of $V$-filtrations, it is strict with respect to $V^\bullet$.

We now prove the remaining two conditions for bistrictness hold.

We know again by the theory of mixed Hodge modules that $\widetilde{\varphi}: j_+ p^*(\cM)[-1] \to j_+ p^*(\cN)[-1]$ is bistrict with respect to the Hodge filtration and the $V$-filtration along $u$.

Let $n \in (F_p \cN + im(\phi))\cap (V^\lambda \cN + im(\phi))$. Write $n = \widetilde{n} + \phi(m)$ for some $\widetilde{n} \in F_p \cN$ and $m \in \cM$. If $\widetilde{n} \in V^\lambda \cN$, then we are done. Otherwise, $\widetilde{n} \in V^{\beta}N$ for some $\beta < \lambda$.

Let us determine a range of $\ell$ such that $\widetilde{n} u^\ell \in F_{p-1} j_+(p^*(\cN)[1])$. Using the formula from \ref{SpHFilt} with $\gamma =0$, since $\widetilde{n} \in F_p V^{\beta}(\cN)$, we can force this to be true if $-\ell + r -1 < \beta$, or $r - \beta - 1 < \ell$. Hence, for all $\ell > r - \beta -1$, we see that $\widetilde{n} u^\ell \in F_{p-1}j_+(p^*(\cN)[1])$.

On the other hand, writing $n = n' + \phi(m')$ for $n' \in V^{\lambda}\cN$, then $n' u^\ell \in V^{\alpha}j_+(p^*(\cN)[1])$ whenever $\alpha - \ell = \lambda - r +1$. To apply \ref{SpHFilt}, it will be useful to have $\alpha \geq 0$. Thus, we need $\ell$ satisfying
\[ 0\leq \alpha = \lambda + \ell -r +1, \ell > r - \beta -1,\]
which are both clearly possible for $\ell \gg 0$. Fix some $\ell$ for which both inequalities are true.

Define $\alpha = \lambda + \ell -r+1$, then $n u^\ell = (\widetilde{n} + \phi(m)) u^\ell \in F_{p-1} j_+(p^*(\cN)[1]) + \text{im}(\widetilde{\phi})$, and $n u^\ell = (n' + \phi(m')) u^\ell \in V^{\alpha} j_+(p^*(\cN)[1]) + \text{im}(\widetilde{\phi})$. Here 
\[ \widetilde{\phi}: j_+(p^*(\cM)[1]) \to j_+(p^*(\cN)[1]),\]
is a morphism of mixed Hodge modules.

In particular, it is bistrict with respect to $F$ and the $V$-filtration along $\{u=0\}$. Thus, we see that $nu^\ell \in (F_{p-1}V^{\alpha}j_+(p^*(\cN)[1]) + \text{im}(\widetilde{\phi})$. As it is homogeneous of degree $\ell$, we get
\[ n = n'' + \phi(m''),\]
where $n'' \in F_p V^\lambda \cN$, where we use the fact that $\alpha \geq 0$ and apply \ref{SpHFilt}. This proves the first condition needed for bistrictness.

For the last condition, let $n \in F_pV^\lambda \cN \cap \text{im}(\phi)$. Find $\gamma \geq 0$ such that $\gamma - \ell +r-1 = \lambda$, by choosing $\ell \gg 0$. Then $n u^\ell \in F_{p-1}V^{\gamma}j_+(p^*(\cN)[1])$, and it lies in $\text{im}(\widetilde{\phi})$, so by the fact that $\widetilde{\phi}$ is bistrict with respect to $V$ and $F$, there exists some $\sum m_i u^i\in F_{p-1}V^{\gamma}j_+(p^*(\cM)[1])$ which maps to $nu^\ell$ under $\widetilde{\phi}$. Taking $m_\ell u^\ell$ and using \ref{SpHFilt}, we see that $m_\ell \in F_p V^{\lambda}\cM$, and $\phi(m_\ell) = n$, proving the second condition for bistrictness.
\end{proof}

\subsection{Saito's Main Theorems about Hodge Modules}
In this section, we state two essential theorems in Saito's theory of mixed Hodge modules.

The first main result is the behavior of mixed Hodge modules with respect to the pushforward functor for a projective morphism $f:Y \to X$. For more details and proofs, see ~\cite[Section 16]{Schnell-MHM} or ~\cite[Section 5.3]{Saito-MHP}.

\begin{defi} We say a morphism $\varphi: (\cM,F) \to (\cN,F)$ is \emph{strict} if $F_p\cN \cap im(\varphi) = \varphi(F_p \cM)$. We say that a filtered complex $(K^\bullet,F)$ is \emph{strict} if all differentials are strict.
\end{defi}

For example, a monomorphism $i: A\hookrightarrow B$ is strict iff the filtration on $A$ is the induced filtration from $B$. The main utility of strictness is that, if $(K^\bullet,F)$ is a filtered complex with strict differentials, then $\cH^k(F_pK^\bullet) \to \cH^k(K^\bullet)$ is injective for all $k\in \Z$. Hence, we can define a filtration $F$ on $\cH^k(K^\bullet)$, and strictness allows us to commute $\cH^k$ with $F_p$.

We begin now with the statement of the direct image theorem in the pure case:
\begin{thm}[{\cite[Thm 5.3.1]{Saito-MHP}}] \label{dirimthm} Let $f:Y \to X$ be a projective morphism of smooth complex varieties, let $M$ be a pure Hodge module on $Y$ of weight $w$. Let $\ell \in H^2(Y,\Z)$ be the class of a relatively ample divisor over $Y$. Then
\begin{enumerate} \item $f_+(\cM,F)$ is strict and $\cH^i f_+(\cM,F)$ underlies a Hodge module on $X$ of weight $w+i$.

\item $\ell^i: \cH^{-i}f_+(\cM,F) \to \cH^i f_+(\cM,F)(i)$ is an isomorphism for all $i\geq 0$.
\end{enumerate}
\end{thm}

As an application, if $X$ is a smooth projective variety, $f: X\to *$ is the constant map, then the strictness of $f_+(\cM,F)$ recovers the fact that the Hodge-de Rham spectral sequence degenerates at $E_1$.

Also, as a formal consequence of the second part of the theorem (see ~\cite[Prop. 2.1]{Del}), one recovers the decomposition theorem, i.e., an isomorphism in the derived category
\[ f_+(\cM,F) \cong \bigoplus_{i \in\Z} \cH^i f_+(\cM,F)[-i].\]

\begin{rmk} \label{smoothproj} The strictness of $f_+(\cM,F)$ in part (a) of Theorem \ref{dirimthm} still holds if we assume $\cM$ is a mixed Hodge module. One particular application of Theorem \ref{dirimthm} will be when the map $f: Y = Z \times X \to X$ is a smooth, projective projection from a product and $(\cM,F)$ underlies a mixed Hodge module. In this case, the $\cD$-module pushforward $f_+(\cM)$ is given by applying $Rf_*$ to the relative de Rham complex (see ~\cite[Prop. 1.5.28]{HTT})
\[ 
K^\bullet = \left\{\cM \xrightarrow[]{d} \Omega_Z^1\otimes\cM  \xrightarrow[]{d} \dots \xrightarrow[]{d} \Omega_Z^{\dim Z}\otimes\cM  \right\}
\]
and this complex is filtered, given by
\[ 
F_pK^\bullet = \left\{F_p\cM \to \Omega_Z^1\otimes F_{p+1}\cM   \to \dots \to F_{p+\dim Z} \Omega^{\dim Z}_Z\otimes \cM  \right\}.
\]
Then strictness tells us that the induced map
\[  R^i f_*(F_pK^\bullet) \to   R^i f_*(K^\bullet) = \cH^i f_+(\cM)\]
is injective, and defines the Hodge filtration on this cohomology module.
\end{rmk}

The second main theorem is called the ``structure theorem for polarizable Hodge modules".

Let $Z\subseteq X$ be an irreducible closed subset. A Hodge module $M$ on $X$ has strict support $Z$ if the underlying $\cD$-module has no subquotient $\cD$-modules supported on a proper subset of $Z$. See ~\cite[Exercise 10.2]{Schnell-MHM} for a characterization of this property in terms of the $V$-filtration along a hypersurface. See also our generalization of this property to higher codimension in Corollary \ref{subob} and Corollary \ref{quotob}.

Built into the definition of the category of pure Hodge modules is the property that \emph{every pure Hodge module has a decomposition by strict support}, meaning, for any $M$ pure on $X$, we have
\[ M = \bigoplus_{Z \subseteq X} M_Z,\]
where the direct sum ranges over irreducible closed subsets of $Z$, $M_Z \neq 0$ for only finitely many $Z$, and each $M_Z$ is a pure Hodge module with strict support $Z$. See ~\cite[Theorem 11.7]{Schnell-MHM} for a characterization of this property in terms of the $V$-filtration. See our generalization of this property to higher codimension in Corollary \ref{strictsuppgen}.

The structure theorem gives a description of those pure Hodge modules with strict support $Z$: they are generically given by (polarizable) variations of Hodge structure on $Z$. See ~\cite[Section 15]{Schnell-MHM}.

\begin{thm} \label{decompthm} Let $X$ be a smooth complex algebraic variety, $Z\subseteq X$ an irreducible subset. Then
\begin{enumerate} \item Every polarizable variation of Hodge structure of weight $w- \dim Z$ on a Zariski open subset of $Z$ extends uniquely to a polarizable Hodge module on $X$ of weight $w$ with strict support $Z$.

\item Every Hodge module with strict support $Z$ arises in this way.
\end{enumerate}
\end{thm}

The difficult claim is to extend a polarizable VHS to a Hodge module with strict support on $Z$. This result will be used to identify certain Hodge modules as strict support direct summands of other Hodge modules.

\subsection{Monodromic $\cD$-modules}
For background on monodromic $\cD$-modules on a vector bundle $E$, see \cite{Bry}. For results on monodromic mixed Hodge modules on a line bundle, see \cite{monoMHM}.

Let $E$ be a vector bundle of rank $r$ on a smooth variety $X$. A $\cD_E$-module $\cM$ is \emph{monodromic} if, for every local trivialization $U\times \A^r$ of $E$, every choice of coordinates $z_1,\dots, z_r$ on the $\A^r$-factor, with vector fields $\de_{z_1},\dots, \de_{z_r}$ and every local section $m\in \cM$, there exists a univariate polynomial $b(w) \in \C[w]$ such that $b(\theta) m = 0$, where $\theta = \sum_{i=1}^r z_i \de_{z_i}$.

Such modules satisfy several nice properties:
\begin{enumerate}[label=\textbf{M.\arabic*}] 
\item If $\cM$ is monodromic, then it decomposes into generalized eigenspaces for the $\theta$ action, i.e.,
\[ \cM = \bigoplus_{\chi \in \C} \cM^\chi,\]
where $\cM^\chi = \ker((\theta - \chi+r)^\infty)$.
\item Every subquotient of a monodromic $\cD$-module is monodromic.

\item Every morphism between monodromic $\cD$-modules preserves the eigenspace decomposition.

\item $z_i \cM^\chi \subseteq \cM^{\chi+1}, \de_{z_i} \cM^\chi \subseteq \cM^{\chi-1}$ for all $1\leq i \leq r$.

\item  \label{monoVFilt} The ($\C$-indexed) $V$-filtration of a coherent monodromic $\cD_E$-module along the zero section $X\subseteq E$ is given by
\[ V^\lambda \cM = \bigoplus_{\chi \geq \lambda} \cM^\chi,\]
and in particular,
\[ \gr_V^\lambda \cM \cong \cM^\lambda.\]
\end{enumerate}

A mixed Hodge module $M$ on $E$ is \emph{monodromic} if the underlying $\cD$-module is. As $V$-filtrations are $\Q$-indexed for mixed Hodge modules by definition, we know that if $\cM$ underlies a mixed Hodge module, then $\cM =\bigoplus_{\chi \in \Q} \cM^\chi$ by Property \ref{monoVFilt}.

\subsection{Conventions for Shifting the Hodge Filtration} \label{subsec:conv}
We refer to \cite{Schnell-MHM} for all conventions regarding the Hodge filtration and weight filtration when applying functors to mixed Hodge modules when considering \emph{right} $\cD$-modules. As noted at the beginning of Section \ref{sec:pre}, these conventions may differ if we want to use left $\cD$-modules instead. For convenience, we will list here those conventions for left $\cD$-modules.

\emph{Tate Twist}: Let $(\cM,F)$ be a filtered $\cD_X$-module. Then we define $(\cM,F)(k)$ for any $k\in \Z$, the \emph{Tate twist} of $(\cM,F)$ by $k$, to be $(\cM,F[k])$, where $F[k]_p \cM = F_{p-k}\cM$.

\emph{Smooth pullbacks}: See Remark (4.4.2) and Formula (2.17.3) in \cite{Saito-MHM}. Let $p: X\times Y \to Y$ be a smooth surjective morphism of relative dimension $r = \dim X$ between smooth varieties. Let $\widetilde{\cM} = p^*(\cM)$ as an $\cO$-module (which is also the $\cD$-module pullback, see \cite[Sect. 1.3]{HTT}). If $(\cM,F)$ is a filtered left $\cD_Y$-module, let $F_p \widetilde{\cM} = p^*(F_p \cM)$.

If $M$ is a mixed Hodge module with underlying filtered $\cD_Y$-module $\cM$, then the pullback $p^*(M) \in D^b \MHM(X\times Y)$ has underlying filtered $\cD_{X\times Y}$-module 
\begin{equation} \label{smooth*Hodge} 
(\widetilde{\cM},F_\bullet)
\end{equation} 
lying in cohomological degree $r$, and $p^!(M) \in D^b \MHM(Y)$ has underlying filtered $\cD_{X\times Y}$-module given by
\begin{equation} \label{smooth!Hodge} 
(\widetilde{\cM},F_\bullet[r])
\end{equation}
lying in cohomological degree $-r$. The weight filtration is given by
\[ 
W_\bullet p^*(\cM)[r] = p^*(W_{\bullet - r}\cM) \quad  \text{and} \quad W_\bullet p^!(\cM)[-r] = p^*(W_{\bullet +r}\cM).
\]

\emph{Nearby and Vanishing Cycles}: Let $X = \{t=0\} \subseteq Y$ be a smooth hypersurface defined by the global function $t$. Let $\cM$ be a holonomic $\cD_Y$-module. We define
\[ \psi_{t,\lambda}(\cM) = \gr_V^{\lambda}\cM \quad \text{ for } \lambda \in (0,1],\]
\[ \phi_{t,\lambda}(\cM) = \psi_{t,\lambda}\cM \quad \text{ for } \lambda \in (0,1)\]
and
\[ \phi_{t,1}(\cM) = \gr_V^0\cM,\]
where $V^\bullet \cM$ is the $V$-filtration of $\cM$ along $X$.

If $(\cM,F)$ is a filtered holonomic $\cD_X$-module, then the filtration on nearby and vanishing cycles is defined to be
\begin{equation} \label{nearbyHodge} 
F_p \psi_{t,\lambda}(\cM) = \frac{F_p V^\lambda\cM}{F_p V^{>\lambda}\cM} \quad \text{ for } \lambda \in (0,1]\end{equation}
\begin{equation} \label{vanHodge} 
F_p \phi_{t,1}(\cM) = \frac{F_{p+1} V^0\cM}{F_{p+1}V^{>0}\cM}.
\end{equation}

Just as the Hodge filtration includes a shift based on if $\lambda = 1$ or $\lambda \in (0,1)$, so does the weight filtration (see ~\cite[Sect. 20]{Schnell-MHM}. We make note of it here for later use: the weight filtration $W_\bullet \phi_{t,\lambda}(\cM)$ for $(\cM,W_\bullet)$ a $\cD$-module underlying a mixed Hodge module is defined to be the relative monodromy filtration (as defined in Subsection \ref{subsec:admiss} above) of $L_\bullet \phi_{t,\lambda}(\cM)$ along the nilpotent operator $N = \de_t t - \lambda$. here, $L_\bullet \phi_{t,\lambda}(\cM)$ is defined as
\begin{equation}\label{nearbyWeight} 
L_k \phi_{t,1}(\cM) = \gr_V^0(W_k\cM),
\end{equation}
\begin{equation}\label{vanWeight} 
L_k \phi_{t,\lambda}(\cM) = \gr_V^{\lambda}(W_{k+1}\cM) \quad \text{ for }  \lambda \in (0,1).
\end{equation}

\section{Topological properties of $V$-filtration}  \label{sec:topo}

In this section we first prove some basic properties of $V$-filtrations along a smooth subvariety. The analogous statements for a codimension 1 subvariety appear in ~\cite[Section 3]{Saito-MHP}. Now let us fix the notation. Let $X$ be a smooth variety and $Z$ be a smooth subvariety of codimension $r$ globally defined by regular functions $t_1,t_2,\dots,t_r$. Assume there exist global vector fields $\partial_1,\partial_2,\dots,\partial_r$ dual to the $1$-forms $dt_1,dt_2,\dots,dt_r$. Let $\cM$ be a \emph{right} holonomic $\cD_X$-module along $Z$ and $V_\bullet \cM$ be the $V$-filtration along $Z$. Recall that we have introduced the following notation: for a \emph{right} holonomic $\cD_X$-module $\cM$, we define 
\[
\begin{aligned}
A_{\alpha}(\cM) &= \left\{V_{\alpha}\cM\rightarrow (V_{\alpha-1}\cM)^r \rightarrow \cdots \rightarrow V_{\alpha-r}\cM\right\},  & \text{in degrees } &0, 1,\dots,r;  \\
B_{\alpha}(\cM) &= \left\{\gr^V_{\alpha}\cM\rightarrow (\gr^V_{\alpha-1}\cM)^r \rightarrow \cdots \rightarrow \gr^V_{\alpha-r}\cM\right\},  & \text{in degrees } &0, 1,\dots,r;  \\
C_{\alpha}(\cM) &= \left\{\gr^V_{\alpha-r}\cM\rightarrow (\gr^V_{\alpha-r+1}\cM)^r \rightarrow \cdots \rightarrow \gr^V_{\alpha}\cM\right\},  &  \text{in degrees } &-r, -r+1,\dots,0. 
\end{aligned}
\]

\begin{thm} \label{thm:gr}
The complexes $B_{\alpha}(\cM)$ and $C_{\alpha}(\cM)$ are exact for $\alpha \neq 0$.
\end{thm}
\begin{proof}
We shall construct a retraction on the complex $B_{\alpha}(\cM)$, i.e. a series of morphisms 
\[
s_\ell: \left(\gr^V_{\alpha-\ell}\cM\right)^{\binom{r}{\ell}} \rightarrow \left(\gr^V_{\alpha-\ell+1}\cM \right)^{\binom{r}{\ell-1}}
\]
such that $s_{\ell+1}\circ d_\ell+d_{\ell-1}\circ s_\ell=\theta+\ell$ where $d$ is the differential of the complex $B_{\alpha}(\cM)$. Note that the collection $\{\theta+\ell\}$ gives an endomorphism of the complex $B_{\alpha}(\cM)$. Let 
\[
(\gr^V_{\alpha-1}\cM)^r=\bigoplus^r_{i=1} \gr^V_{\alpha-1}\cM e_i
\] 
where $e_1,e_2,\dots,e_r$ is a standard basis such that the Koszul differential works as
\[
d_\ell \left(\eta\, e_{i_1}\wedge e_{i_2}\wedge\cdots\wedge e_{i_\ell} \right)= \sum^r_{i=1} \eta t_i\, e_i\wedge  e_{i_1}\wedge e_{i_2}\wedge\cdots\wedge e_{i_\ell},
\]
where $\eta$ is a local section of $\gr^V_{\alpha-\ell}\cM$. Now we can define the morphism
\[
s_\ell \left(\eta\, e_{i_1}\wedge e_{i_2}\wedge\cdots\wedge e_{i_\ell} \right)= \sum^r_{j=1}\eta\partial_j \, e^*_j( e_{i_1}\wedge e_{i_2}\wedge\cdots\wedge e_{i_\ell}),
\]
where $\{e^*_1,e^*_2,\dots,e^*_r\}$ is the dual basis and 
\[
e^*_j( e_{i_1}\wedge e_{i_2}\wedge\cdots\wedge e_{i_\ell})=\sum_{k=1}^{\ell} (-1)^{k-1} e_{i_1}\wedge e_{i_2}\wedge\cdots \wedge e_j^*(e_{i_k}) \wedge \cdots \wedge e_{i_\ell}.
\]
Without loss of generality, we only do the computation on $\eta e_1\wedge e_2\wedge\cdots\wedge e_\ell$:
\[
\begin{aligned}
 & (s_{\ell+1}\circ d_\ell+d_{\ell-1}\circ s_\ell) (\eta\,  e_1\wedge e_2\wedge\cdots\wedge e_\ell) \\
= & s_{\ell+1} (\sum^r_{i=1}\eta t_i\, e_i\wedge e_1\wedge e_2\wedge\cdots\wedge e_\ell) + d_{\ell-1}\big(\sum^r_{j=1}\eta\partial_j \, e^*_j( e_{1}\wedge e_{2}\wedge\cdots\wedge e_{\ell})\big) \\
=& \sum^r_{k=1}  \sum^r_{i=1}\eta t_i\partial_k \, e^*_k( e_i\wedge e_1\wedge e_2\wedge\cdots\wedge e_\ell) +\sum^r_{a=1} \sum^r_{j=1}\eta\partial_j t_a\, e_a\wedge e^*_j( e_{1}\wedge e_{2}\wedge\cdots\wedge e_{\ell}) \\
=& \eta \left(\sum^r_{i=1} t_i\partial_i+\ell \right) \, e_1\wedge e_2\wedge\cdots\wedge e_\ell  \\
=& \eta (\theta+\ell)\,e_1\wedge e_2\wedge\cdots\wedge e_\ell.
\end{aligned}
\]
Because $\theta+\ell=\left(\theta-\left(\alpha-\ell\right)\right)+\alpha$, the scalar multiplication by $\alpha$ is equal to the nilpotent operator $\theta-\left(\alpha-\ell\right)$ on the $\ell$-th cohomology of $B_{\alpha}(\cM)$. This can happen for $\alpha\neq 0$ if and only if the $\ell$-th cohomology vanishes. We conclude that the complex $B_{\alpha}(\cM)$ is exact for $\alpha\neq 0$.

The proof of the exactness of the complex $C_{\alpha}(\cM)$ is similar and we leave the rest of the proof to the readers. 
\end{proof}

\begin{thm} \label{thm:alphaless0}
The complex $A_{\alpha}(\cM)$ is exact for $\alpha<0$.
\end{thm}
\begin{proof}
It suffices to show that the complex $A_{\alpha}(\cM)$ is exact for $\alpha \ll 0$ by Theorem~\ref{thm:gr}. We will prove the claim for $j \in \Z_{\ll 0}$. The $\Z$-indexed filtration $V_\bullet \cM$ is a good filtration with respect to the $\Z$-indexed $V$-filtration $(\cD_X,V_\bullet)$, and so we can locally take a $(\cD_X,V)$-free filtered resolution
\[
\cdots \to (L^1,V) \to (L^0,V) \to (\cM,V)
\]
of $(\cM, V)$ such that $(L^i, V)=\bigoplus_{j=1}^{m_i} (\cD_X,V[n_j^{(i)}])$. It follows that the total complex $T^\bullet$ of the double complex 
\[
\cdots \to A_\alpha(L^1,V) \to A_\alpha(L^0,V)
\]
is quasi-isomorphic to $A_\alpha(\cM)$. As $A_\alpha(-)$ gives a bounded complex, the terms $T^{-1},T^0,\dots, T^r$ only concern $L^0, L^{1},\dots, L^{r+1}$. We show that for $\alpha < \min_{0\leq i \leq r+1,j} \{n_j^{(i)}\}$ the complexes $A_\alpha(L^i)$ are acyclic for $0 \leq i \leq r+1$, proving the claim.

The claim is thus equivalent to $A_\alpha(\cD_X)$ is exact for $\alpha \leq 0$. We will prove $\gr^F A_\alpha(\cD_X)$ is acyclic, which implies the desired claim.

To see this, using local coordinates we reduce to the case $X = \A^n$ and $Z$ is defined by the vanishing of coordinates $t_1,\dots, t_r$. Then $\gr^F \cD_X = R[t_1,\dots, t_r,\xi_1,\dots, \xi_r]$ where $R = \cO_{\A^{n-r}}[\xi_{r+1},\xi_{r+2},\cdots,\xi_n]$ and $\xi_i$ is the principal symbol of $\de_{t_i}$. We put a grading on this ring where $R$ has degree $0$, $t_i$ has degree $-1$ and $\xi_i$ has degree $+1$ for $i=1,2,\cdots,r$. Then the module $\gr^F V_j\cD_X$, in terms of grading, is given by:
\[ \gr^F(V_j \cD_X) = \bigoplus_{\beta \leq j} (\gr^F\cD_X)_\beta,\]
where $(\gr^F\cD_X)_\beta$ is the $\beta$-graded piece.

The elements $t_1,\dots, t_r$ clearly form a regular sequence on $\gr^F \cD_X=\bigoplus_{\beta \in \Z} (\gr^F\cD_X)_\beta$, and so the corresponding Koszul complex 
\[
\gr^FA_\infty(\cD_X):=\left\{\gr^F \cD_X\xrightarrow{(t_1,t_2,\cdots,t_r)} (\gr^F\cD_X)^r\to \cdots \to \gr^F\cD_X\right\}
\] 
gives a resolution of $R[\xi_1,\dots,\xi_r]$. All maps in the Koszul complex are graded with respect to the grading just defined. In particular,  the resulting degree $\beta$ complex
\[
K_\beta:=\left\{(\gr^F \cD_X)_\beta \xrightarrow{(t_1,t_2,\cdots,t_r)} (\gr^F\cD_X)^r_{\beta-1} \to \cdots \to (\gr^F\cD_X)_{\beta-r}\right\}
\] 
is a resolution, as it is a direct summand of $\gr^FA_\infty(\cD_X)$. The claim is that for $\alpha \leq 0$, the complex
\[ \gr^F A_\alpha(\cD_X) = \bigoplus_{\beta \leq \alpha} K^\bullet_\beta\]
is acyclic. The only possibly non-zero cohomology is the rightmost one whose grading is $\beta - r$ because $K_\beta$ is a resolution. But since $\alpha\leq 0$, by the definition of $V_\bullet\cD_X$, the rightmost cohomology of $\gr^FA_\alpha(\cD_X)$ is automatically zero, proving the claim.
\end{proof}

We give some elementary applications of Theorem~\ref{thm:gr} and Theorem~\ref{thm:alphaless0}. As a consequence we give a criterion for when $\cM$ has strict support decomposition along $Z$.

\begin{cor}\label{subob}  A $\cD_X$-module $\cM$ with a $V$-filtration along $Z$ has no submodules supported on $Z$ if and only if $\gr_0^V \cM \xrightarrow[]{t} \bigoplus_{i=1}^r \gr_{-1}^V \cM$ is injective.
\end{cor}
\begin{proof} If $m\in \cM$ is such that $m t_i= 0$ for all $i$, then $m \in V_0\cM$. Indeed, $m\in V_{\lambda}\cM$ for some $\lambda \in \Q$. If $\lambda \leq 0$, we are done. Otherwise, considering the short exact sequence
\[
0 \to A_{<\lambda}(\cM) \to A_{\lambda}(\cM) \to B_{\lambda}(\cM) \to 0,
\]
by acyclicity of $B_{\lambda}(\cM)$ for $\lambda \neq 0$, the left-most map being injective implies $m\in V_{<\lambda}\cM$. Since the $V$-filtration is discrete, by induction we know that $m\in V_0\cM$. This means that $\cM$ has no submodules supported on $Z$ if and only if $\bigcap^r_{t=1}\ker(t_i:V_0\cM \to V_{-1}\cM)$ vanishes.

Since $A_{<0}(\cM)$ is acyclic, it follows from the short exact sequence and the snake lemma
\[
0 \to A_{<0}(\cM) \to A_0(\cM) \to B_0(\cM) \to 0.
\] 
that $\bigcap^r_{t=1}\ker(t_i:\gr^V_0\cM \to \gr^V_{-1}\cM)=\bigcap^r_{t=1}\ker(t_i:V_0\cM \to V_{-1}\cM)$, which concludes the proof.
\end{proof}

\begin{cor} \label{quotob} Let $\cM'$ be the smallest submodule of $\cM$ such that $\cM'|_U \cong \cM|_U$. Then 
\[
\cM/\cM' \cong i_+\mathrm{coker}\left(\bigoplus_{i=1}^r \gr_{-1}^V\cM \xrightarrow{\de_{t}} \gr_0^V\cM\right)
.\] 
In particular, the morphism $\bigoplus_{i=1}^r \gr_{-1}^V \cM \to \gr_0^V\cM$ is surjective if and only if $\cM$ has no quotients supported on $Z$.
\end{cor}
\begin{proof} Note that $\cM' =  V_{\lambda}\cM \cdot \cD_X$ for any $\lambda <0$. Indeed, we know that $V_\lambda \cM' = V_\lambda \cM$ if $\lambda <0$, as they restrict to the same module on $X -Z$. Thus, $ V_\lambda \cM \cdot  \cD_X=  V_\lambda \cM' \cdot \cD_X  \subseteq \cM'$. For the other inclusion, note that $(V_\lambda \cM \cdot \cD_X)|_U = \cM|_U$, because the $V$-filtration is all of $\cM$ away from $Z$. Hence, by minimality of $\cM'$, we get the desired equality.

Note that $\cM/\cM'$ is supported on $Z$, so by Kashiwara's equivalence $\cM/\cM'=i_+\gr^V_0\left(\cM/\cM' \right)$, where $i: Z \to X$ is the inclusion. We know $\gr^V_0 \cM/\cM' = \gr^V_0 \cM /\gr^V_0 \cM' $ and 
\[ \gr^V_0(\cM') = \frac{V_0\cM \cap \cM'}{V_{<0}\cM},\]
because $V_{<0}\cM = V_{<0}\cM'$ and $V_\bullet\cM \cap \cM'=V_\bullet\cM'$ by the uniqueness of the $V$-filtration. Thus, the claim reduces to proving
\[ V_0\cM \cap \cM' = \sum_{i=1}^r V_{-1}\cM\de_{t_i} + V_{<0}\cM.\]

In fact, we can define inductively a filtration $U_\bullet\cM'$ by $U_\lambda\cM'= \sum_{i=1}^r U_{\lambda-1}\cM\de_{t_i}+U_{<\lambda}\cM$ for $\lambda\geq 0$ and $U_\lambda\cM'=V_\lambda \cM'$ for $\lambda<0$. Note that $V_\lambda\cM'=V_\lambda\cM$ for $\lambda<0$ is discrete so $U_\bullet\cM'$ is well-defined. Since $\cM'=V_{<0}\cM\cdot \cD_X$, the filtration $U_\bullet \cM$ is exhausted. Then it is easy to check that $U_\bullet\cM'$ satisfies all the characterization of $V$-filtration, i.e. $U_\bullet\cM'=V_\bullet\cM'$ which concludes the proof. 
\end{proof}

We prove here an analogue of the fact from the codimension one case that you can test if a module has a strict support decomposition by looking at $\phi_{f,1}$ as $f\in \cO_X$ varies.

\begin{cor} \label{strictsuppgen} Let $\cM$ be a $\cD_X$-module admitting a $V$-filtration along $Z$. Then there exists a decomposition $\cM = \cM' \oplus \cM''$ with $\text{supp}(\cM') \subseteq Z$ and $\cM''$ having no submodules or quotient modules supported on $Z$ if and only if
\[ 
\gr^V_0\cM = \left(\bigcap_{i=1}^r \ker(t_i: \gr^V_0\cM \to \gr^V_{-1}\cM \right) \bigoplus \left(\sum_{i=1}^r \gr^V_{-1}\cM\de_{t_i}\right).
\]
\end{cor}
\begin{proof} For the ``only if'' part, by the previous lemma we know $\gr^V_0\cM'' = \text{im}(\de_{z_i})$ and $\bigcap_{i=1}^r \ker(t_i: \gr^V_0\cM'' \to \gr^V_{-1}\cM'') = 0$. Also, by Kashiwara's equivalence, we know $\cM'$ satisfies $\gr^V_{-1}\cM'= 0$. By taking $\gr_0^V$ of the equality $\cM = \cM' \oplus \cM''$, we conclude.

For the other implication, note that we must certainly set $\cM' = \cH^0_Z(\cM)$, as this is the maximal submodule of $\cM$ supported on $Z$. Let $\cM'' = V_{<0}\cM \cdot \cD_X$, which we know is the smallest submodule such that $\cM''|_U = \cM|_U$, and satisfies
\[ 
\cM/\cM'' = i_+(\text{coker} \left(\bigoplus_{i=1}^r \gr^V_{-1}\cM \xrightarrow[]{\de_{t_i}} \gr^V_{0}\cM \right)
.\]
By the assumption, this cokernel is isomorphic to $\bigcap_{i=1}^r \ker\left(t_i:\gr^V_0\cM \to \gr^V_{-1}\cM \right)$, and so $\cM/\cM'' \cong \cM'$. But the inclusion $\cM' \to \cM$ splits this quotient map, yielding the direct sum
\[ \cM \cong \cM' \oplus \cM'',\]
which proves the claim.
\end{proof}

For convenience, denote by $B(\cM)=B_0(\cM)$ and $C(\cM)=C_0(\cM)$. To close this section, we give a comparison of the restriction $i_Z^*\cM$ and $i_Z^!\cM$ with $B(\cM)$ and $C(\cM)$ for $i_Z:Z\to X$.

\begin{thm}\label{thm:topop}
With notation as above, the complex $B(\cM)$ (resp. $C(\cM)$) is isomorphic to $i^!_Z \cM$ (resp. $i_Z^*\cM$) in $D^b_{rh}(\cD_Z)$, where $i_Z: Z\rightarrow X$ is the closed embedding.
\end{thm}
\begin{proof}
First, we exhibit the quasi-isomorphism $B(\cM) \cong i^! \cM$. Recall that $i^! \cM$ agrees with the derived $\cO$-module pullback of $\cM$, hence, we can use the Koszul complex on $\cM$ along $t_1,\dots, t_r$.

    Since $B_\alpha(\cM)$ is exact by Theorem~\ref{thm:gr} when $\alpha>0$ and the $V$-filtration is discretely indexed, we find that the natural inclusion of complexes $A_0(\cM) \hookrightarrow A_j(\cM)$ is a quasi-isomorphism for any $j\geq 0$. This implies that the natural inclusion of complexes 
    \begin{equation}
        A_0(\cM) \hookrightarrow A_\infty(\cM):=\left\{\cM \to \cM^{\bigoplus r} \to \cdots \to \cM \right\}
    \end{equation} 
    is a quasi-isomorphism, where the right hand side is the Koszul complex of $\cM$ along $t_1,\dots, t_r$. Indeed, $V_\bullet \cM$ is exhaustive, so the right hand side is the direct limit as $j\to \infty$ of $A_j(\cM)$, and the direct limit is exact. By Theorem~\ref{thm:alphaless0}, the quotient map 
    \[
     A_0(\cM) \to B_Z(\cM)
    \]
    is a quasi-isomorphism.

The statement about $C(\cM)$ just follows from applying Proposition~\ref{prop:cvvc} to $T_ZX\rightarrow Z$ and Theorem~\ref{thm:gr}. Indeed, $\Sp(\cM)$ is monodromic on $T_Z X$, and it is not hard to show that $\sigma^*(\Sp(\cM)) = i^*(\cM)$, where $\sigma: Z\to T_Z X$ is the zero section of the normal bundle.
\end{proof}

\begin{prop}[{\cite[Proposition 10.4]{cvvc}}] \label{prop:cvvc} For a monodromic $\cD_E$-module $\cM$, there are quasi-isomorphisms 
\[
p_+\cM \simeq i^*\cM, \quad p_\dagger\cM \simeq i^!\cM
\]
where $p:E\rightarrow Z$ is a vector bundle and $i: Z\rightarrow E$ is the zero section.
\end{prop}

\begin{rmk}\label{rmk:cvvcMHM} Proposition \ref{prop:cvvc} also holds true if $M$ is a monodromic mixed Hodge module on $E$. The proof of the first claim uses the exact triangle
\[ j_!j^*\cM \to \cM \to i_* i^*\cM \xrightarrow[]{+1},\]
which holds true in the category of mixed Hodge modules, too. Then, applying $p_*$ to this triangle, it suffices to show that
\[ p_* j_! j^!M = 0,\]
but in the proof, Ginzburg shows the underlying $\cD$-module is 0, so the mixed Hodge module must necessarily be 0, as well. The second claim is related to the first by duality.
\end{rmk}

\begin{rmk}\label{}
Using the previous theorem, we can rephrase the results of Lemma \ref{strictsuppgen} and Lemma \ref{quotob} respectively as $\cH^0 i^! \cM = 0$ if and only if $\text{Hom}(i_+ \cN,\cM) = 0$ for all $\cN$ supported on $Z$, and $\cH^0 i^* \cM = 0$ if and only if $\text{Hom}(\cM,i_+\cN) = 0$ for all $\cN$ supported on $Z$.

We can describe the vanishing of other cohomologies in terms of Ext groups, similar to the characterization of vanishing of local cohomology for $\cO$-modules. Specifically, the result is
\[ \cH^{-j} i^* \cM = 0 \text{ for all } 0 \leq j \leq k \iff \text{Ext}^j(\cM,i_+\cN) = 0 \text{ for all } \cN \text{ supported on }Z, 0\leq j \leq k \]
\[ \cH^j i^! \cM = 0 \text{ for all } 0 \leq j \leq k \iff \text{Ext}^j(i_+\cN,\cM)=0 \text{ for all } \cN \text{ supported on }Z, 0\leq j \leq k.\]
The proofs of these are not hard, and we leave them to the reader. 
\end{rmk}

\section{Filtered Acyclicity of Koszul Complex} \label{sec:pacyclic}

Recall our setting: let $X\rightarrow \A^r$ be a smooth regular map of smooth varietes where $\A^r$ is the  affine space of dimension $r$ and let $Z$ be the fiber over the origin.  Suppose $(t_1, t_2,\dots,t_r)$ is a coordinate system on the $\A^r$ term and assume there exist global vector fields $\partial_1,\partial_2,\dots,\partial_r$ on $X$ dual to the one-forms $dt_1,dt_2,\dots,dt_r$. We give two different methods to prove Theorem~\ref{thm:acyclic}.

\subsection{First proof of the Theorem~\ref{thm:acyclic}}
We restate Theorem~\ref{thm:acyclic} in terms of \emph{right} $\cD$-modules: for any \emph{right} filtered regular holonomic $\cD_X$-module $\cM$ and rational number $\alpha$, define Koszul-type filtered complexes
\[
A_{\alpha}(\cM) =\left\{ (V_\alpha \cM,F) \xrightarrow[]{t} \bigoplus_{i=1}^r (V_{\alpha-1}\cM,F)\xrightarrow[]{t} \cdots \xrightarrow[]{t} (V_{\alpha+r}\cM,F)\right\}
\]
placed in degrees $0,1,\dots,r$,
\[
B_{\alpha}(\cM) =\left\{ (\gr_\alpha^V \cM,F) \xrightarrow[]{t} \bigoplus_{i=1}^r (\gr_{\alpha -1}^V\cM,F)\xrightarrow[]{t} \cdots \xrightarrow[]{t} (\gr_{\alpha-r}^V\cM,F)\right\}
\] 
as the quotient $A_{\alpha}/A_{>\alpha}$ and 
\[
C_{\alpha}(\cM)=\left\{ (\gr_{\alpha-r}^V \cM,F[r]) \xrightarrow[]{\de_t} \bigoplus_{i=1}^r (\gr_{\alpha-r+1}^V \cM,F[r-1]) \xrightarrow[]{\de_t} \cdots \xrightarrow[]{\de_t} (\gr_{\alpha}^V\cM,F)\right\}
\] 
in degrees $-r,-r+1,\dots,0$, where $V_\bullet\cM$ is the $V$-filtration along $Z$ and $F[i]_k=F_{k-i}$.

\begin{thm} \label{thm:acyc}
With the above notation, assume that $(\cM,F_\bullet\cM)$ is a filtered holonomic $\cD_X$-module underlying a mixed Hodge module. Then 
\begin{enumerate}
    \item the complex $F_\ell A_{\alpha}(\cM)$ is exact for $\alpha<0$;
    \item the complex $F_\ell C_{\alpha}(\cM)$ is exact for $\alpha>0 $.
\end{enumerate}
\end{thm}
\begin{proof}
By Lemma~\ref{thm:compatible}, we only need to prove the case when $(\cM,F)$ underlies a polarizable Hodge module. If the support of $\cM$ is contained in $Z$, then by Kashiwara's equivalence, there exists a Hodge module $(\cN,F_\bullet\cN)$ on $Z$ such that $(\cM,F_\bullet\cM)={i}_+(\cN,F_\bullet\cN)$. One can easily check that (see Example \ref{VFiltEg})
\[
F_\ell V_\alpha \cM =\left\{
    \begin{aligned}
    &  \sum_{i_1+i_2+\cdot +i_r\leq \alpha} F_{\ell-i_1-i_2-\cdots-i_r} \cN \partial^{i_1}_1\partial^{i_2}_2\cdots\partial^{i_r}_r, \quad &\alpha\geq 0; \\
    & 0, \quad &\alpha<0. \\
    \end{aligned}
    \right. 
\]

Thus, $\left(\gr^V_0\cM,F_\bullet\gr^V_0\cM\right)$ recovers the filtered $\cD_Z$-module $(\cN,F_\bullet\cN)$ and $\gr^V_\alpha\cM$ vanishes for $\alpha<0$. The statement $(a)$ is clear now. The statement $(b)$ follows from the fact that $\partial_1,\partial_2,\dots,\partial_r$ form a regular sequence on the polynomial ring $\C[\partial_1,\partial_2,\dots,\partial_r]$.

Now we are in the case that no submodule of $\cM$ is supported in $Z$. Let $\widehat{X}$ denote the blowup of $X$ along $Z$, with exceptional divisor $E$. Let $(\widehat\cM,F_\bullet\widehat\cM)$ be the minimal extension of $(\cM,F_\bullet\cM)|_{X\setminus Z}$ over $E$ on $\widehat X$. By the structure theorem of Hodge modules (see Theorem ~\ref{decompthm}), $(\widehat\cM,F_\bullet\widehat\cM)$ underlies a polarizable Hodge module. Then by the decomposition theorem of polarizable Hodge modules, the filtered holonomic $\cD_X$-module $(\cM,F_\bullet\cM)$ is a direct summand of $\cH^0\pi_+(\widehat\cM,F_\bullet\widehat\cM)$. Thus, it suffices to prove the theorem for $\cH^0\pi_+(\widehat\cM,F_\bullet\widehat\cM)$. Let $\pi:\widehat X \rightarrow X$ be the blow up of $X$ along $Z$ and $E=\pi^{-1}Z$ be the exceptional divisor. Consider the factorization $\pi=i_\pi\circ p$ and the Cartesian diagram
\[
\begin{tikzcd}
E \arrow{r}\arrow{d} & \widehat X\times Z \arrow{r}\arrow{d} & Z \arrow {d} \\
\widehat X \arrow{r}{i_\pi} \arrow[rr, bend right, "\pi"] & \widehat X\times X \arrow{r}{p} & X,
\end{tikzcd}
\]
where $i_\pi: \widehat X\rightarrow \widehat X\times X$ is the graph embedding and $p:\widehat X\times X\rightarrow X$ is the second projection. Denote by $\Gamma_\pi$ the graph of $\pi$. Since the problem is local on $X$, we can assume that $X$ is affine and that $(t_1,t_2,\dots,t_r)$ extends to a coordinate system $(t,s)=(t_1,t_2,\dots,t_r,s_1,s_2,\dots,s_{n-r})$ on $X$. Note that the blow-up is given by 
\[
\widehat X=\Proj_X \bigoplus_{i\geq 0} \cI^i_Z, \quad \text{where } \cI_Z \text{ is generated by } t_1,t_2,\dots,t_r.
\]
Let $u=[u_1:u_2:\dots:u_r]$ be the homogeneous coordinates on $\mathbf{P}^{r-1}$. Then $\widehat X$ is a subvariety of $\mathbf P^{r-1}_X$ defined by $u_i t_j-u_j t_i=0$ for any $1\leq i,j \leq r$. Denote also by $({x,y})=(x_1,x_2,\dots,x_r, y_1,\dots,y_{n-r})$ the parameter $(t,s)$ on $X$ so that 
\[
\pi ({u,t,s})= ({t,s})=({x,y}).
\] 
Define a subvariety 
\[
H=\{ (u,t,s,x,y)\in \widehat X\times X : u_ix_j-u_jx_i=0\text{ for any }1\leq i,j\leq r\}
\]
with codimension $r-1$ in $\widehat{X} \times X$. Since the graph $\Gamma_\pi$ is defined by equations ${t}={x}$ and ${s}={y}$, it is contained in $H$. Therefore, we can further factor the graph embedding $i_\pi=f\circ g$ to get a Cartesian diagram
\[
\begin{tikzcd}
E \arrow{r}\arrow{d} & \widehat X\times Z \arrow[r,equal]\arrow[d] & \widehat X\times Z \arrow{d} \\
\widehat X\arrow{r}{g}\arrow[rr, bend right,"i_\pi"] & H \arrow{r}{f}  & \widehat X\times X
\end{tikzcd}
\]
where $g:\widehat X\rightarrow H$ and $f:H\rightarrow \widehat X\times X$ are the natural embeddings. Note that $\widehat X\times Z$ is a hypersurface in $H$.

The claim is that the Koszul complex
\begin{equation}\label{eq:kosgrav}
F_\ell A_{\alpha}({i_\pi}_+ \widehat \cM) =\left\{F_\ell V_\alpha {i_\pi}_+ \widehat \cM \rightarrow \left(F_\ell V_{\alpha-1} {i_\pi}_+ \widehat \cM \right)^r\rightarrow\cdots\rightarrow F_\ell V_{\alpha-r}{i_\pi}_+\widehat \cM \right\}
\end{equation}
 is exact if $\alpha<0$ where $V_\bullet {i_\pi}_+\widehat \cM$ is the $V$-filtration of $\widehat \cM$ along $\widehat X\times Z $. The exactness of the complex ~\ref{eq:kosgrav} is local so without loss of generality, we restrict everything to the open subset $U\times X$ where $U$ is the open subset of $\widehat X$ defined $u_1\neq 0$. The blow-up over $U$ is given in coordinates by
\[
\pi: (t_1, u_2, u_3,\dots,u_r, s_1, s_2,\dots,s_{n-r}) \mapsto (t_1, t_1u_2,t_1u_2,\dots, t_1u_r, s_1,s_2,\dots,s_{n-r}).
\]
To give a concrete description of ${i_\pi}_+\widehat\cM$, we make the following local coordinate charge: 
\[
\begin{aligned}
w_i &=\left\{
    \begin{aligned}
    &  t_1 \quad \text{for } i=1 \\
    & u_i \quad \text{for } 2\leq i\leq r \\
    \end{aligned}
    \right. ,
    \quad 
&p_i = & s_i \quad \text{for } 1\leq i \leq n-r, \\
z_i &= \left\{
    \begin{aligned}
    &  x_1 \quad \text{for } i=1  \\
    &  x_i-u_i x_1 \quad \text{for } 2\leq i\leq r 
    \end{aligned}
    \right. ,
    \quad
&q_i = & y_i \quad \text{for } 1\leq i\leq n-r
\end{aligned}
\]
so that $z_2,z_3,\dots,z_{r}$ are the local defining equations of $H$. It follows from ${i_\pi}_+\widehat\cM=f_+g_+\widehat\cM$ that 
\[
{i_\pi}_+\widehat\cM=g_+\widehat\cM[\partial_{z_2},\partial_{z_3},\dots,\partial_{z_r}].
\]
In fact, a simple calculation using the the chain rule indicates that 
\[
\partial_{z_2}=\partial_{x_2}=\partial_2, \quad \partial_{z_3}=\partial_{x_3}=\partial_3, \quad \dots, \quad \partial_{z_r}=\partial_{x_r}=\partial_r.
\]  
Then $F_\ell V_\alpha {i_\pi}_+\widehat\cM$ can be written as
\begin{equation}\label{eq:fvgra}
\sum_{k\geq 0}\sum_{a_2+a_3+\cdots a_r=k} F_{\ell-k}V_{\alpha-k} g_+  \widehat\cM  \partial^{a_2}_2\partial^{a_3}_3\cdots \partial^{a_r}_r,
\end{equation}
for every $\alpha$ where $V_\bullet g_+\widehat\cM$ is the $V$-filtration along $\widehat X\times Z$. Notice that the morphism
\[
\begin{tikzcd}
 F_\ell V_\alpha g_+ \widehat\cM \arrow{r}{x_1} & F_\ell V_{\alpha-1} g_+ \widehat \cM
\end{tikzcd}
\]
is bijective when $\alpha<0$ because $V_\bullet g_+\widehat \cM$ is the $V$-filtration along $\widehat X\times Z$ defined by $\{x_1=0\}$ in $H$. We deduce that the morphism 
\[
x_1:  F_{\ell-k}V_{\alpha-k} g_+  \widehat\cM  \partial^{a_2}_2\partial^{a_3}_3\cdots \partial^{a_r}_r \rightarrow F_{\ell-k}V_{\alpha-k-1} g_+  \widehat\cM  \partial^{a_2}_2\partial^{a_3}_3 \cdots \partial^{a_r}_r
\] 
is also bijective for $\alpha<0$ and $k\geq 0$. It follows that the Koszul complex~\eqref{eq:kosgrav} is exact when $\alpha<0$. 

Similarly, the complex 
\begin{equation}\label{eq:kosgrac}
F_\ell C_{\alpha}({i_\pi}_+ \widehat \cM) = \left\{F_{\ell-r} \gr^V_{\alpha-r}{i_\pi}_+ \widehat \cM \rightarrow \left(F_{\ell-r+1}\gr^V_{\alpha-r+1}{i_\pi}_+ \widehat \cM \right)^r\rightarrow \cdots \rightarrow F_{\ell}\gr^V_{\alpha}{i_\pi}_+ \widehat \cM\right\}
\end{equation}
is exact for $\alpha >0$. By the expression~\eqref{eq:fvgra},
\[
F_\ell \gr^V_\alpha {i_\pi}_+\widehat\cM=\sum_{k\geq 0}\sum_{a_2+a_3+\cdots a_r=k} F_{\ell-k}\gr^V_{\alpha-k} g_+  \widehat\cM  \partial^{a_2}_2\partial^{a_3}_3\cdots \partial^{a_r}_r.
\]
Since for each $2\leq i\leq r$ the morphism
\[
\partial_i:  F_{\ell-k}\gr^V_{\alpha-k} g_+  \widehat\cM  \partial^{a_2}_2\partial^{a_3}_3\cdots \partial^{a_r}_r \rightarrow F_{\ell-k}\gr^V_{\alpha-k} g_+  \widehat\cM  \partial^{a_2}_2\partial^{a_3}_3\cdots \partial^{a_i+1}_i \cdots \partial^{a_r}_r
\]
is bijective, the complex~\eqref{eq:kosgrac} is quasi-isomorphic to,
\[
\begin{tikzcd}
 \{F_{\ell-1} \gr^V_{\alpha-1} g_+ \widehat\cM \arrow{r}{\partial_1} & F_\ell \gr^V_{\alpha} g_+ \widehat \cM\}, \quad \text{placed in degrees }r-1, r.
\end{tikzcd}
\]
which is exact for $\alpha>0$ also because again $V_\bullet g_+\widehat \cM$ is the $V$-filtration along the hypersurface $\widehat X\times Z\subset H$.

It remains to prove the exactness of~\eqref{eq:kosgrav} and~\eqref{eq:kosgrac} are invariant under higher direct image of $p$. This is Theorem~\ref{thm:bistrict} below. Applying Theorem~\ref{thm:bistrict} to~\eqref{eq:kosgrav} gives us that the Koszul complex 
\[
\begin{aligned}
F_\ell &A_{\alpha}(\cH^k p_+{i_\pi}_+\widehat \cM) \\
=& \left\{F_\ell V_\alpha \cH^k p_+ {i_\pi}_+ \widehat \cM \rightarrow \left(F_\ell V_{\alpha-1}  \cH^k p_+ {i_\pi}_+ \widehat \cM\right)^r\rightarrow\cdots\rightarrow F_\ell V_{\alpha-r} \cH^k p_+ {i_\pi}_+ \widehat \cM \right\}
\end{aligned}
\]
is exact for $\alpha<0$ and every $k$ where $V_\bullet \cH^k p_+ {i_\pi}_+ \widehat \cM$ is the $V$-filtration along $Z$. Due to 
\[
\cH^k p_+{i_\pi}_+=\cH^k \pi_+,
\] 
we have finished the proof of the first statement in Theorem~\ref{thm:acyclic}. The second statement follows similarly and we leave it to the readers.
\end{proof}

\begin{thm}\label{thm:bistrict}
Let $X$ be a nonsingular quasi-projective variety and $Y$ be an affine space with $Z$ an affine subspace defined by $x_1,x_2,\dots,x_r$.  Let $(\cM,F)$ be a filtered holonomic $\cD_{X\times Y}$-module underlying a polarizable Hodge module. Suppose that the second projection $p:X\times Y\rightarrow Y$ is projective on the support of $\cM$. Let $V_\bullet\cM$ be the $V$-filtration along $p^{-1}(Z)$. Let $V_\bullet \cH^k {p}_+\cM$ be the $V$-filtration along $Z$ for every $k$. 
\begin{enumerate}
    \item If the complex 
    \begin{equation} \label{eq:kosprv}
    F_\ell A_{\alpha}(\cM)=\left\{ F_\ell V_\alpha  \cM \rightarrow (F_\ell V_{\alpha-1} \cM)^r\rightarrow\cdots\rightarrow F_\ell V_{\alpha-r} \cM \right\}
    \end{equation}
     is exact for some $\alpha$, then the complex $ F_\ell A_{\alpha}(\cH^k {p}_+ \cM)$
     is also exact for every $k$.

    \item Similarly, if the Koszul complex
    \begin{equation}\label{eq:kosprc}
     F_\ell C_{\alpha}(\cM)=\left\{F_{\ell-r} \gr^V_{\alpha-r}  \cM \rightarrow \left(F_{\ell-r+1}\gr^V_{\alpha-r+1} \cM\right)^r\rightarrow \cdots \rightarrow F_{\ell}\gr^V_{\alpha} \cM \right\}
    \end{equation}
     is exact for some $\alpha$, then the complex $F_\ell C_{\alpha}( \cH^k p_+\cM)$  is exact for every $k$.
    \end{enumerate}
\end{thm}

\begin{proof}
Because of the bistrictness proved in~\cite{BMS} on the complex $p_+\left(\cM, V_\bullet, F_\bullet\right)=$
\[
\left(\bR p_* \left(\cM\otimes \bigwedge^{-\star} \mathscr{T}_{X\times Y/Y} \right), \bR p_*  \left(V_\bullet \cM\otimes  \bigwedge^{-\star} \mathscr{T}_{X\times Y/Y} \right), \bR p_* \left( F_{\bullet+\star}\cM\otimes  \bigwedge^{-\star}\mathscr{T}_{X\times Y/Y} \right) \right),
\]
 we know that the $k$-th cohomology of $\cH^k F_\ell V_\alpha p_+ \cM= \bR^k p_* \left( F_{\ell+\star} V_\alpha \cM\otimes \bigwedge^{-\star}\mathscr{T}_{X\times Y/Y}\right)$ is canonically isomorphic to $F_\ell V_\alpha \cH^k p_+\cM$. It follows from the Hard Lefschetz theorem on the direct image of polarizable Hodge modules (see part (b) of Theorem \ref{dirimthm}) that the morphism
\[
\left(2\pi\sqrt{-1}L \right)^k : F_\ell V_\alpha \cH^{-k} p_+\cM \rightarrow F_{\ell-k} V_\alpha \cH^k p_+\cM.
\]
is an isomorphism induced by the Lefschetz operator $L=\omega\wedge$ of a hyperplane class $\omega$ on $X$. Therefore, we have the decomposition 
\[
F_\ell V_\alpha p_+ \cM \simeq \bigoplus_{k\in \Z} F_\ell V_\alpha \cH^k p_+\cM[-k]
\]
in the bounded derived category ${D}^b_{\mathrm{coh}}(Y,\cO_Y)$ of $Y$. If we apply $p_+$ on~\eqref{eq:kosprv}, by the above decomposition, we obtain
\[
F_\ell  p_+  A_{\alpha}(\cM) \simeq   \bigoplus_{k\in\Z } F_\ell A_{\alpha}(\cH^k p_+\cM)[-k]
\]
in $D^b_{\mathrm{coh}}(Y,\cO_Y)$. But by our assumption in $(a)$, the complex $F_\ell p_+ A_{\alpha}(\cM)$ is exact. It follows that each summand 
\[
F_\ell A_{\alpha}(\cH^k p_+\cM)= \left\{  F_\ell V_\alpha \cH^k p_+ \cM \rightarrow (F_\ell V_{\alpha-1}\cH^k p_+ \cM)^r\rightarrow\cdots\rightarrow F_\ell V_{\alpha-r}\cH^k p_+ \cM  \right\}
\]
in the decomposition is exact. We have thus proved $(a)$. 

The proof of $(b)$ is similar. Since we still have the isomorphism from the Hard Lefschetz theorem
\[
\left(2\pi\sqrt{-1}L \right)^k : F_\ell \gr^V_\alpha \cH^{-k} p_+\cM \rightarrow F_{\ell-k} \gr^V_\alpha \cH^k  p_+\cM,
\]
we get a decomposition
\[
\begin{aligned}
 p_+  F_\ell C_{\alpha}(\cM) \simeq   \bigoplus_{k\in\Z } F_\ell C_{\alpha}(\cH^k p_+\cM)[-k]
 \end{aligned}
\]
in $D^b_{\mathrm{coh}}(Y,\cO_Y)$. The remaining goes like in $(a)$ and is left to the readers.
\end{proof}

\begin{rmk}
One can bypass the decomposition theorem in the above proof by the argument in Theorem~\ref{thm:compare} and the double complexes~\eqref{eq:BK} and~\eqref{eq:CK}
\end{rmk}

\subsection{An alternative proof of $C^\alpha$ acyclicity}
In this subsection, we describe an alternative proof for Theorem \ref{thm:acyclic}. The main idea is to use Saito's Direct Image Theorem \ref{dirimthm}, specifically in the setting of Remark \ref{smoothproj}. Throughout, we use \emph{left} $\cD$-modules, and the decreasing $V$-filtration.

Saito's theorem would be useful if the projection $Z\times \A^r \to Z$ were projective, but it is clearly not. Hence, we compactify $j: Z\times \A^r \to Z\times \P^r$, where we have homogeneous coordinates $[t_0: \dots : t_r]$ on $\P^r$ and think of $Z\times \A^r = U_0 = \{t_0 \neq 0\}$. In this way, $z_i = \frac{t_i}{t_0}$. Let $\rho: Z \times \P^r \to Z$ be the projection. Then Saito's theorem applies to the direct image of a mixed Hodge module along $\rho$.

We set up notation for the other standard affine open subsets of $Z\times \P^r$. Let $j_i: U_i = \{t_i \neq 0\} \hookrightarrow Z \times \P^r$ be the inclusion. Let $U_{i0} = U_i \cap U_0$, with inclusion maps $J_{i0}: U_{i0} \hookrightarrow U_0$ and $j_{i0}: U_{i0} \hookrightarrow U_i$.

\subsubsection{Computation of $V$-filtration for extension of localization}
Assume throughout that $X$ (hence the closed subscheme $Z\subseteq X$) is affine, so $\{U_i\}_{i=0}^r$ gives an affine open cover of $Z \times \P^r$.

For any $\cM$ a bifiltered $\cD$-module underlying a mixed Hodge module on $X$, we get $\cM' = \Sp(\cM) = \bigoplus_{\chi \in \Q} \gr^\chi \cM$ a bifiltered $\cD$-module underlying a monodromic mixed Hodge module on $T_Z X = Z \times \A^r$. The motivation behind trying to use Saito's strictness theorem is that, if $\pi: Z \times \A^r \to Z$ is the projection, then $\pi_+(\cM') \cong \bigoplus_{\chi \in \Q} C^\chi(\cM)$ as complexes of $\cD$-modules. Hence, because we know by Theorem \ref{thm:acyclic} that $C^\chi(\cM)$ is acyclic for $\chi \neq 0$, we  use the injectivity from Saito's direct image theorem (see Remark \ref{smoothproj}) to conclude some vanishing for the cohomology of the filtered complex as well.

The point of this subsection is to compute the Hodge filtration on the mixed Hodge module $j_+(\Sp(M))$. For this, we will make use of the affine open cover $\{U_i\}_{i=0}^r$ of $Z\times \P^r$ and the following remark (which we already made use of in computing the Hodge filtration of $\Sp(M)$ in Lemma \ref{SpHFilt}):

\begin{rmk} \label{HodgeFiltComplement} Let $V = \{f \neq 0\} \subseteq W$ be the inclusion of the complement of a smooth hypersurface, which is defined by the nonvanishing of some global function $f\in \cO_W(W)$. If $(\cM,F_\bullet \cM)$ is a filtered $\cD_V$-module underlying a mixed Hodge module $M$ on $V$, then formula (3.2.3.2) of \cite{Saito-MHP} allows us to compute the Hodge filtration on the $\cD_W$-module underlying $j_+(M)$.

First of all, the underlying $\cD_W$-module is simply $j_+(\cM)$, which is the $\cO$-module pushforward of the sheaf $\cM$. Then the filtration is given by
\[ F_p j_+(\cM) = \sum_{q\geq 0} \de_f^q\left(V^0 j_+(\cM) \cap j_*(F_{p-q}\cM)\right),\]
where $\de_f \in \cT_W(W)$ is a globally defined vector field such that $[\de_f ,f ] = 1$ and $V^0 j_+(\cM)$ is the $V$-filtration along the smooth hypersurface $\{f=0\}$.
\end{rmk}

We are not exactly in this situation, however, because there is no global function $f$ for which $U_0 = \{f\neq 0\}$. Indeed, $U_0$ is the non-vanishing locus of a \emph{rational} function $t_0$ on $Z\times \P^r$. This causes no problems for our computation, as we will in any case be using this affine open cover to compute higher direct images.

For $1\leq i \leq r$ fixed, we consider $U_i = Z \times \A^r$, but the coordinates we choose on $\A^r$ are given by $z_i^{-1},z_1z_i^{-1},\dots, \widehat{z_i z_i^{-1}}, \dots, z_rz_i^{-1}$ (in terms of the compactification, these are the standard affine coordinates $\frac{t_0}{t_i}, \frac{t_1}{t_i},\dots, \frac{t_r}{t_i}$). To be clear which coordinate system we consider on each copy of $\A^r$, we will denote this copy by $\A^r_i$.

Note that $J_{i0}: U_{i0} \subseteq U_0$ corresponds to the subset $\{z_i \neq 0\}$. Similarly, the inclusion $J_i: U_{i0} \subseteq U_i$ is the inclusion of the non-vanishing locus of the globally defined function $z_i^{-1}$. Now, we have the isomorphism of mixed Hodge modules $j_i^{-1}(j_+(M')) \cong J_{i+}(J_{i0}^{-1}M')$, and so the Hodge filtration on $j_i^*(j_+(\cM'))$ can be computed using Remark \ref{HodgeFiltComplement}.

As $\cM'$ is a quasi-coherent $\cO_{U_0}$-module, we have the identification $J_{i0}^*(\cM') \cong \cM'_{z_i}$, and the Hodge filtration is given by $F_p J_{0i}^*(\cM') = (F_p \cM')_{z_i}$. Let $V^\bullet J_{i*}(\cM'_{z_i})$ be the $V$-filtration along $z_i^{-1}$, then Remark \ref{HodgeFiltComplement} tells us
\begin{equation} \label{SaitoFiltFormula} F_p J_{i*}(\cM'_{z_i}) = \sum_{q\geq 0} \de_{z_i^{-1}}^q(V^0\cap J_{i*}((F_{p-q}\cM')_{z_i})).\end{equation}

Thus, it will be worthwhile to study this $V$-filtration. For this, note that by the change of coordinates formula, using $(z_1,\dots, z_r)$ and $(z_i^{-1},z_1z_i^{-1},\dots, z_rz_i^{-1})$ on $U_{i0}$, we have
\[ 
\begin{aligned}
\de_{z_j z_i^{-1}} & = z_i \de_{z_j}, \quad \de_{z_i^{-1}} & = - z_i \theta,
\end{aligned}
\]
where $\theta = \sum_{j=1}^r z_j \de_{z_j}$. 

Given any $\frac{m}{z_i^k} \in \cM'_{z_i}$, we see that
\[ \de_{z_i^{-1}}\left(\frac{m}{z_i^k}\right) = \frac{(k-\theta)m}{z_i^{k-1}},\]
so
\[ (z_i^{-1} \de_{z_i^{-1}} - \lambda +1)\left(\frac{m}{z_i^k}\right) = \frac{-(\theta+\lambda - k -1)(m)}{z_i^k}.\]

In this way, we see that the module $\cM'_{z_i}$ (and its pushforward $J_{i*}(\cM'_{z_i})$) is monodromic along the single coordinate $z_i^{-1}$.

The $\lambda$th monodromic piece is given by $\sum_{k\in \Z} \frac{\gr_V^{r+k+1-\lambda}\cM}{z_i^k}$. Hence, by \ref{monoVFilt}, we know
\[ V^{\bullet} J_{i*}(\cM_{z_i}) = \bigoplus_{\lambda \geq \bullet} \left(\sum_{k\in \Z} \frac{\gr_V^{r+k+1-\lambda}\cM}{z_i^k}\right).\]
In particular,
\begin{equation} \label{0VFilt} V^0 J_{i*}(\cM_{z_i}) = \bigoplus_{\lambda \geq 0} \left(\sum_{k \in \Z} \frac{\gr_V^{r+k+1-\lambda}\cM}{z_i^k}\right) \end{equation}
This gives us the following
\begin{equation} \label{HodgeFilt} m \in F_p \gr_V^\chi\cM, \quad \chi \leq r+1 \implies \frac{m}{1} \in F_p J_{i*}(\cM_{z_i}).\end{equation}

Indeed, $\chi \leq r+1$ implies $\chi = r+0+1-\lambda$ for some $\lambda \geq 0$, and so $\frac{m}{1} \in V^0 J_{i*}(\cM'_{z_i})$, then the claim follows from Formula \ref{SaitoFiltFormula}, with $q=0$.

\subsubsection{Proof of the Theorem}
Recall that the strategy is to use Theorem \ref{thm:gr} and Saito's strictness result Theorem \ref{dirimthm} to obtain the desired vanishing. The idea is that the Koszul-like complex naturally arises as the underlying complex of $\cD_Z$-modules when applying $\pi_+$ to a $\cD_{T_ZX}$-module. If it were the case that Saito's strictness theorem told us that
\[ F_p(\DR_{T_ZX/Z}(\Sp(\cM)) =\left\{ F_p \Sp(\cM) \xrightarrow[]{\de_z} F_{p+1} \Sp(\cM) \otimes \Omega_{\A^r}^1\xrightarrow[]{\de_z} \dots \xrightarrow[]{\de_z} F_{p+r} \Sp(\cM) \otimes \Omega_{\A^r}^r \right\}\]
was a \emph{strict} complex, then we would be done. This is because if a complex of filtered modules is exact (which we know to be true for our complex by Theorem \ref{thm:gr}) and strict, then the complex must be filtered acyclic.

This is too good to be true, though, because Theorem \ref{thm:gr} gives us acyclicity for all $\chi \neq 0$, whereas in the case $r=1$, we must restrict to $\chi <0$ for the filtered acyclicity to hold as in the definition of quasi-unipotent and regular.

The reason Saito's result does not apply to the map $\pi$ is that it is not projective. However, using the natural compactification $j: Z\times \A^r \to Z\times \P^r$ introduced in this section, and the map $\rho: Z\times \P^r \to Z$, we can still make use of Saito's result. Indeed, $\pi = \rho \circ j$, so by \cite{Saito-MHM} (4.3.2) we have $\pi_+ = \rho_+ j_+$. In particular, the complex of $\cD$-modules underlying $\rho_+(j_+(\Sp(\cM))$ is $\DR_{T_ZX/Z}(\Sp(\cM))$. To make notation less cumbersome, let $\Sp(\cM) = : \cM'$.

As $j$ is the inclusion of the complement of a divisor, $j_+$ sends mixed Hodge modules on $Z\times \A^r$ to mixed Hodge modules on $Z\times \P^r$. To compute $\rho_+(j_+(\cM'))$, we use the relative de Rham complex of $Z\times \P^r$ over $Z$ for the module $j_+
(\cM')$, see \cite{HTT} Proposition 1.5.28. This is a complex in degrees $-r,\dots, 0$ denoted by
$\DR_{Z\times \P^r/Z}(j_+(\cM'))$, explicitly, it is
\[ j_*(\cM') \otimes_{\cO_\P} q^*(\cO_\P) \to j_*(\cM') \otimes q^*(\Omega_\P^1) \to \dots \to j_*(\cM') \otimes_{\cO} q^*(\omega_{\P^r}),\]
where $q:Z\times \P^r \to \P^r$ is the other projection. This is a filtered complex, with $p$th filtered piece given by
\[ F_p j_+(\cM')\otimes_{\cO} q^*(\cO_{\P^r}) \to F_{p+1} j_+(\cM') \otimes_{\cO} q^*(\Omega_{\P^r}^1) \to \dots \to F_{p+r}j_+(\cM') \otimes_{\cO} q^*(\omega_{\P^r}) .\]

Saito's result on the strictness of the direct image for a projective morphism tells us that we have the following identification
\[ F_\bullet \cH^k \pi_+(\cM') = R^k \rho_*(F_\bullet \DR_{Z\times \P^r/Z}(j_+\cM')) \hookrightarrow R^k \rho_*(\DR_{Z\times \P^r/Z}(j_+(\cM')) = \cH^k \pi_+(\cM').\]

To compute these higher direct image sheaves, we use the standard open cover of $Z\times \P^r$ introduced in this section to construct the \v Cech complex. Namely, fix coordinates $[t_0:\dots :t_r]$ on $\P^r$ and let $j: Z\times \A^r \to Z\times \P^r$ be the inclusion of $\{t_0\neq 0\}$. Also, denote $j_i: \{t_i\neq 0\} \to Z\times \P^r$ the inclusion of the other standard open subsets. We have an isomorphism $U_i : = \{t_i \neq 0\} \cong Z\times \A^r$, with coordinates $(z_i^{-1},z_1z_i^{-1},\dots, z_rz_i^{-1})$ on $\A^r$.

For any $i_0 < \dots < i_q$, let $U_{i_0 i_1\dots i_q} = U_{i_0}\cap U_{i_1}\cap \dots \cap U_{i_q}$. The ordering is important for keeping track of signs in the \v Cech complex.

For any sheaf $\cF$ on $Z\times \P^r$, let $\cC^q(\cF)= \prod_{i_0 < \dots < i_q} \pi_{i_0\dots i_q,*}(\cF\mid_{U_{i_0i_1\dots i_q}})$. This forms a chain complex of sheaves by the differential defined (on local sections) as
\[ \delta: \cC^q(\cF) \to \cC^{q+1}(\cF), \delta(\alpha)_{i_0\dots i_{q+1}} = \sum_{i=0}^{q+1} (-1)^i \alpha_{i_0\dots \widehat{i_i} \dots i_{q+1}}.\]

Denote $\pi_{i_0\dots i_q,*}(\cF \mid_{U_{i_0i_1\dots i_q}})$ by $\cC^q_{i_0,\dots,i_q}(\cF)$.

If we have a bounded below complex of sheaves $(\cF^\bullet,d)$, then we form the double complex $\cC^\bullet(\cF^\bullet)$ and take its associated total complex, which we call the \emph{\v Cech complex} for $\cF^\bullet$. The $n$th term of this complex is 
\[ \prod_{p+q = n} \cC^q(\cF^p),\]
and the differential is $\delta + (-1)^{q} d$. Now, because $U_{i_0\dots i_q} \to Z$ is affine for all $i_0 < \dots < i_q$, we know that, if each $\cF^k$ is quasi-coherent on $Z\times \P^r$, the $i$th cohomology of the \v Cech complex computes $R^i \rho_*(\cF^\bullet)$. By abuse of notation, we will drop the $\pi_*$ with the understanding that all sheaves are on $Z$.

Let $T^\bullet$ be the \v Cech complex associated to $\DR_{Z\times \P^r/Z}(j_+(\cM'))$, and let $T_p^\bullet$ be the \v Cech complex associated to $F_p \DR_{Z\times \P^r/Z}(j_+(\cM'))$. Using our notation,
\[ \cC^0_0(F_p \DR_{Z\times \P^r/Z}(j_+(\cM')) = F_p \DR_{U_0/Z}(\cM'),\]
and for $0 <i\leq r$,
\[ 
\begin{aligned}
    &\cC^0_i(F_p \DR_{Z\times \P^r/Z}(j_+(\cM')) = j_i^{-1}(F_p \DR_{Z\times \P^r/Z}(j_+(\cM')) \\
    &= \left\{ F_p J_{i*}(\cM_{z_i}) \otimes_{\cO} \cO_{\A^r_i} \to F_{p+1} J_{i*}(\cM_{z_i}) \otimes_{\cO} \Omega_{\A^r_i}^1 \to \dots \to F_{p+r} J_{i*}(\cM_{z_i}) \otimes_{\cO} \omega_{\A^r_i} \right\}.
\end{aligned}
\]

Saito's strictness result Theorem \ref{dirimthm} tells us that the map $\cH^k(T_p^\bullet) \to \cH^k(T^\bullet)$ is injective for all $k\in \Z$. Also, by the functoriality of pushforward of mixed Hodge modules, $\pi_+ = \rho_+ \circ j_+$, and so we have a quasi-isomorphism of complexes of $\cD_Z$-modules
\[ \DR_{U_0/Z}(\cM') \cong T^\bullet.\]

We would like to make this explicit using the \v Cech complex. Note that  any map of complexes $\psi: \DR_{U_0/Z}(\cM') \to T^\bullet$ cannot land solely in $\cC^0_0(DR_{Z\times \P^r/Z}(\cM'))$. Indeed, the way the differential of the \v Cech complex is defined, whatever the target element is would have to restrict to 0 on $U_{i0}$, which is not true in general.

We can work around this by introducing for any $\mu \in \cM' \otimes \Omega_{\A^r_0}^k$ elements $\mu|_{U_i} \in J_{i*}(\cM'_{z_i}) \otimes \Omega_{\A^r_i}^k$.

\begin{lem} \label{degreeshift} For any $\mu \in \cM' \otimes \Omega^k_{\A^r_0}$, there is a well-defined $\mu_i \in \cM'_{z_i} \otimes \Omega^k_{\A^r_i}$ such that
\[ \mu |_{U_{i0}} = (\mu_i)|_{U_{i0}}.\]
\end{lem}
\begin{proof} We prove the claim for a simple tensor $\mu = m\otimes dz_I$, where $|I| = k$ is a subset of $\{1,\dots, r\}$. Then $\omega = dz_I$ is a holomorphic $k$-form on $\A^r_0$, and we can restrict to the subset $\{z_i \neq 0\}$, which has two systems of coordinates: $(z_1,\dots, z_r)$ and $(z_i^{-1},z_1z_i^{-1},\dots, z_r z_i^{-1})$.

The differential forms change as
\[ dz_i = d((z_i^{-1})^{-1}) = -z_i^2 d(z_i^{-1}),\]
\[ dz_j = d((z_jz_i^{-1})(z_i^{-1})^{-1})) = z_i d(z_j z_i^{-1}) - z_jz_i d(z_i^{-1}).\]

Hence, we can write $\omega|_{U_{i0}} = z_i^{k+1} (\omega_i)|_{U_{i0}}$, where $\omega_i$ is a holomorphic $k$-form on $\A^r_i$.

Then set $\mu' : = \frac{z_i^{k+1}m}{1} \otimes \omega_i$. It is clear that this satisfies the desired property.
\end{proof}

Now, using this we can define a map $\psi^k: \cM' \otimes \Omega_{\A^r}^k \to \cC^0(j_+(\cM')\otimes \Omega_{\P^r}^k)\subseteq T^{-(r-k)}$ on simple tensors by $m \otimes \omega \mapsto (m\otimes \omega, (m\otimes \omega)|_{U_i})_{i=1}^r$. By definition of the differential in the \v Cech complex, this gives a morphism of complexes $\psi: \DR_{U_0/Z}(\cM') \to T^\bullet$.

Now, fix $p\in \Z, \lambda \in \Q$ and define $\cB_{p,\lambda} = \bigoplus_{\chi \leq \lambda} F_p \gr_V^{\chi}\cM$. We can consider the Koszul-like complex $\cB_{p,\lambda}^\bullet$, defined as
\[ \cB_{p,\lambda} \xrightarrow[]{\de_t} \cB_{p+1,\lambda-1}\otimes \Omega_{\A^r} \xrightarrow[]{\de_t} \dots \xrightarrow[]{\de_t} \cB_{p+r,\lambda-r}\otimes \omega_{\A^r}.\]

Now, we can consider the commutative square of complexes
\[ \begin{tikzcd} \cB_{p,<r}^\bullet \ar[r] \ar[d,"\psi_p"] & \DR_{U_0/Z}(\cM') \ar[d,"\psi"]\\ T_p^\bullet \ar[r] & T^\bullet
\end{tikzcd}.\]

Note that by the proof of Lemma \ref{degreeshift}, we know $\psi_p$ does actually land in $T_p^\bullet$, by Formula \eqref{HodgeFilt}. Indeed, if we start with $m\otimes \omega$ with $m \in F_p \gr^{\chi}\cM$ with $\chi < r-k$, then $z_i^{k+1} m \in F_p \gr^{\chi+k+1} \cM$ and $\chi +k +1 < r+1$. This tells us that $\frac{z_i^{k+1}m}{1} \in F_p J_{i*}(\cM_{z_i}')$, so $\frac{z_i^{k+1}m}{1} \otimes \omega_i \in T^k_p$, as desired.

The obvious map $T_p^\bullet \to \cB_{p,<r}^\bullet$ mapping to the $\cC_0^k$ part of the \v Cech tuple, and then projecting to the $\cB_{p+k,<r-k}\otimes \Omega_{\A^r}^k$ part, gives a splitting of the left hand map. This splitting is preserved under taking cohomology, so the map $\cH^k(\cB_{p,0}^\bullet) \to \cH^k(T_p^\bullet)$ is injective for all $k$. Also, by strictness, the map $\cH^k(T_p^\bullet) \to \cH^k(T^\bullet)$ is injective.

Hence, the map $\cH^k(\cB_{p,<r}) \to \cH^k(\DR_{U_0/Z}(\cM'))$ is injective. By decomposing along the $\bigoplus_{\chi \in \Q}$ decomposition, this gives that the inclusion of Koszul-like complexes
\[ \begin{tikzcd} F_p \gr_V^{\alpha - \ell +r -1}\cM \ar[r,"\de_t"] \ar[d] & \bigoplus_{i=1}^r F_{p+1}\gr_V^{\alpha - (\ell+1)+r-1}\cM \ar[r,"\de_t"] \ar[d] & \dots \ar[r,"\de_t"]& F_{p+r}\gr_V^{\alpha - (\ell+r) +r -1}\cM \ar[d] \\
\gr_V^{\alpha - \ell + r -1}\cM \ar[r,"\de_t"] & \bigoplus_{i=1}^r \gr_V^{\alpha -(\ell+1)+r-1}\cM \ar[r,"\de_t"] & \dots \ar[r, "\de_t"] & \gr_V^{\alpha - (\ell+r)+r-1}\cM
\end{tikzcd}\]
induces injections on cohomology.

By Theorem \ref{thm:gr}, this completes the proof.

Note that, this also proves the filtered acyclicity of $B^\alpha(\cM)$ for $\alpha >0$ by using the computation of the Hodge filtration on the Fourier transform of $\Sp(\cM)$ (which is a monodromic mixed Hodge module) in the final section of this paper. The result on the Hodge filtration only uses the filtered acyclicity of $C^\alpha(\cM)$ and the formula for $i^*$ in terms of the $V$-filtration along $Z$.

\section{Applications of Theorem~\ref{thm:acyclic}} \label{sec:applic}

Here we collect some applications of the theorem on filtered acyclicity of the Koszul-like complexes coming from Theorem \ref{thm:acyclic}.

The first of these applications is that we can give a Hodge-theoretic proof of Skoda's theorem.
\begin{cor}[Skoda]
Let $\mathfrak a$ be a coherent ideal of $\cO_X$ generated by $r$ many elements and $\cJ(X,\mathfrak a^c)$ be the multiplier ideal of exponent $c$. Then we have 
\[
\cJ(X,\mathfrak a^{c})=\mathfrak a \cJ(X,\mathfrak a^{c-1})
\]
for any $c\geq r$.
\end{cor}
\begin{proof}
Let $f_1,f_2,\dots,f_r$ be the generators of $\mathfrak a$ and let $\iota: X\rightarrow X\times \A^r$ be the graph of $f_1,\dots,f_r$. Then by~\cite[Theorem 1]{BMS}, the $\cO_X$-module $F_{r}V^{c+\varepsilon}\iota_+ \cO_X$ is the multiplier ideal $\mathcal J(X,{\mathfrak a}^c)$ for $\varepsilon>0$ sufficiently small where $V^\bullet\iota_+ \cO_X$ is the $V$-filtration along $X\times \{0\}$. Note that we use a different convention from~\cite{BMS}. The statement follows from the exactness of $A^{c-r+\varepsilon}(\iota_+ \cO_X)$ when $c\geq r$ by Theorem~\ref{thm:acyclic}.
\end{proof}

\begin{rmk}
One can also prove that $\cJ(X,\mathfrak a^{c})=\mathfrak a \cJ(X,\mathfrak a^{c-1})$ for $c\geq \dim X$ by making use of a reduction $\mathfrak c$ of $\mathfrak a$ as in~\cite{Laz}*{Corollary 9.6.17, Example 9.6.19}. More precisely, we can find a subideal $\mathfrak c$ of $\mathfrak a$ locally generated by at most $\dim X$  many elements such that $\mathfrak c$ and $\mathfrak a$ have the same integral closure and therefore $\cJ(X,\mathfrak c^c)=\cJ(X,\mathfrak a^c)$.  
\end{rmk}

Next, we prove an analogue of ~\cite[Prop. 3.2.2, Rem. 3.2.3]{Saito-MHM} which gives information about the Hodge filtration for filtered $\cD$-modules which satisfy the filtered acyclicity of the Koszul-like complexes.

We will make use of the following useful criterion for when an element $m\in V^{>0}\cM$ lies in $F_p \cM$. As $X - Z = \bigcup_{i=1}^r \{z_i \neq 0\}$, this is an analogue of ~\cite[Formula (3.2.2.1)]{Saito-MHP}.

\begin{prop} \label{HodgeTest} Assume $m\in V^{>0}\cM$. Then $m \in F_p \cM$ if and only if for some $\ell \geq 0$, $z^\alpha m \in F_p \cM$ for all $|\alpha| = \ell$. In other words,
\[ 
F_p V^{>0}\cM = V^{>0}\cM \cap (j_*j^*F_p\cM),
\]
where the right hand side are those elements in $V^{>0}\cM$ which, under the canonical map $\cM \to \cH^0(j_*j^*\cM)$ land in $j_*j^* F_p \cM$. 

If, moreover, the map $(\gr^0_V\cM,F)\xrightarrow[]{z_i} \bigoplus_{i=1}^r (\gr_V^1 \cM,F)$ is strictly injective, then 
\[ F_pV^0\cM = V^0\cM \cap (j_*j^*F_p\cM).\]
\end{prop}
\begin{proof} The ``only if'' part is clear, as $z_i$ preserves the Hodge filtration.

For the converse, induce on $\ell$. The base case $\ell = 0$ is obvious. For $\ell >0$, assume $z^\alpha m \in F_p\cM$ for all $|\alpha| = \ell$, and inductively, that if $m'$ is any element such that $z^{\beta}m' \in F_q \cM$ for all $|\beta| = \ell -1$, then $m' \in F_q \cM$.

Well, fix $\beta$ with $|\beta| = \ell -1$. Then $|\beta+e_i| = \ell$, so we know $z_i (z^\beta m) \in F_p \cM$ for all $1\leq i \leq r$. Assume $z^\beta m \in F_q \cM$ for some $q>p$.

Look at the complex $\gr^F_q A_{>0}(\cM)$, which is acyclic by Theorem \ref{thm:acyclic}. Then $z^\beta m$ defines an element in the leftmost module of this complex which maps to 0, as $p < q$. Hence, by acyclicity, $z^\beta m = 0$ in $\gr_q^F$, so $z^\beta m \in F_{q-1}\cM$. Repeating in this way, we conclude that $z^\beta m \in F_p \cM$. As this is true for all $\beta$ with $|\beta| = \ell -1$, the inductive hypothesis implies $m \in F_p \cM$.
\end{proof}

Now, we give the analogue of ~\cite[Formula (3.2.2.2)]{Saito-MHP}, which gives a description of the Hodge filtration of a mixed Hodge module in terms of the $V$-filtration along $Z$.

\begin{prop} \label{HodgeFormula} Let $(\cM,F)$ be a filtered $\cD_X$-module underlying a mixed Hodge module on $X$. Let $V^\bullet \cM$ be the $V$-filtration along the smooth subvariety $Z$. Then for all $p\in \Z$,
\[ 
F_p \cM = \sum_{\alpha \in \N^r} \de_z^\alpha(F_{p-|\alpha|}V^{0}\cM).
\]
Moreover, if the map
\[ \bigoplus_{i=1}^r (\gr_V^1 \cM,F) \xrightarrow[]{\de_z} (\gr_V^0 \cM,F)\]
is strictly surjective, then
\[ F_p \cM = \sum_{\alpha \in \N^r} \de_z^\alpha(F_{p-|\alpha|} V^{>0}\cM).\]
\end{prop}
\begin{proof} We argue the first claim, the second one is proved in exactly the same way. Define a second filtration by
\[ F_p' \cM = \sum_{\alpha \in \N^r} \de_z^\alpha(F_{p-|\alpha|} V^0 \cM).\]

The claim is that $F_p'\cM = F_p \cM$. By definition, $F_p' \cM \subseteq F_p \cM$. Note that it is clear that $F_p' V^0 \cM = F_p V^0\cM$. Indeed, the left hand side is contained in the right, and by taking $\alpha = 0$ in the definition of $F_p' \cM$, we get the other containment.

Now, we prove that $F_p V^\chi \cM \subseteq F_p' \cM$ for all $p\in \Z$ and $\chi < 0$ by descending induction (which works because the $V$-filtration is discrete). Note that, by Theorem \ref{thm:acyclic}, we know 
\[ F_p V^\chi \cM = \sum_{i=1}^r \de_{z_i}(F_{p-1} V^{\chi+1} \cM) + F_p V^{>\chi}\cM,\]
because the rightmost cohomology vanishes in $F_p C_\chi(\cM)$.

Now, by definition of $F_p'\cM$, we know $\de_{z_i} F_p' \cM \subseteq F_{p+1}'\cM$.

Hence, if inductively we know $F_p V^{>\chi}\cM \subseteq F_p'\cM$ for all $p\in \Z$, we get $F_p V^\chi \cM \subseteq F_p' \cM$, proving the claim.
\end{proof}

\begin{rmk}
The previous lemmas are not used in the proof of Theorem \ref{thm:main}. As a consequence of this theorem, we can remove the adjective ``strictly" in the second claim of both Proposition \ref{HodgeTest} and Proposition \ref{HodgeFormula}, because the morphisms appearing in a mixed Hodge complex are strict with respect to the Hodge filtration (see the remarks after ~\cite[Def. 2.2]{MHC}).

As a first application, assume $\cN = \cH^0(j_*j^*(\cM))$ for $j: X -Z \to X$ the inclusion of the complement of $Z$. Then by Lemma \ref{subob}, we see that the map in the statement of Proposition \ref{HodgeTest} is automatically injective, and hence we get the formula
\[ F_p \cN = \sum_{\alpha \in \N^r} \de_z^\alpha(V^0\cN \cap j_*(j^*(F_{p-|\alpha|}\cM))).\]

Secondly, we get a formula for the Hodge filtration of any filtered $\cD$-module $(\cM,F)$ underlying a mixed Hodge module with strict support not contained in $Z$, as follows: by definition, the hypotheses of Propositions \ref{HodgeTest} and \ref{HodgeFormula} are satisfied, and so we have
\[ F_p \cM = \sum_{\alpha \in \N^r} \de_z^\alpha(V^{>0}\cM \cap j_*(j^*(F_{p-|\alpha|}\cM))\]
\end{rmk}

Finally, we study the Hodge filtration of monodromic mixed Hodge modules on $E = X\times \A^r$. The result is a generalization to $r>1$ of ~\cite[Thm. 2.2]{monoMHM}. We first record an interesting application of the fact from loc. cit. that $N = \bigoplus_{\chi \in \Q} (z \de_z - \chi +1) = 0$ on a pure monodromic Hodge module, when $r=1$.

\begin{cor} Let $\cM$ be a $\cD_E$-module underlying a pure, monodromic Hodge module on a vector bundle $E = X\times \A^1$ of rank 1. Then
\[ z: \cM^0 \to \cM^1 \quad \text{and} \quad \de_z: \cM^1 \to \cM^0\]
are both 0.
\end{cor}
\begin{proof} By the last statement in ~\cite[Prop 2.12]{monoMHM}, we know $N = 0$. In particular, $N \cM^1 = (z\de_z)\cM^1 = 0$. From this, we see that $\text{im}(\de_z: \cM^1 \to \cM^0) \subseteq \ker(z: \cM^0 \to \cM^1)$.

But $\cM$ admits a decomposition by strict support, and so we know
\[\cM^0 = \phi_{z,1}\cM = \ker(var) \oplus \text{im}(can) = \ker(z) \oplus \text{im}(\de_z),\]
and since this sum is direct, this implies $\text{im}(\de_z) = 0$ and $\ker(z) = \cM^0$, proving the claim.
\end{proof}

Now, we prove an analogue of the vanishing $N =0$ for $r>1$.
\begin{prop} \label{puremono} Let $\cM$ be a $\cD_E$-module underlying a monodromic polarizable pure Hodge module $M \in \MHM(E)$. Then
\[ \cM^\chi = \ker(\theta - \chi +r),\]
i.e., $N = \bigoplus_{\chi \in \Q} (\theta - \chi +r)$ is identically 0 on $\cM$.
\end{prop}
\begin{proof} As $M$ is a polarizable Hodge module, its underlying $\cD_E$-module $\cM$ is semi-simple. Indeed, on some locally closed subset, it is a polarizable variation of Hodge structure $V$, whose underlying local system is semisimple. By Saito's theory, $M = j_{!*}(V)$, and since $j_{!*}$ preserves simple $\cD$-modules, the proof is complete.

Now, each simple direct summand of $\cM$ is monodromic, with operator $N = \bigoplus_{\chi \in \Q} (\theta - \chi +r)$. As this is nilpotent, it cannot be an isomorphism, and so it must be 0 on each simple summand. In particular, the operator is $0$ on $\cM$.
\end{proof}

Finally, we prove that the Hodge filtration on a monodromic mixed Hodge module decomposes along the monodromic decomposition. This will be important in the study of the Fourier transform for monodromic Hodge modules later on.

\begin{thm} \label{HodgeDecomp} Let $(\cM,F_\bullet)$ be a filtered $\cD_E$-module underlying a monodromic mixed Hodge module on $E = X\times \A^r$, with decomposition $\cM = \bigoplus_{\chi \in \Q} \cM^\chi$. Then
\[ F_p \cM = \bigoplus_{\chi \in \Q} F_p\cM \cap \cM^\chi\]
for all $p \in \Z$.
\end{thm}
\begin{proof} The key observation is that, on $\{z_i \neq 0\}$, there are two sets of coordinates: $(z_1,\dots, z_r)$ and $(z_i^{-1},z_1z_i^{-1},\dots, z_rz_i^{-1})$. The change of coordinates formula says that, if $\theta = \sum_{j=1}^r z_j \de_{z_j}$, then
\[ \de_{z_i^{-1}} = - \theta.\]

Thus, if we look at $\cM_{z_i} = j_i^*(\cM)$, where $j_i: \{z_i\neq 0\} \subseteq E$ is the open immersion, then
\[ \theta\left(\frac{m}{z_i^\ell}\right) = \frac{(\theta - \ell)m}{z_i^{\ell}}, \quad z_i^{-1} \de_{z_i}^{-1}\left(\frac{m}{z_i^\ell}\right) = \frac{(\ell -\theta)m}{z_i^\ell},\]
and so $\cM_{z_i}$ is monodromic along $z_1,\dots, z_r$ \emph{and} along the single coordinate $z_i^{-1}$. We already saw this in the previous section. When we apply $j_{i*}$ to this module, the resulting module is monodromic in both senses, and so $j_{i*}(j_i^*(\cM))$ has Hodge filtration which decomposes by ~\cite[]{monoMHM}.

Now, we have a canonical morphism $\cM \to \bigoplus_{i=1}^r j_{i*}(j_i^*(\cM))$ of mixed Hodge modules, whose kernel $K$ is supported on the zero section. Write this as an exact sequence
\[ 0 \to K \to \cM \to \bigoplus_{i=1}^r j_{i*}(j_i^*(\cM)).\]

Applying the exact (by Lemma \ref{bistrict}) functor $F_p V^{>0}$ to this sequence, and using the fact that $V^{>0}K =0$, we get an injection
\[ F_p V^{>0}\cM \to \bigoplus_{i=1}^r F_p V^{>0}j_{i*}(j_i^*(\cM)).\]

In particular, if $m = \sum_{\chi >0} m_\chi \in F_p V^{>0}\cM$, then we conclude that $m_\chi \in F_p \cM$.

If $m = m_0 + \sum_{\chi >0} m_\chi \in F_p \cM$, then $z_i m =z_i m_0 + \sum_{\chi > 0} z_i m_\chi \in F_p V^{>0}\cM$, and so we know each piece lies in $F_p \cM$. In particular, $z_im_\chi \in F_p \cM$ for all $\chi >0, 1\leq i\leq r$, so again by filtered acylicity of $A^\chi(\cM)$, this implies $m_\chi \in F_p\cM$ for all $\chi>0$. Finally, $m_0 = m - \sum_{\chi >0} m_\chi \in F_p \cM$, too.

Now, we are able to proceed by descending induction on $\beta<0$. Let $m = m_\beta + \sum_{\chi > \beta} m_\chi \in F_p \cM$ for some $\beta <0$. Then $\overline{m} \in F_p \gr_V^\beta \cM$, so by filtered acyclicity of $C_{\beta}(\cM)$, there exists $n_1,\dots, n_r \in F_p V^{\beta+1}\cM$ and $\epsilon \in F_pV^{>\beta} \cM$ such that
\[ m = \sum_{i=1}^r \de_{z_i}(n_i) + \epsilon,\]
by assumption, $\epsilon$ and $n_i$ for all $i$ have their homogeneous pieces lying in $F_p$ (resp. $F_{p-1}$), so $m$ has its homogeneous pieces lying in $F_p$, too.
\end{proof}

\section{The restriction functors} \label{sec:pmhc}

In this section we prove Theorem~\ref{thm:main} and it is more convenient to work with \emph{right} $\cD$-modules. Recall that the convention for right $\cD$-modules is that the $V$-filtration be indexed increasingly. The proof is split into three parts: Theorem~\ref{thm:mhc}, Theorem~\ref{thm:compare} and Theorem~\ref{thm:comp}. For simplicity, we denote by $B_Z(\cM)=B_0(\cM)$ and $C_Z(\cM)=C_0(\cM)$ to emphasize the $V$-filtration is along the smooth subvariety $Z$. If the $V$-filtration is clear from the context, we will simply use the notation $B(\cM)$ or $C(\cM)$. 

\subsection{Mixed Hodge complex} \label{sec:mhc}
We first prove that for $\cM$ underlying a mixed Hodge module the complex $B(\cM)$ together with $W$ induced by the relative monodromy filtration is a mixed Hodge complex. A \emph{mixed Hodge complex}, roughly speaking, is a bifiltered complex of $\cD$-modules $(C,F,W)$, where $F$ is a decreasing ``Hodge" filtration by $\cO$-submodules and $W$ is an increasing ``weight" filtration by $\cD$-submodules with $\Q$-structure $(C_{\Q},W_{\Q})$. These data should satisfy $\DR(C,W)\simeq (C_{\Q},W_{\Q})\otimes_\Q\C$ and that 
\[
\gr^W_k (C,F) \simeq \bigoplus_{\ell\in\Z}\cH^\ell \gr^W_k (C,F)[-\ell]
\] 
in the derived category of filtered $\cD$-modules. Moreover, $(\cH^\ell \gr^W_k C,F)$ together with the induced $\Q$-structure underlies a polarizable Hodge module of weight $k+\ell$ for any $k$ and $\ell$. Theorem~\ref{thm:main}(a) is restated as follows:

\begin{thm}\label{thm:mhc}
Let $M=(\cM,F,L,\cK)$ be a mixed Hodge module on a smooth variety $X$ as in Theorem~\ref{thm:main} and let $Z$ be a smooth subvariety of $X$. Then $B_Z(\cM)$ together with the relative monodromy filtration is a mixed Hodge complex.
\end{thm}

\begin{proof}
We first remark that $B(\cM)$ carries a $\Q$-structure. Indeed, by Theorem~\ref{thm:gr}
\[
\DR_Z \big(B(\cM)\big) \simeq \DR_Z(i^!\cM) \simeq i^!\cK \otimes_\Q \C.
\]
In fact, if $W$ is the filtration on $B(\cM)$ induced by the monodromy filtration on each $\gr^V_\alpha\cM$ relative to $\gr^V_\alpha L_\bullet \cM$ then $W_kB(\cM)$ also carries a $\Q$-structure. This is because 
\[
\DR_Z\big(i^!_Z W_k \Sp(\cM)\big)\simeq i^!_Z W_k\Sp(\cK)\otimes_\Q \C, \quad i_Z:Z \to T_ZX
\]
and $i^!_Z W_k\Sp(\cM) \simeq W_kB(\cM)$ by the fact that the retraction constructed in the proof of Theorem~\ref{thm:gr} also preserves the filtration $WB(\cM)$. Recall that $\Sp(\cM)$ is the specialization of $\cM$ introduced in~\ref{subsec:dfn}.

\textit{Pure case.} We first prove the case when $(\cM,F,\cK)$ is a polarizable Hodge module of weight $w$. If $\cM$ is supported on $Z$ then $B(\cM)\simeq i_+\gr^V_0\cM$ in the $(F,W)$-bifiltered category and therefore, the theorem follows easily. Now assume that the support of $\cM$ is not contained in $Z$. Let $\pi: \widehat X\rightarrow X$ be the blow up along $Z$ and $\widehat \cM$ be the minimal extension of $\cM$ to $\widehat{X}$ from $\widehat{X} - E \cong X - Z$. Then we can factor the blow up into the graph embedding followed by the smooth projection
\[
\begin{tikzcd}
\widehat X \arrow{r}{{i_{\pi}}} & \widehat X\times X \arrow{r}{p} & X
\end{tikzcd}
\]
The proof consists of two steps:

\textbf{Step 1.} We show that $B_{p^{-1}Z}\left({i_{\pi}}_+\widehat \cM\right)$ is a mixed Hodge complex. 

In fact, the complex $B_{p^{-1}Z}\left({i_{\pi}}_+\widehat \cM\right)$ together with the monodromy filtration is quasi-isomorphic to $B_{E}\left(\widehat\cM\right)$ locally, where $E$ is the exceptional divisor of $\pi$. Note that, although $E$ is not defined by a global function, we can make the complex $B_{E}\left(\widehat\cM\right)$ well-defined by 
\[
\gr^V_0\widehat\cM\otimes \cO(-E)|_E \rightarrow \gr^V_{-1}\widehat\cM.
\]
As we can see in the proof of Theorem~\ref{thm:acyc}: the formula~\eqref{eq:fvgra} is compatible with the monodromy filtration, i.e.
\[
F_\ell \gr^W \gr^V_\alpha {i_\pi}_+\widehat\cM=\sum_{k\geq 0}\sum_{a_2+a_3+\cdots a_r=k} F_{\ell-k} \gr^W\gr^V_{\alpha-k} g_+  \widehat\cM  \partial^{a_2}_2\partial^{a_3}_3\cdots \partial^{a_r}_r
\] 
But since $B_E\left(\widehat\cM\right)$ is a mixed Hodge complex, and this property (like the property of being a Hodge module) is local, it follows that $B_{p^{-1}Z}\left({i_{\pi}}_+\widehat \cM\right)$ is also a mixed Hodge complex. Due to the decomposition theorem of polarizable Hodge modules, the module $\cM$ is a summand of $\cH^0 p_+{i_{\pi}}_+\cM$. Therefore, we reduce the proof to the following.

\textbf{Step 2.} We prove that if $B_{p^{-1}Z}\left(\cM\right)$ is a mixed Hodge complex for a polarizable Hodge module $\cM$ of weight $w$ on $Y\times X$, where $p:Y\times X\rightarrow X$ is the second projection proper over the support of $\cM$, then $B_{Z}\left(\cH^\ell p_+\cM\right)$ is a mixed Hodge complex of weight $w+\ell$ for any $\ell\in\Z$. 

In fact, we have 
\[
\begin{aligned}
p_+\left( \gr^W_k B_{p^{-1}Z}(\cM)\right) &\simeq \bigoplus_{i\in\Z}p_+ \left(\cH^i\gr^W_k B_{p^{-1}Z}(\cM)\right) [-i] \\
& \simeq \bigoplus_{i,j\in\Z} \cH^jp_+\left(\cH^i\gr^W_k B_{p^{-1}Z}(\cM)\right)[-i-j]
\end{aligned}
\]
in the derived category of filtered $\cD$-modules. On the other hand, we also have the decomposition in the derived category of filtered $\cD$-modules by Lemma~\ref{lem:e2}(d):
\[
p_+\left( \gr^W_k B_{p^{-1}Z}(\cM)\right) \simeq \bigoplus_{\ell\in\Z} \cF_{k,\ell}[-\ell],
\]
where $\cF^i_{k,\ell}=\cH^\ell p_+ \gr^W_k B^i_{p^{-1}Z}(\cM)$. This implies 
\begin{equation}\label{eq:decompf}
\cF_{k,\ell}\simeq \bigoplus_{i\in\Z}\cH^i \cF_{k,\ell}[-i]    
\end{equation}
and $\cH^i \cF_{k,\ell}$ is a polarizable Hodge module of weight $w+k+i+\ell$. For each $k$ we have a weight spectral sequence 
\[
E^{i,j}_1(k)=\cH^{i+j}p_+\gr^W_{-i} B^k_{p^{-1}Z}(\cM) \Rightarrow E^{i,j}_\infty(k)= \gr^W_{-i}\cH^{i+j}p_+ B^k_{p^{-1}Z}(\cM) 
\]
so that $E^{i,j}_1(k)=\cF^k_{-i,i+j}$. Note that by the bistrictness proved in~\cite{BMS}, we have
\[
E^{i,j}_\infty(k)=\gr^W_{-i}B^k_Z(\cH^{i+j}p_+\cM).
\] 
We gather some facts deduced from the deformation to the normal bundle argument (Lemma~\ref{lem:e2}): 
\begin{enumerate}
    \item the spectral sequence degenerates at the second page;
    \item the induced filtration $W\cH^{i+j}p_+ B^k_{p^{-1}Z}(\cM)$ is the monodromy filtration on 
    \[
    \cH^{i+j}p_+ B^k_{p^{-1}Z}(\cM)=\left(\gr^V_{-k}\cH^{i+j}p_+\cM \right)^{\binom{r}{k}};
    \] 
    \item lastly, $E^{i,j}_2(k)$ is a summand of $E^{i,j}_1(k)$ in the category of filtered $\cD$-modules.
\end{enumerate} 
Therefore, the differential $d_1$ on the first page induces a double complex 
\[
\cdots \xrightarrow[]{d_1} \cF_{k+1,\ell-1} \xrightarrow[]{d_1} \cF_{k,\ell} \xrightarrow[]{d_1} \cF_{k-1,\ell+1} \xrightarrow[]{d_1} \cdots.
\]
Let $T$ be the total complex of this double complex. Then by~\eqref{eq:decompf} and semisimplicity, $T$ decomposes into
\begin{equation}\label{eq:decompt}
\begin{aligned}
     &\bigoplus_{i}\left\{ \cdots \xrightarrow[]{d_1} \cH^i \cF_{k+1,\ell-1} \xrightarrow[]{d_1} \cH^i\cF_{k,\ell} \xrightarrow[]{d_1} \cH^i \cF_{k-1,\ell+1} \xrightarrow[]{d_1} \cdots
\right\}[-i]  \\
     \simeq &\bigoplus_{i,j} \cH^j_{d_1}\cH^i\cF_{k-\bullet,\ell+\bullet}[-i-j]
\end{aligned}
\end{equation}
in the derived category of filtered $\cD$-modules. On the other hand, by the claim $(c)$ above, we also have another decomposition in the derived category:
\[
\begin{aligned}
T & \simeq \bigoplus_{j} \cH^j_{d_1}\cF_{k-\bullet,\ell+\bullet}[-j]. 
\end{aligned}
\]
Since $\cH^j_{d_1}\cF_{k-\bullet,\ell+\bullet}=\gr^W_{k-j}B_Z(\cH^{\ell+j} p_+\cM)$, the decomposition~\eqref{eq:decompt} implies $\gr^W_{k-j}B_Z(\cH^{\ell+j} p_+\cM)$ decomposes into the direct sum of its cohomology in the derived category of filtered $\cD$-modules and the cohomology $\cH^i\gr^W_{k}B_Z(\cH^{\ell} p_+\cM)$ is of weight $w+\ell+k+i$. It is easy to see that the decomposition is compatible with $\Q$-structures and therefore, we conclude the proof.

\textit{Mixed case.} By Lemma~\ref{lem:relsplitfun} below, there exists a functorial splitting
\[
\gr^W \gr^V_\alpha \cM \simeq \gr^W \gr^L \gr^V_\alpha \cM,
\]
with respect to $t_1,t_2,\dots,t_r$ which implies $\gr^WB(\cM)\simeq \gr^WB(\gr^L\cM)$. Therefore, we reduce the proof to the case where $\cM$ underlies a pure Hodge module.
\end{proof}

We collect some corollaries of Deligne's Theorem which we have already applied in the previous theorem and will apply these results in the proof of Theorem~\ref{thm:comp}. The proof is based on~\cite[1.5]{Saito-MHM} and a theorem of Deligne (Theorem~\ref{thm:delsy}). For the purpose of the exposition, we postpone the proof to the end of this section.
\begin{lem}\label{lem:relsplitfun}
Let $\cM,\cM'$ be mixed Hodge modules on a smooth variety $X$ and $V$ be the $V$-filtration along a smooth subvariety $Z$. Let $L$ be the filtration on $\gr^V_\alpha$ induced by the weight filtration and $W=W(\theta-\alpha,L)$ be the relative monodromy filtration on $\gr^V_\alpha$. Then we have:
\begin{enumerate}
 \item For any local defining equation $f$ of $Z$, the induced filtered morphism
     \[
     f: (\gr^W\gr^V_\alpha\cM,F) \rightarrow (\gr^W\gr^V_{\alpha-1}\cM,F) 
     \]
    splits into $f:\gr^W\gr^L\gr^V_\alpha\cM\rightarrow \gr^W\gr^L\gr^V_{\alpha-1}\cM$.
 \item For any local vector fields $\xi$ normal to $Z$, the induced filtered morphism
     \[
     \xi: (\gr^W\gr^V_\alpha\cM,F)\rightarrow (\gr^W\gr^V_{\alpha+1}\cM,F[-1]) 
     \]
     splits into $\xi:(\gr^W\gr^L\gr^V_\alpha\cM,F)\rightarrow (\gr^W\gr^L\gr^V_{\alpha+1}\cM,F[-1])$.
      
 \item If $T:\cM\rightarrow \cM'$ is a morphism of mixed Hodge modules, then the filtered morphism
     \[
     \gr^WT: (\gr^W\gr^V_\alpha\cM,F) \rightarrow (\gr^W\gr^V_\alpha\cM',F) 
     \] 
    splits into $\gr^WT: (\gr^W\gr^L\gr^V_\alpha\cM,F) \rightarrow (\gr^W\gr^L\gr^V_\alpha\cM',F)$.
\end{enumerate}
\end{lem}

Now we turn to the complex $C(\cM)$. The filtration $W_kC(\cM)$ also carries a $\Q$-structure. In fact, it follows from Proposition~\ref{prop:cvvc} and the fact that the retraction constructed in Theorem~\ref{thm:gr} respects the filtration $W$ that 
\[
\DR_Z \big(W_kC(\cM)\big) \simeq \DR_Z \left(p_+ W_k\Sp(\cM)\right) \simeq p_*W_k \Sp\cK\otimes_\Q \C
\]
where $p:T_ZX\rightarrow Z$ is the projection. Therefore, we can simply modify the proof of Theorem~\ref{thm:mhc} to prove the following. 

\begin{thm}\label{thm:cmhc}
Let $(\cM,F,L,\cK)$ be a mixed Hodge module on a smooth variety $X$ and $Z$ is a smooth subvariety. Then $C_Z(\cM)$ together with the relative monodromy filtration is also a mixed Hodge complex.
\end{thm}

By a formal argument Lemma~\ref{lem:formal} of Deligne in~\cite{Hodge2} on the mixed Hodge complexes (see also~\cite{MHC}*{Proposition 2.3}), we deduce the following by noting that a mixed Hodge complex satisfies all three conditions in Lemma~\ref{lem:formal}:

\begin{cor}\label{thm:degsep}
The Hodge spectral sequences of $B(\cM)$ and $C(\cM)$ degenerate at the first page while the weight spectral sequences degenerate at the second page.
\end{cor}

\begin{lem}\label{lem:formal}
For a complex $(C,F)$ of filtered $\cD_X$-modules with a finite increasing weight filtration $W$ and a constructible complex $(C_\Q,W_\Q)$ over $\Q$ such that 
\begin{enumerate}
    \item $\DR_X(C,W)\cong (C_\Q,W_\Q)\otimes_\Q \C$;
    \item each cohomology module $\big(\cH^\ell(\gr^W_k C), {}^{\mathfrak p}\cH^\ell(\gr^W_kC_\Q),F\big)$ underlies a polarizable Hodge module of weight $k+\ell$ and
    \item the Hodge filtration $F$ is strict on $\gr^W_kC$ for all $k\in \Z$.
\end{enumerate}
then the Hodge spectral sequence degenerates at the first page and the weight spectral sequence degenerates at the second page.
\end{lem}
\begin{proof}
For all $p,q \in Z$, the term $E_1^{p,q}=\cH^{p+q}\gr^W_{-p}C$ of the first page of the weight spectral sequence associated to $W$ is a polarizable Hodge module of weight $q$. The differential $d_1$ of the first page of the spectral sequence is a morphism of polarizable Hodge modules, since it is compatible with $F$ and the $\Q$-structure. It follows that $E_r^{p,q}$ as the cohomology of $d_{r-1}$ is also a polarizable Hodge module of weight $q$ for all $r\geq 1$. But then $d_r$ must vanish if $r\geq 2$ because $d_r: E_r^{p,q}\to E_r^{p+r,q-r+1}$ is a morphism of polarizable Hodge modules of different weights when $r\geq 2$. This proves that the weight spectral sequence degenerates at the second page. In particular, $(\cH^\ell(C),F)$ underlies a weakly mixed Hodge module.

As for the Hodge spectral sequence, the degeneration at the first page is equivalent to the Hodge filtration on $C$ being strict, i.e. the canonical maps $\cH^\ell(F_i C) \to \cH^\ell(C)$ are injective for all $i$ and $\ell$. We prove this by induction on the length of the weight filtration $W$. If there is some $k$ that $\gr^W_kC=C$ then by the condition $(c)$ we obtain the base case. Assume that $W_kC=C$ and the Hodge filtration is strict on $W_{k-1}C$. By the short exact sequence 
\[
0 \to W_{k-1}C \to W_kC \to \gr^W_{k}C \to 0
\]
we get a commutative diagram between two exact sequences:
\[
\begin{tikzcd}[sep=scriptsize]
\cH^{\ell-1}(F_i\gr^W_k C) \arrow{r}\arrow[d,hook] & \cH^\ell(F_iW_{k-1}C) \arrow{r}\arrow[d,hook] & \cH^\ell(F_iW_{k}C) \arrow{r}\arrow{d} & \cH^\ell(F_i\gr^W_k C) \arrow{r}\arrow[d,hook] & \cH^{\ell+1}(F_iW_{k-1}C) \arrow[d,hook] \\
  \cH^{\ell-1}(\gr^W_k C) \arrow{r}{\alpha} & \cH^\ell(W_{k-1}C) \arrow{r} & \cH^\ell(W_{k}C) \arrow{r} & \cH^\ell(\gr^W_k C) \arrow{r} & \cH^{\ell+1}(W_{k-1}C)
\end{tikzcd}
\]
We finish the proof a diagram chase. Assume $m\in \cH^\ell(F_iW_{k}C)$ is sent to $0$ in $\cH^\ell(W_kC)$. Then the image of $m$ in $\cH^\ell(F_i\gr^W_kC)$ must vanish since, by the condition $(c)$, $\cH^\ell(F_i\gr^W_kC) \to \cH^\ell(\gr^W_kC)$ is injective. By exactness, there is an element $m'\in \cH^\ell(F_iW_{k-1}C)$ whose image is $m$. Let $m''$ be the image of $m'$ in $\cH^\ell(W_{k-1}C)$. Then there is $\tilde m\in \cH^{\ell-1}(\gr^W_k C)$ whose image is $m''$, as $m''$ is sent to zero by $\cH^\ell(W_{k-1}C)\to\cH^\ell(W_{k}C)$. We will conclude the proof if we can find $\bar m \in \cH^{\ell-1}(F_i\gr^W_k C)$ such that $\bar m \mapsto m'$ because this will imple $m=0$. This is done by noticing that $\alpha$ is a morphism of graded polarizable weakly mixed Hodge modules since $\alpha$ preserves the Hodge filtration and the $\Q$-structure and also $\cH^{\ell-1}(\gr^W_k C)$ is of weight $k+\ell-1$ which is the top weight of $\cH^\ell(W_{k-1}C)$. It follows that $\alpha$ is strict and there is $\bar m \in \cH^{\ell-1}(F_i\gr^W_k C)$ whose image is $m'$.
\end{proof}

\subsection{Comparison to the restriction functors} The goal of this part is to prove Theorem~\ref{thm:main}(b):

\begin{thm} \label{thm:compare}
If $(\cM,F)$ is a graded polarizable mixed Hodge module then the complex $B(\cM)$ (resp. $C(\cM)$) is isomorphic to $(i^!\cM,F)$ (resp. $(i^*\cM,F)$) in the derived category of filtered $\cD$-modules with $\Q$-structures.
\end{thm}

Before starting the proof, we give a lemma comparing \v{C}ech complexes. Let $Z_i$ be the hypersurface defined by $t_i=0$. Then the complex ${i}_+i^! M$ can be expressed by the \v{C}ech complex
\begin{equation}\label{eq:11}
K(M,Z_1,Z_2,\dots,Z_r)=\left\{M\rightarrow \bigoplus^r_{i=1}M(* Z_i) \rightarrow \cdots \rightarrow M\left(* \sum^r_{i=1} Z_i\right) \right\}
\end{equation}
placed in degrees $0,1,\dots,r$ where the morphism is induced by natural morphisms $N\rightarrow N(*Z_i)$ for any mixed Hodge module $N$ on $X$. Similarly, the complex ${i}_+i^* M$ can be expressed by the \v{C}ech complex (see the proof of ~\cite[Prop. 2.19]{Saito-MHM})
\begin{equation}\label{eq:12}
K_!(M,Z_1,Z_2,\dots,Z_r)=\left\{M\left(! \sum^r_{i=1}Z_i\right)\rightarrow  \cdots \rightarrow \bigoplus^r_{i=1}M(! Z_i) \rightarrow M \right\}
\end{equation}
placed in degree $-r,-r+1,\dots,0$, where the morphism is induced by the natural morphisms $N(! Z_i)\rightarrow N$ for any mixed Hodge module $N$ on $X$. Recall that $N(!D)=\bD\big(\bD(N)(*D)\big)$ for a hypersurface $D$, where $\bD$ is duality for mixed Hodge modules. By definition, this is $j_! j^*(N)$, where $j: X - Z \to X$ is the inclusion of the complement of $Z$.

\begin{lem}\label{lem:graph}
Let $\gamma: X\rightarrow X\times \A^r$ be the graph embedding of $f$ and $i_H:H=X\times \{0\}\rightarrow X\times \A^r$ be the closed embedding of the central fiber. Then we have natural isomorphisms
\begin{enumerate}
    \item $\gamma_+ K(M, Z_1,Z_2,\dots,Z_r)\simeq {i_{H}}_+ K(M, Z_1,Z_2,\dots,Z_r)$ and
    \item $\gamma_+ K_!(M, Z_1,Z_2,\dots,Z_r)\simeq {i_{H}}_+ K_!(M, Z_1,Z_2,\dots,Z_r)$
\end{enumerate}
in  $D^b(\MHM(X\times \A^r))$.
\end{lem}
\begin{proof}[Proof of the lemma]
Let $\widetilde M= M \boxtimes \Q_{\A^r}^H[r]$ be the pullback of $M$ to $X\times \A^r$. Denote by $D_j$ be the divisor on $X\times \A^r$ defined by $f_j-t_j=0$ for $j=1,2,\dots,r$ and denote by $H_j$ be the divisor on $X\times \A^r$ defined by $t_j=0$. Then we have 
\[
\begin{aligned}
K\left(\widetilde M, D_1,D_2,\dots,D_r\right) \simeq \gamma_+M(-r)  \quad \text{and}\quad K\left(\widetilde M,H_1,H_2,..,H_r\right) \simeq {i_{H}}_+M(-r),
\end{aligned}
\]
recalling that $(-r)$ is the Tate twist. It follows that
\[
\begin{aligned}
K(\widetilde M, D_1,D_2,\dots,D_r,H_1,H_2,\dots,H_r) & =K\left(K\left(\widetilde M, D_1,D_2,\dots,D_r\right),H_1,H_2,\dots,H_r\right) \\
    & \simeq K(\gamma_+ M(-r), H_1,H_2,\dots,H_r) \\
    & \simeq \gamma_+ K(M(-r), Z_1,Z_2,\dots,Z_r).
\end{aligned}
\]
On the other hand, 
\[
\begin{aligned}
K(\widetilde M, D_1,D_2,\dots,D_r,H_1,H_2,\dots,H_r) & =K\left(K\left(\widetilde M, H_1,H_2,\dots,H_r \right),D_1,D_2,\dots,D_r\right) \\
    & \simeq K({i_{H}}_+ M(-r), D_1,D_2,\dots,D_r) \\
    & \simeq {i_{H}}_+ K(M(-r), Z_1,Z_2,\dots,Z_r).
\end{aligned}
\]
We conclude the first statement of the lemma by undoing the Tate twist. The second statement is similar and, we leave it to the reader.
\end{proof}

\begin{proof}[Proof of the Theorem]
Note that the $\Q$-structure has already been handled in Theorem~\ref{thm:mhc}.

\textit{1.} We first deal with the complex $B(\cM)$. As above, the functor $i_+i^!\cM$ can be defined by the \v{C}ech complex in the derived category of mixed Hodge modules:
\begin{equation}\label{eq:K}
K(\cM)= K(\cM,Z_1,\cdots, Z_r) = \left\{\cM\rightarrow \bigoplus \cM(*Z_i)\rightarrow \cdots \rightarrow \cM(*\sum^r_{i=1}Z_i)\right\}
\end{equation}
placed in degrees $0,1,\dots,r$. Moreover, the complex $K(\cM)$ is isomorphic to ${i}_+ \gr^V_0 K(\cM)$ in the derived category of $(F,W)$-bifiltered $\cD$-modules by Lemma~\ref{lem:graph}. Consider the double complex $BK(\cM)$:
\begin{equation}\label{eq:BK}
\begin{tikzcd}
\gr^V_0\cM \arrow{r}{\delta_0}\arrow{d} & (\gr^V_{-1}\cM)^r \arrow{r}{\delta_1}\arrow{d} & \cdots \arrow{r}{\delta_{r-1}} & \gr^V_{-r}\cM \arrow{d} \\
\bigoplus^r_{i=1}\gr^V_0\cM(*Z_i) \arrow{r}{\delta_0} \arrow{d} & \bigoplus^r_{i=1}(\gr^V_{-1}\cM(*Z_i))^r \arrow{r}{\delta_1}\arrow{d} & \cdots \arrow{r}{\delta_{r-1}} & \bigoplus^r_{i=1}\gr^V_{-r}\cM(*Z_i) \arrow{d} \\
\cdots\arrow{d} & \cdots \arrow{d} &\cdots & \cdots \arrow{d} \\
\gr^V_0\cM (*\sum^r_{i=1}Z_i) \arrow{r}{\delta_0} & (\gr^V_{-1}\cM (*\sum^r_{i=1}Z_i))^r \arrow{r}{\delta_1} & \cdots \arrow{r}{\delta_{r-1}} & \gr^V_{-r}\cM(*\sum^r_{i=1}Z_i)  
\end{tikzcd}
\end{equation}
whose uppermost row is $BK^0(\cM)=B(\cM)$ and leftmost column is $B^0K(\cM)=\gr^V_0K(\cM)$. The total complex of $BK(\cM)$ is $(F,W)$-bifiltered quasi-isomorphic to $\gr^V_0K(\cM)$ because $\gr^V_\alpha K(\cM)$ is $(F,W)$-bifiltered acyclic for $\alpha<0$ and Lemma~\ref{lem:relmonexact}. On the other hand, the total complex of $BK(\cM)$ is also  $F$-filtered quasi-isomorphic to $B(\cM)$ because each row $BK^i(\cM)$ is $F$-filtered acyclic when $i\neq 0$ by Corollary~\ref{thm:degsep} and Theorem~\ref{thm:topop}. We conclude that $\gr^V_0K(\cM)$ and $B(\cM)$ are isomorphic in the derived category of $F$-filtered $\cD_Z$-modules. But $\gr^V_0K(\cM)$ is $(F,W)$-bifiltered quasi-isomorphic to $i^!(\cM,F,W)$. We conclude the proof of this part.

\textit{2.} Next, we deal with the complex $C(\cM)$. The functor $i_+i^*\cM$ can be computed by the the \v{C}ech complex 
\begin{equation}\label{eq:K!}
K_!(\cM)=\left\{\cM\left(! \sum^r_{i=1}Z_i\right)\rightarrow  \cdots \rightarrow \bigoplus^r_{i=1}\cM(! Z_i) \rightarrow \cM \right\}
\end{equation}
placed in degrees $-r,-r+1,\dots,0$. Moreover, the complex $K_!(\cM)$ is isomorphic to ${i}_+ \gr^V_0 K_!(\cM)$ in the derived category of $(F,W)$-bifiltered $\cD$-modules by Lemma~\ref{lem:graph}. Consider the double complex $CK_!(\cM)$
\begin{equation}\label{eq:CK}
\begin{tikzcd}
\gr^V_{-r}\cM \arrow{r}{\delta_{-r}} & \cdots \arrow{r}{\delta_{-r+1}} & (\gr^V_{-r+1}\cM)^r \arrow{r}{\delta_{-1}} & \gr^V_{0}\cM  \\
\bigoplus^r_{i=1}\gr^V_{-r}\cM(!Z_i) \arrow{r}{\delta_{-r}} \arrow{u} &  \cdots \arrow{r}{\delta_{-r+1}}\arrow{u} &  \bigoplus^r_{i=1}(\gr^V_{-r+1}\cM(!Z_i))^r\arrow{r}{\delta_{-1}}\arrow{u} & \bigoplus^r_{i=1}\gr^V_{0}\cM(!Z_i) \arrow{u} \\
\cdots\arrow{u} & \cdots \arrow{u} &\cdots \arrow{u} & \cdots \arrow{u} \\
\gr^V_{-r}\cM (!\sum^r_{i=1}Z_i) \arrow{r}{\delta_{-r}}\arrow{u} & \cdots \arrow{r}{\delta_{-r+1}}\arrow{u} &  (\gr^V_{-r+1}\cM (!\sum^r_{i=1}Z_i))^r\arrow{r}{\delta_{-1}}\arrow{u} & \gr^V_{0}\cM(!\sum^r_{i=1}Z_i)\arrow{u}  
\end{tikzcd}
\end{equation}
whose uppermost row is $CK^0_!(\cM)=C(\cM)$ and leftmost column is $C^0K(\cM)=\gr^V_0K_!(\cM)$. The total complex of $CK_!(\cM)$ is $(F,W)$-bifiltered quasi-isomorphic to $\gr^V_0K_!(\cM)$ because $\gr^V_\alpha K_!(\cM)$ is $(F,W)$-bifiltered acyclic for $\alpha<0$. On the other hand, the total complex of $CK_!(\cM)$ is also $F$-filtered quasi-isomorphic to $C(K)$ because each row $CK^i_!(\cM)$ is $F$-filtered acyclic when $i\neq 0$ because of Corollary~\ref{thm:degsep} and Theorem~\ref{thm:topop}. We conclude that $\gr^V_0K_!(\cM)$ and $C(\cM)$ are isomorphic in the derived category of $F$-filtered $\cD_Z$-modules. Finally, $\gr^V_0K_!(\cM)$ is bifiltered quasi-isomorphic to $i^*(\cM,F,W)$. We conclude the proof of this part.
\end{proof}

\begin{rmk}\label{rmk:bypass}
If one is just interested in the isomorphisms 
\[
(B(\cM),F)\simeq (i^!\cM,F) \quad \text{and}  \quad (C(\cM),F)\simeq (i^*\cM,F)
\] 
in the derived category of filtered $\cD$-modules, there is a way to bypass mixed Hodge complexes as are used in Theorem~\ref{thm:mhc} and Theorem~\ref{thm:cmhc}. To prove $(B(\cM),F)\simeq (i^!\cM,F)$, we just need to show that $(B(\cM(*Z_i)),F)$ is filtered acyclic for any $Z_i$ as in the proof Theorem~\ref{thm:compare}. For this we consider $\widehat{\cM}(*\widehat Z_i +E)$ on the blow-up $\pi:\widehat X\to X$ along $Z$ where $\widehat{\cM}$ is the minimal extension of $\cM|_{X-Z}$, $\widehat Z_i$ is the strict transform of $Z_i$ and $E$ is the exceptional divisor. Note that $\pi_+\widehat{\cM}(*\widehat Z_i+E)=\cM(*Z_i)$. It follows from the computation in the proof of Theorem~\ref{thm:acyc} that $B(i_{\pi_+}\cM(*\widehat Z_i+E))$ is filtered acyclic where $i_\pi:\widehat X\to \widehat X\times X$ is the graph embedding because of the fact that one of the Koszul differentials is filtered bijective. We can conclude by applying $p_+$ to $B(i_{\pi_+}\cM(*\widehat Z_i+E))$ and the bistrictness result for smooth, projective morphisms. The same idea works for the filtered acyclicity of $(C(\cM(*Z_i)),F)$.
\end{rmk}

\subsection{Finishing the proof} We now prove the last part of Theorem~\ref{thm:main}:

\begin{thm}\label{thm:comp}
If $\cM$ is a graded polarizable mixed Hodge module and $W$ is the filtration on $B(\cM)$ and $C(\cM)$ induced by the relative monodromy filtration on $\gr^V_\alpha\cM$, then the quasi-isomorphisms in Theorem~\ref{thm:compare} induce isomorphisms on the cohomologies:
\[
\gr^W_k\cH^\ell B(\cM)\simeq \gr^W_{k+\ell}\cH^\ell i^!_Z\cM \quad \text{and} \quad \gr^W_{k}\cH^{-\ell} C(\cM)\simeq \gr^W_{k-\ell}\cH^{-\ell} i^*_Z\cM
\]
as polarizable Hodge modules for $\ell\geq0$.
\end{thm}

\begin{proof}
\textit{1.} We first focus on the complex $B(\cM)$. We shall prove the following as a preparation:

\begin{lem}\label{lem:grwgrl}
The complex $\cH^\ell_{\delta}\gr^W_k BK(\cM)$ is exact for $\ell\neq0$ and any $k\in\Z$ and the natural inclusion 
\[
\cH^0_{\delta}\gr^W_k BK(\cM)=\ker \gr^W_k\delta_0\rightarrow \gr^W_k\gr^V_0K(\cM)
\] 
is a filtered quasi-isomorphism, where $BK(\cM)$ is defined in~\eqref{eq:BK}.
\end{lem} 
\begin{proof}[Proof of the lemma]
We first prove that the inclusion 
\[
\ker \gr^W \delta_0\rightarrow \gr^W\gr^V_0K(\cM)
\]
is a bifiltered quasi-isomorphism. By Lemma~\ref{lem:relsplitfun}, the double complex $\gr^WBK(\cM)$ decomposes into 
\[
\begin{tikzcd}[column sep=tiny]
\gr^W\gr^L\gr^V_0\cM \arrow{r}{}\arrow{d} & (\gr^W\gr^L\gr^V_{-1}\cM)^r \arrow{r}{}\arrow{d} & \cdots \arrow{r}{} &  \gr^W\gr^L\gr^V_{-r}\cM \arrow{d} \\
\bigoplus^r_{i=1} \gr^W\gr^L\gr^V_0\cM(*Z_i) \arrow{r}{} \arrow{d} & \bigoplus^r_{i=1}(\gr^W\gr^L\gr^V_{-1}\cM(*Z_i))^r \arrow{r}{}\arrow{d} & \cdots \arrow{r}{} & \bigoplus^r_{i=1} \gr^W\gr^L\gr^V_{-r}\cM(*Z_i) \arrow{d} \\
\cdots \arrow{d} & \cdots  \arrow{d} &\cdots  & \cdots \arrow{d} \\
\gr^W\gr^L\gr^V_0\cM (*\sum^r_{i=1}Z_i) \arrow{r}{} & (\gr^W\gr^L\gr^V_{-1}\cM (*\sum^r_{i=1}Z_i))^r \arrow{r}{} & \cdots \arrow{r}{} & \gr^W\gr^L\gr^V_{-r}\cM(*\sum^r_{i=1}Z_i)  
\end{tikzcd}
\]
where $L$ is the filtration induced by the weight filtration on $K(\cM)$. Since the category of polarizable Hodge modules on an algebraic variety is semisimple, the cohomology $\cH^\ell \gr^L K(\cM)$ is a summand  of $\gr^LK^\ell(\cM)$. It follows that $\gr^W\gr^V_0\cH^\ell \gr^L K(\cM)$ is contained in $\cH^\ell \ker \gr^W\delta_0$ because the support of $\gr^W\gr^V_0\cH^\ell \gr^L K(\cM)$ is contained in $Z$. Then due to the fact that
\[
\gr^W\gr^V_0\cH^\ell \gr^L K(\cM)\rightarrow \cH^\ell \ker \gr^W\delta_0
\] 
is injective, we conclude that $\ker \gr^W\delta_0\rightarrow \gr^W \gr^V_0K(\cM)$ is an isomorphism.

Next, we prove that the complex $\cH^\ell_{\delta}\gr^W_k BK(\cM)$ is exact for $\ell>0$. By Theorem~\ref{thm:mhc}, the total complex of $\gr^WBK(\cM)$ decomposes into 
\[
\bigoplus_{\ell\in\Z} \cH^\ell_{\delta}\gr^WBK(\cM)[-\ell].
\]
On the other hand, we know $B^i K(\cM)$ is acyclic for $i>0$, as $K(\cM)$ has cohomology supported on $\{t_1 = \cdots = t_r =0\}$. Hence, $\gr^WB^iK(\cM)$ is $F$-filtered exact for all $i>0$, the total complex of $\gr^WBK(\cM)$ is filtered quasi-isomorphic to $\gr^W\gr^V_0K(\cM)$ which is also filtered quasi-isomorphic to $\cH^0_{\delta}\gr^W BK(\cM)$ as we just proved. This completes the proof of the lemma.
\end{proof}

Returning to the proof of the theorem, we have a weight spectral sequence on $BK^j(\cM)$
\[
E^{p,q}_1=\cH^{p+q}_\delta \gr^W_{-p} BK^j(\cM) \Rightarrow E^{p,q}_\infty=\gr^W_{-p} \cH^{p+q}_\delta BK^j(\cM).
\]
which degenerates at $E^{p,q}_2$ by Theorem~\ref{thm:mhc}. The differential of the first page of the spectral sequence induces morphisms of complexes
\[
S_{k,\ell}=\{\cH^0_\delta\gr^W_{k+\ell} BK(\cM) \rightarrow \cH^1_\delta \gr^W_{k+\ell-1} BK(\cM) \rightarrow \cdots \rightarrow \cH^r_\delta \gr^W_{k+\ell-r} BK(\cM)\}
\]
for any $\ell\in\Z$. By the above lemma, the total complex of $S_{k,\ell}$ is filtered isomorphic to $\cH^0_\delta\gr^W_{k+\ell} BK(\cM)$ and thus, $\gr^W_{k+\ell} \gr^V_0K(\cM)$. On the other hand, because of Theorem~\ref{thm:mhc}, the second page of the weight spectral sequence on $B(\cN)$ is zero if one of the $t_i$ acts bijectively on a graded polarizable mixed Hodge module $\cN$. This means $S_{k,\ell}$ is also filtered isomorphic to the first page of the weight spectral sequence of $B(\cM)$:
\[
\cH^0_\delta\gr^W_{k+\ell} B(\cM) \rightarrow \cH^1_\delta \gr^W_{k+\ell-1} B(\cM) \rightarrow \cdots \rightarrow \cH^r_\delta \gr^W_{k+\ell-r} B(\cM),
\]
which is filtered isomorphic to $\gr^W_{k+\ell} \gr^V_0K(\cM)$. If we take cohomology at degree $\ell$, we conclude that 
\[
\gr^W_k\cH^\ell B(\cM)\simeq \gr^W_{k+\ell}\cH^\ell K(\cM)
\]
as polarizable Hodge modules. 

\textit{2. }We deal with the complex $C(\cM)$. The proof of the following lemma is parallel to the one of Lemma~\ref{lem:grwgrl} and therefore, we leave it to the readers.

\begin{lem}
The complex $\cH^\ell_{\delta}\gr^W_k CK_!(\cM)$ is exact for $\ell\neq0$ and any $k\in\Z$ and the natural quotient 
\[
\gr^W_k\gr^V_0K_!(\cM) \rightarrow \cH^0_{\delta}\gr^W_k CK_!(\cM)=\coker \gr^W_k\delta_{-1}
\] 
is a filtered quasi-isomorphism.
\end{lem}

We also have a weight spectral sequence 
\[
E^{p,q}_1=\cH^{p+q}_\delta \gr^W_{-p} CK^j_!(\cM) \Rightarrow E^{p,q}_\infty=\gr^W_{-p} \cH^{p+q}_\delta CK^j_!(\cM).
\]
which degenerates at the second page by Theorem~\ref{thm:cmhc}. The differential of the first page of the spectral sequence induces morphisms of complexes
\[
T_{k,\ell}=\{\cH^{-r}_\delta\gr^W_{k-\ell+r} CK_!(\cM) \rightarrow \cH^{-r+1}_\delta \gr^W_{k-\ell+r-1} CK_!(\cM) \rightarrow \cdots \rightarrow \cH^0_\delta \gr^W_{k-\ell} CK_!(\cM)\}
\]
for any $\ell\in\Z$. By the above lemma, the total complex of $T_{k,\ell}$ is filtered isomorphic to $\cH^0_\delta\gr^W_{k-\ell} CK_!(\cM)$ and thus, $\gr^W_{k-\ell} \gr^V_0K_!(\cM)$. On the other hand, because of Theorem~\ref{thm:cmhc}, the second page of the weight spectral sequence on $B(\cN)$ is zero if $\cN=\cN(!Z)$. This means $T_{k,\ell}$ is also filtered isomorphic to the first page of the weight spectral sequence of $C(\cM)$:
\[
\cH^{-r}_\delta\gr^W_{k-\ell+r} C(\cM) \rightarrow \cH^{-r+1}_\delta \gr^W_{k-\ell+r-1} C(\cM) \rightarrow \cdots \rightarrow \cH^0_\delta \gr^W_{k-\ell} C(\cM)
\]
which is filtered isomorphic to $\gr^W_{k-\ell} \gr^V_0K_!(\cM)$. If we take cohomology at degree $-\ell$, we conclude that 
\[
\gr^W_{k}\cH^{-\ell} C(\cM)\simeq \gr^W_{k-\ell}\cH^{-\ell} K_!(\cM)
\]
as polarizable Hodge modules. 
\end{proof}

\subsection{Deligne's theorem}
The aim of this part is to prove Lemma~\ref{lem:relsplitfun}. For this purpose, we generalize, with little effort, the theorem on relative monodromy filtrations to the abstract setting, proved by Deligne in his personal letter to Cattani and Kaplan. Then Lemma~\ref{lem:relsplitfun} will be an immediate corollary.

Let $\cA$ be an abelian category and $V$ be an object in $\cA$. Let $L$ be a finite increasing filtration of $V$ and $N$ be a nilpotent endomorphism preserving the filtration $L$. We will now assume that the relative monodromy filtration $W=W(N,L)$ exists and that there is a splitting operator $Y$ for $W$, i.e. $Y$ is a semisimple operator on $V$ with eigenvalues in $\Z$ such that $W_k=\bigoplus_{i\leq k} E_i(Y)$ where $E_i(Y)$ is the $i$-eigenspace of $Y$. We say the splitting operator $Y$ satisfies the \textit{admissibility conditions} if
\begin{equation}\label{eq:adm}
 [Y,N]=-2N,\quad \text{and} \quad YL_i \subset L_i, \quad \text{for all } i.
\end{equation}

Suppose that $Y'$ is a splitting operator for $L$ that commutes with $Y$. Then the pair $(N_0, Y-Y')$ determines an $\fsl_2$-representation on $V$. We will denote the standard $\fsl_2$-triple by $(e^+,e^-,H)$: 
\[
[e^+,e^-]=H, \quad [H,e^{-}]=-2e^{-},\quad [H,e^+]=2e^+.
\]
Then $e^-=N_0$ and $H=Y-Y'$. We call the collection $(V,L,N,Y,Y')$ a \textit{Deligne-system}, a notion introduced in~\cite{Schwarz01}, if in addition
\[
[e^+, N_{j}]=0, \quad \text{for all } j\neq 0
\]
where $N_{j}$ is the $j$-th $\ad Y'$-homogeneous component of $N$. In other words, $N_j$ is $\ad e^-$-primitive in the adjoint representation for $j\neq 0$.

\begin{thm}\label{thm:delsy}
Let $(V,N,L,Y)$ be as above and assume $Y$ satisfies the admissibility condition~\eqref{eq:adm}. If the set of splitting operators of $L$ commuting with $Y$ is not empty then there exists a unique splitting operator $Y'$ of $L$ such that $(V,L,N,Y,Y')$ is a Deligne-system.
\end{thm}

\begin{proof}
Fix a splitting operator of $L$ commuting with $Y$. We can modify the splitting of $L$ by conjugating by an automorphism $g$ such that $g$ respects $W$ and $(g-1)L_i\subset L_{i-1}$, and consequently, $g$ induces an automorphism on $\gr^L$. We want to achieve that 
\[
\left[N-ge^- g^{-1},ge^+g^{-1}\right]=0,
\]
or equivalently,
\begin{equation}\label{eq:conj}
\left[g^{-1}Ng-e^{-},e^+\right]=0.
\end{equation}
We find $g$ by successive approximations: if $[N_i,e^+]=0$ for $0>i>-k$, we take $g=1+\gamma_{-k}$ for $\gamma_{-k}$ of degree $-k$ with respect to the $L$-grading for $k\geq 1$. Then to make the $k$-th $L$-degree in~\eqref{eq:conj} valid, we need 
\[
\left[-\left[\gamma_{-k},e^-\right]+N_{-k},e^+\right]=0,
\]
which is equivalent to
\begin{equation}\label{eq:adr}
\left( \ad e^+ \right)\left(\left( \ad e^{-}\right) \left(\gamma_{-k}\right)+N_{-k}\right)=0.
\end{equation}
As $k-2\geq -1$, we can write uniquely $N_{-k}=N'+ (\ad e^-) N''$, by the Lefschetz decomposition, such that $N'$ is in the kernel of $ \ad e^+$ and the $\ad H$-degree of $N''$ is $k$ because $N_{-k}$ is of $\ad H$-degree $k-2$. Then~\eqref{eq:adr} becomes
\[
\left( \ad e^+ \right)\left( \ad e^{-}\right) \left(\gamma_{-k}+N''\right)=0.
\]
It follows from the fact that the $H$-degree of $\gamma_{-k}+N''$ is $k$ that $\gamma_{-k}$ has to equal $-N''$. It remains to show that $[\gamma_{-k},Y]=0$, i.e $[N'',Y]=0$. By the admissible condition,
\[
\begin{aligned}
(\ad Y) N_{-k}  =-2N_{-k}.
\end{aligned}
\]
Substituting $N_{-k}$ by $N'+ (\ad e^-) N''$,
\[
(\ad Y) N'+ (\ad Y) (\ad e^-) N'' = (\ad Y)N'+  (\ad e^-)(\ad Y) N''-2 (\ad e^{-})N''=-2N' -2(\ad e^-) N'.
\]
Then we get
\[
(\ad Y+2)N'+(\ad e^{-})(\ad Y )N''=0.
\]
Applying $(\ad e^{-})^{k-1}$ yields 
\[
(\ad e^{-})^{k}(\ad Y )N''=0,
\]
which forces $(\ad Y) N''=0$. This completes proof.
\end{proof}

The \textit{morphisms} of a pair of Deligne-systems $(V,L,N,Y,Y')$ and $(\widehat V,\widehat L,\widehat N,\widehat Y,\widehat Y')$ are the operators $T\in \Hom(V,\widehat V)$ such that $\widehat Y T=T Y$, $\widehat N T=T N$ and $T L \subset \widehat L$ for all $i$. In fact, the morphisms of Deligne-systems are functoral:

\begin{cor}\label{cor:fun}
If $T$ is a morphism of a pair of Deligne-systems 
\[
(V,L,N,Y,Y') \quad \text{and} \quad (\widehat V,\widehat L,\widehat N,\widehat Y,\widehat Y'),
\] 
then $\widehat Y'T=T\widehat Y'$.
\end{cor}

\begin{proof}
Let $T=\sum_{i\leq 0} T_i$ be the $\ad Y'$-homogeneous decomposition of $T$. Then the $H$ degree of $T_i$ is $-i$ because $\widehat Y T=T Y$. Suppose that $T_i$ vanishes for $i=-1,2,\dots,-k+1$. Then $(\ad N)T=0$ gives 
\[
[N_0,T_{-k}]+[N_{-k},T_0]=0.
\]
It follows that $( \ad e^+)( \ad e^-)T_{-k}$ vanishes since
\[
( \ad e^+)( \ad e^-)T_{-k}=[e^+,[e^{-},T_{-k}]]=[e^{+},[T_0,N_{-k}]]=[[e^{+},T_0],N_{-k}]+[T_0,[e^+,N_{-k}]]
\]
and $[e^+,T_0]=[e^+,N_{-k}]=0$. Then $T_{-k}$ must vanish because the $H$-degree of $T_{-k}$ is $k>0$.
\end{proof}

Finally we can give 
\begin{proof}[Proof of Lemma~\ref{lem:relsplitfun}]
By~\cite[1.5]{Saito-MHM}, we have a canonical splitting 
\[
\gr^W_k\gr^V_\alpha\cM\simeq \bigoplus_{i\in\Z}\gr^W_k\gr^L_i \gr^V_\alpha\cM.
\]
If we set $(V,L,N)=(\gr^W\gr^V_\alpha\cM,L\gr^W\gr^V_\alpha\cM,\theta-\alpha)$ and $Y=i$ on $\gr^W_i\cM$, then we can apply Theorem~\ref{thm:delsy} to this situation: there exists a unique splitting operator $Y'$ for $L$ such that $(V,L,N,Y,Y')$ is a Deligne-system. As a consequence, for any local defining equation $f$ of $Z$, it follows from Corollary~\ref{cor:fun} the induced morphism
\[
f: \gr^W\gr^V_\alpha\cM\rightarrow \gr^W\gr^V_{\alpha-1}\cM
\]
commute the splitting operator $Y'$ which concludes $(a)$. 

For part $(b)$, it is easy to see that the morphism $\gr^WT$ is a morphism of Deligne's systems $(\gr^W\gr^V_\alpha\cM,L\gr^W\gr^V_\alpha\cM,\theta-\alpha)$ and $(\gr^W\gr^V_\alpha\cM',L\gr^W\gr^V_\alpha\cM',\theta-\alpha)$. Then by Corollary~\ref{cor:fun}, $\gr^WT$ commutes with the splitting operator $Y'$ which concludes $(b)$.
\end{proof}

\section{Fourier Transform for Monodromic Mixed Hodge Modules}\label{sec:fourier}

\subsection{Notation}
Let $M \in \MHM(E)$ be a mixed Hodge module on the trivial vector bundle $E = X\times \A^r$ over $X$. Let $z_1,\dots, z_r$ be coordinates on the $\A^r$ term of the vector bundle. Assume $M$ is \emph{monodromic} along the $z_1,\dots, z_r$, i.e., the underlying $\cD_E$-module $\cM$ is monodromic. Then
\[ \cM = \bigoplus_{\chi \in \Q} \cM^\chi,\]
where $\cM^\chi$ is the subspace on which $\theta - \chi +r$ acts nilpotently.

Let $E^\vee = X\times \A^r$ with coordinates $w_1,\dots, w_r$ be another trivial bundle, which we think of as the dual of $E$. We consider $\cE = E\times_X E^\vee = X\times \A^{2r}$, with projections $p:\cE \to E$ and $q: \cE \to E^\vee$.

Consider the mixed Hodge module on $\cE$ given by $p^!(M)[-r]$. The underlying $\cD_{\cE}$-module is isomorphic to $\cM[w]$. Let $g = \sum_{i=1}^r z_i w_i$, a function on $\cE$, which is the natural pairing between $E$ and its dual $E^\vee$. Let $\Gamma: \cE \to \cE\times \A^1$ be the graph embedding along $g$, with coordinate $\xi$ on $\A^1$. We consider $\Gamma_+(\cM[w]) = \bigoplus_{j\geq 0} \cM[w]\otimes \de^j$, as our goal is to compute the vanishing cycles of $\cM[w]$ along $g$.

Recall the action on the graph embedding is as follows:
\[
\begin{aligned}
 P(mw ^\alpha \otimes \de^j) & = P(mw^\alpha)\otimes \de^j \text{ for all } P \in \cD_X + \cO_{\cE} \\
 \de_{z_i}(m w^\alpha \otimes \de^j) &= \de_{z_i}(m)w^\alpha \otimes \de^j - m w^{\alpha+e_i} \otimes \de^j, \\
 \de_{w_i} (mw ^\alpha \otimes \de^j) &= \alpha_i m w^{\alpha-e_i} \otimes \de^j - z_i m w^\alpha \otimes \de^j, \\
 \xi(mw ^\alpha \otimes \de^j) &= gm w^\alpha \otimes \de^j - jmw ^\alpha \otimes \de^{j-1}, \\
 \de(mw ^\alpha \otimes \de^j) &= mw ^\alpha \otimes \de^{j+1}.
\end{aligned}
\]

In particular, if $\widetilde{\theta_z}$ is defined as
\[ \widetilde{\theta_z}(m w^\alpha \otimes \de^j) = \theta_z(m) w^\alpha \otimes \de^j,\]
then
\[ \theta_z (mw^\alpha \otimes \de^j) = (\widetilde{\theta_z} - (j+1) - \xi \de)(mw^\alpha \otimes \de^j).\]

Similarly,
\[ \theta_w(mw^\alpha \otimes \de^j) = (|\alpha| - (j+1) - \xi \de)(mw^\alpha \otimes \de^j).\]

In particular, if we look at the commuting operators
\[ T : = \theta_z + \xi \de +1, \quad S : = \theta_w + \xi \de +1,\]
we see that $\Gamma_+(\cM[w])$ breaks into simultaneous eigenspaces for these operators. Indeed, the equality $(T-\lambda)^a (m w^\alpha \otimes \de^j) = 0$ for $m\in \cM^{\chi}$ holds iff $(\theta_z- j - \lambda)^a(m) = 0$, which itself is true iff $\chi = j+ \lambda + r$. Similarly, $(S-\ell)(mw^\alpha \otimes \de^j) =0$ iff $|\alpha| = j +\ell$. For any element $m w^\alpha \otimes \de^j$ (where $m$ is homogeneous in $\cM$, i.e., $m\in \cM^\chi$ for some $\chi\in \Q$), there always exists some $\lambda$ and $\ell$ for which these conditions are true, hence $\Gamma_+(\cM[w])$ breaks into these simultaneous eigenspaces. We shift these eigenspaces for convenience of notation later, and denote
\[ E_{\beta,\ell} = \sum_{\alpha \in \N^r, |\alpha|\geq -\ell} \cM^{\beta+|\alpha|+\ell} w^\alpha \otimes \de^{|\alpha|+\ell},\]
and it will be useful to pull out specifically the term involving $\de^j$, so we denote this as
\[ F_{\beta,\ell}^j = \sum_{|\alpha| = j-\ell} \cM^{\beta+j} w^\alpha \otimes \de^j \subseteq E_{\beta,\ell},\]
so that
\[ 
E_{\beta,\ell} = \bigoplus_{j\geq \ell} F_{\beta,\ell}^j.
\]
These eigenspaces are mapped to one another via the elements of $\cD_{\cE\times \A^1}$ in the following way
\[
\begin{aligned}
 z_i E_{\beta,\ell} &\subseteq E_{\beta+1,\ell} \\
 \de_{z_i} E_{\beta,\ell} &\subseteq E_{\beta-1,\ell} 
\end{aligned}
\quad 
\begin{aligned}
 w_i E_{\beta,\ell} &\subseteq E_{\beta,\ell-1} \\
 \de_{w_i} E_{\beta,\ell} &\subseteq E_{\beta,\ell+1} 
\end{aligned}
\quad 
\begin{aligned}
 \xi E_{\beta,\ell} &\subseteq E_{\beta+1,\ell-1} \\
 \de E_{\beta,\ell} &\subseteq E_{\beta-1,\ell+1}.
\end{aligned}
\]

Let $V^\bullet \Gamma_+(\cM[w])$ be the $V$-filtration along $g$. Then, for any $\lambda \in \Q$, $V^\lambda \Gamma_+(\cM[w])$ is invariant under both $T$ and $S$, hence also decomposes into its eigenspace decomposition, so we write
\[ V^\lambda \Gamma_+(\cM[w]) = \bigoplus_{\beta\in \Q,\ell \in \Z} E^\lambda_{\beta,\ell}, \quad \text{where} E^\lambda_{\beta,\ell} = E_{\beta,\ell}\cap V^\lambda.\]

\subsection{Fourier-Laplace Transform}
The \emph{Fourier-Laplace} transform of a $\cD_E$-module $\cM$ is a $\cD_{E^\vee}$-module $\mathsf{FL}(\cM)$ which, as a $\cD_X$-module, is the same as $\cM$, and such that
\[ w_i m = -\de_{z_i} m, \quad \de_{w_i} m = z_i m.\]

If $\cM$ is monodromic along the $z_1,\dots, z_r$, then $\mathsf{FL}(\cM)$ is monodromic along the $w_1,\dots, w_r$. Indeed,
\[ (\theta_w - \chi +r)m = (\sum w_i \de_{w_i} - \chi +r)m = (-\sum \de_{z_i}z_i -\chi +r)m = -(\theta -(r-\chi)+r)m,\]
and so we see that, in terms of their monodromic decomposition,
\[ \mathsf{FL}(\cM)^{r-\chi} = \cM^\chi.\]

The main goal of this section is to prove Theorem \ref{thm:fourier}:
\fourier*

We want to compute $\sigma^* \phi_g(\cM[w])$, which involves understanding the $V$-filtration along the zero section $\sigma: E^\vee \to \cE$ by Theorem \ref{thm:main}. The first step will be to show that $\phi_{g,\lambda}(\cM[w])$ is monodromic along $z_1,\dots, z_r$ for all $\lambda \in [0,1)$. This will allow us (by Property \ref{monoVFilt}) to compute the associated graded pieces of the $V$-filtration along the zero section.

\begin{lem} \label{lem:bimono} Using the notation in the previous section, we have
\[ \gr^\lambda_V \Gamma_+(\cM[w]) = \bigoplus_{\beta,\ell} E^\lambda_{\beta,\ell}/E^{>\lambda}_{\beta,\ell}.\]

For any $\lambda \in \Q$, $\gr^\lambda_V \Gamma_+(\cM[w])$ is monodromic along $z_1,\dots,z_r$ and along $w_1,\dots, w_r$, and its $\chi$th homogeneous piece in the decomposition is
\[(\gr^\lambda_V \Gamma_+(\cM[w]))^\chi_z = \bigoplus_{\ell \in \Z} E^\lambda_{\chi+\lambda,\ell}/E^{>\lambda}_{\chi+\lambda,\ell}\]
\[(\gr^\lambda_V \Gamma_+(\cM[w]))^\chi_w = \bigoplus_{\beta \in \Q} E^\lambda_{\beta,r-\lambda-\chi}/E^{>\lambda}_{\beta,r-\lambda-\chi}\]
for the $z_i$'s and $w_i$'s respectively.
\end{lem}
\begin{proof} This is easy to see, using the fact that $N = \xi \de - \lambda +1$ is nilpotent on $\gr_V^\lambda \Gamma_+(\cM[w])$. Then we can write
\[ T = \theta_z + \xi\de +1 = \theta_z +N + \lambda\]
\[ S = \theta_w + \xi\de+1 = \theta_w +N + \lambda,\]
and $N$ commutes with $\theta_z$ and $\theta_w$. Hence, $(\theta_z - \chi +r)$ is nilpotent iff $T - \lambda - \chi +r$ is nilpotent. Similarly, $(\theta_w - \chi +r)$ is nilpotent iff $S - \lambda - \chi +r$ is nilpotent.

Then use the fact that $E_{\beta,\ell}$ is the simultaneous eigenspace for $T$ with eigenvalue $\beta-r$ and $S$ with eigenvalue $-\ell$.
\end{proof}

In terms of computing $\cH^0$ of $\sigma^*$, we are interested in the $z$-monodromic pieces corresponding to $\chi = 0$ and $1$. Using Theorem \ref{thm:main}, we have
\[
\begin{aligned} 
    \cH^0\sigma^*(\phi_g \Gamma_+\cM[w]) &= \bigoplus_{\lambda \in [0,1)} \cH^0\sigma^*(\phi_{g,\lambda} \Gamma_+ \cM[w])  \\
    &= \bigoplus_{\lambda \in [0,1)} \text{coker}\left(\bigoplus_{\ell \in \Z, 1\leq i \leq r} E^\lambda_{\lambda+1,\ell}/E^{>\lambda}_{\lambda+1,\ell} \xrightarrow[]{\de_{z_i}} E^\lambda_{\lambda,\ell}/E^{>\lambda}_{\lambda,\ell}\right).
\end{aligned}
\]

Note that this cokernel is monodromic along the $w_i$'s. Using the lemma, the $r-\chi$th monodromic piece is
\[\text{coker}\left(\bigoplus_{i=1}^r E^\lambda_{\lambda+1,\chi-\lambda}/E^{>\lambda}_{\lambda+1,\chi-\lambda} \xrightarrow[]{\de_{z_i}} E^\lambda_{\lambda,\chi-\lambda}/E^{>\lambda}_{\lambda,\chi-\lambda}\right)\]
where $\lambda \in [0,1)$ is the fractional part of $\chi$ (because $\chi -\lambda$ must be an integer for it to be an eigenvalue of $S$).

We are interested in the $r-\chi$th monodromic piece because we have $\mathsf{FL}(\cM)^{r-\chi} = \cM^\chi$ as sets. Thus, we would like to construct an isomorphism
\begin{equation} \label{desirediso} \text{coker}\left(\bigoplus_{i=1}^r E^{\lambda}_{\lambda+1,\chi - \lambda}/E^{>\lambda}_{\lambda+1,\chi-\lambda} \xrightarrow[]{\de_{z_i}} E^{\lambda}_{\lambda,\chi-\lambda}/E^{>\lambda}_{\lambda,\chi-\lambda}\right) \to \cM^\chi \end{equation}

The eigenspaces $E_{\beta,\ell}$ have natural morphisms defined as follows: let $\varphi_{\beta,\ell}: E_{\beta,\ell} \to \cM^{\beta+\ell}$ be such that
\[ \sum m_\alpha w^\alpha \otimes \de^{|\alpha|+\ell} \mapsto (-1)^{\ell} \sum_{\alpha \in \N^r} \de_z^{\alpha}(m_\alpha).\]

Note that, by definition of $E_{\beta,\ell}$, the coefficient $m_\alpha$ lies in $\cM^{\beta+\ell + |\alpha|}$, so $\de_z^{\alpha}(m_\alpha) \in \cM^{\beta+\ell}$ and so the map lies in the correct eigenspace of $\cM$. This map should be thought of as evaluating $w_i$ at $-\de_{z_i}$ (which is how the Fourier transform behaves), and also $\de$ at $-1$. Then for any $\beta, \ell$, the map $\varphi_{\beta,\ell}$ is $\cD_X$-linear, and satisfies the following easy to check relations
\[
\begin{aligned} 
 \varphi_{\beta-1,\ell}\circ \de_{z_i} &= 0  \\
 \varphi_{\beta,\ell-1} \circ w_i &= -\de_{z_i} \circ \varphi_{\beta,\ell} \\
 \varphi_{\beta,\ell+1} \circ \de_{w_i} &= z_i \circ \varphi_{\beta,\ell} \\
 \varphi_{\beta+1,\ell-1} \circ \xi &= -(\theta - \ell + r)\circ \varphi_{\beta,\ell} \\
 \varphi_{\beta-1,\ell+1} \circ \de &= - \varphi_{\beta,\ell}.
\end{aligned}
 \]

In fact, the first relation can be strengthened:
\begin{lem} \label{kernel} If $\ell \geq 0$, then
\[ \ker(\varphi_{\beta,\ell}) = \sum_{i=1}^r \de_{z_i}E_{\beta+1,\ell}.\]
\end{lem}
\begin{proof} The containment $\supseteq$ is clear from the first relation above.

For the other containment, let $\eta = \sum_{|\alpha| =0}^{a} m_\alpha w^\alpha \otimes \de^{|\alpha|+\ell}$ lie in the kernel. We induce on $a$. If $a=0$, then $\varphi(m_0) = (-1)^\ell m_0 =0$ implies $m_0 =0$, so this is obvious.

For $a>0$, assume we know it for any sum of elements with $|\alpha| \leq a-1$. By definition, we know $(-1)^\ell \sum_{0\leq |\alpha|\leq a} \de_z^{\alpha}(m_\alpha) = 0$. Order the $\alpha \in \N^r$ with $|\alpha| = a$ in some way, write them as $\alpha^{(1)},\dots,\alpha^{(d)}$. As $a>0$, there exists some $i_1$ such that $\alpha^{(1)}_{i_1} > 0$. Define $m_{\alpha^{(1)} - e_{i_1}}^{(1)}:= m_{\alpha^{(1)} - e_{i_1}} + \de_{z_{i_1}}(m_{\alpha^{(1)}})$ and let $m_{\beta}^{(1)} = m_\beta$ for all other $|\beta| = a-1$. Repeat with $\alpha^{(2)}$, finding some $i_2$ with $\alpha^{(2)}_{i_2}>0$ and adjusting the $m_{\beta}^{(1)}$ using $\de_{z_{i_2}}(m_{\alpha^{(2)}})$. Eventually, this terminates, as there are only finitely many $\alpha \in \N^r$ with $|\alpha| = a$. Let $m'_\beta :=m_\beta^{(d)}$ for all $\beta \in \N^r$ with $|\beta| = a-1$.

Now, define $\eta' = \sum_{0\leq |\alpha| \leq a-1} m'_\beta w^\beta \otimes \de^{|\beta|+\ell}$. Then this still lies in the kernel of $\varphi$ (it has the same value as that for $m$, which we assumed is 0). By induction, it lies in $\sum \de_{z_i}(E_{\beta+1,\ell})$.

Finally, use $\eta = \eta' + \sum_{|\alpha| = a} \de_{z_{i_\alpha}}(m_\alpha w^{\alpha -e_{i_\alpha}} \otimes \de^{|\alpha|+\ell -1})$ to conclude.
\end{proof}

Let us now restrict to looking at $E^{\lambda}_{\lambda,\ell}$, for $\lambda \in [0,1)$. Restricting $\varphi_{\lambda,\ell}$ to this subspace, we get a map
\[ E^{\lambda}_{\lambda,\ell} \to \cM^{\lambda+\ell}.\]

To show that this map $\varphi_{\lambda,\ell}$ is the desired isomorphism \ref{desirediso}, it suffices to show that this map is surjective and its kernel is precisely
\[ \sum_{i=1}^r \de_{z_i}(E_{\lambda+1,\ell}^{\lambda}) + E^{>\lambda}_{\lambda,\ell}.\]

For this, we must have an understanding of the $V$-filtration along $g$ of $\Gamma_+ \cM[w]$. However, we can restrict to proving it is an isomorphism for $\ell \geq 0$, using the fact that the Koszul-like complexes for monodromic $\cD$-modules are acyclic. Indeed, for $k<0$, look at the morphism of complexes
\[ \begin{tikzcd} \cN^{\lambda +k+r} \ar[r,"w"] \ar[d,"\varphi"] & \bigoplus_{i=1}^r \cN^{\lambda+k+r-1} \ar[r, "w"] \ar[d,"\varphi"] & \dots \ar[r,"w"] & \cN^{\lambda +k} \ar[d,"\varphi"]\\ 
\cM^{\lambda +k+r} \ar[r,"\de_z"]  & \bigoplus_{i=1}^r \cM^{\lambda+k+r-1} \ar[r, "\de_z"]  & \dots \ar[r,"\de_z"]& \cM^{\lambda + k}. \end{tikzcd}\]

Hence, if by induction all but the rightmost vertical map is an isomorphism, then so is the rightmost map. Here the terms in the top row are the corresponding domains of the maps in Formula \ref{desirediso}.

\subsection{Computation of the $V$-filtration along $\xi$}
To compute the $V$-filtration along $\xi$, it suffices to break up $\cM = \bigoplus_{\lambda \in [0,1)} \bigoplus_{\ell \in \Z} \cM^{\lambda+\ell}$, and handle each $\lambda \in [0,1)$ separately.

Recall the notation
\[ F^j_{\beta,\ell} = \bigoplus_{|\alpha| = j-\ell} \cM^{\beta+j} w^\alpha \otimes \de^j.\]

We use the following easy to prove facts
\begin{align}
\sum_{i=1}^r w_i F_{\chi,\ell}^j &= F_{\chi,\ell-1}^j \label{waction}\\
\chi+j \neq r-1 &\implies \sum_{i=1}^r z_i F_{\chi,\ell}^j = F_{\chi+1,\ell}^j  \label{zaction}\\
\chi+j \neq r-1 &\implies \sum_{i=1}^r \de_{w_i} F_{\chi,j}^j = F_{\chi,j+1}^{j+1} \label{dwaction}
\end{align}
The first is trivial from the definition, and the latter two use the fact that $\sum z_i \cM^{\chi} = \cM^{\chi+1}$ whenever $\chi +1\neq r$, by the acyclicity of the Koszul-like complex for monodromic modules.

Throughout, we will define filtrations $U^\bullet$ only for $\bullet \in [0,1]$. They are extended as
\[ U^{\lambda} = \xi^j U^{\lambda-j} \text{ for } \lambda >1, \lambda -j \in (0,1],\]
and inductively,
\[ U^{\lambda} = \de^k U^{\lambda+k} + U^{>\lambda} \text{ for } \lambda<0, \lambda +k \in [0,1).\]

In this way, the only conditions we must check for the $V$-filtration are that the filtration defined in this way is exhaustive and 

\begin{enumerate}[label=\textbf{V.\arabic*}] 
\item \label{v1}  $U^\lambda$ is coherent over $V^0\cD_{\cE\times \A^1}$ for $\lambda \in[0,1]$,
\item \label{v2}$U^0\supseteq U^\lambda \supseteq U^{\lambda'} \supseteq U^1$ for all $\lambda' >\lambda$,
\item\label{v3} $\xi U^0 \subseteq U^1$,
\item\label{v4}  $\de U^1 \subseteq U^0$,
\item\label{v5} $(\de \xi - \lambda)^a U^\lambda \subseteq U^{>\lambda}$ for all $\lambda \in [0,1)$.
\end{enumerate}

Indeed, note that we do not have to check $(\de\xi -1)^a U^1 \subseteq U^{>1}$ if we know the facts preceding it. This is because $U^{>1} = \xi U^{>0}$, and so we would have to show $(\de \xi-1)^a U^1 = (\xi \de)^a U^1 \subseteq \xi U^{>0}$, which is true if $(\de \xi)^{a-1}(\de U^1)\subseteq U^{>0}$, but $\de U^1\subseteq U^0$, and so we are done.

\textit{Case 1}: $\lambda = 0$.

Define \[U^0 : = V^0\cD_{\cE\times \A^1} \cdot F_{0,0}^0 + V^0\cD_{\cE\times \A^1} \cdot F_{0,r}^r\] and
\[ U^1 : = V^0 \cD_{\cE\times \A^1} \cdot F_{1,0}^0 + V^0\cD_{\cE\times \A^1} \cdot F^{r-1}_{1,r-1}.\]

Exhaustive: let $\cU = \bigcup_{k \in \Q} U^k$. It suffices (as $\cU$ is closed under $\de$ and multiplication by $w_1,\dots, w_r$) to show that $\cM^\ell \otimes 1 \subseteq \cU$ for all $\ell \in \Z$. By definition, $F_{0,0}^0 = \cM^0 \otimes 1 \subseteq U^0 \subseteq \cU$. Hence, by multiplying by $z_1,\dots, z_r$, we conclude that $\cM^1\otimes 1,\dots, \cM^{r-1}\otimes 1 \subseteq U^0 \subseteq \cU$.

Also, inductively, we conclude that $\cM^{\ell} \otimes 1 \subseteq \cU$ for all $\ell \leq 0$. Indeed, if $\cM^{\ell+1}\otimes 1 \subseteq \cU$, then $\cM^{\ell+1}[w] \otimes \de \subseteq \cU$. Also, $\de_{z_i}(\cM^{\ell+1}\otimes 1) \subseteq \cU$, so we get
\[ (\sum \de_{z_i} \cM^{\ell+1}) \otimes 1 \subseteq \cU,\]
and we use acyclicity of the Koszul-like complex in the $\de_{z_i}$'s to conclude that $\cM^{\ell} \otimes 1 \in \cU$.

Finally, $F_{0,r}^r = \cM^r \otimes \de^r \subseteq \cU$. By applying $(\xi -g)^r$, we see that $\cM^r \otimes 1 \subseteq \cU$, and then by applying the $z_i$'s and using acyclicity of the Koszul like complex, we get $\cM^{r+\ell}\otimes 1 \subseteq \cU$ for all $\ell \geq 0$, proving  that the filtration is exhaustive.

\ref{v1}: To see $U^0$ is finitely generated over $V^0\cD_{\cE\times \A^1}$, let $m_1,\dots, m_N$ be finitely many $\gr^0_V\cD_{E}$ generators of $\cM^0$ and let $\eta_1,\dots, \eta_M$ be generators for $\cM^r$ over $\gr^0_V\cD_{E}$. Then these elements generate $U^0$, by the following fact: given $m\otimes \de^k \in U^0$, we obtain $(\gr_V^0\cD_{E}\cdot m)\otimes \de^k \subseteq U^0$. Indeed, we easily get $\cD_X \cdot m \otimes \de^k$, and to get $z_i \de_{z_j}(m) \otimes 1$, we use
\[ z_i \de_{z_j}(m) \otimes 1 = z_i \de_{z_j}(m \otimes 1) + w_j \de_{w_i}(m\otimes 1),\]
which lies in $U^0$. The same proof works for $U^1$.

For the remaining conditions, we use the following lemma
\begin{lem} \label{VFiltContain} We have containment $F_{\chi,\ell}^j \subseteq U^0$ for any triple satisfying either of the two conditions
\begin{itemize}
    \item  $\chi \geq 0, j \geq r, \ell \leq j$.
    \item  $0 \leq \chi \leq r-1, 0 \leq j < r-\chi, \ell \leq j.$
\end{itemize}
Also, we have $F_{\chi,\ell}^j \subseteq U^1$ in either of the following cases:
\begin{itemize}
    \item $\chi \geq 1, j\geq r-1, \ell \leq j$
    \item $1 \leq \chi \leq r-1, 0 \leq j \leq r-\chi, \ell \leq j.$
\end{itemize}
In particular, we have
\begin{itemize}
    \item  $E_{\chi,\ell} \subseteq U^0$ for all $\chi\geq 0,\ell \geq r$,
    \item $E_{\chi,\ell} \subseteq U^1$ for all $\chi \geq 1, \ell \geq r-1$,
    \item  $E_{1,\ell}\subseteq U^1$ for all $\ell\leq 0$.
\end{itemize}
\end{lem}
\begin{proof} We make use of the fact that $U^0$ is closed under $z_i,\de_{w_i}$ and $w_i$ for all $i$.

Starting from $F_{0,r}^r \subseteq U^0$, by Formula \ref{zaction} we get $F_{\chi,r}^r\subseteq U^0$ for all $\chi \geq 0$. Then by Formula \ref{dwaction} we get $F_{\chi,j}^j\subseteq U^0$ for all $j\geq r$. Finally, by Formula \ref{waction} we get $F_{\chi,\ell}^j\subseteq U^0$ for all $\chi\geq 0, j\geq r, \ell \leq j$.

Starting from $F_{0,0}^0\subseteq U^0$, we get by Formula \ref{zaction} $F_{\chi,0}^0\subseteq U^0$ for all $0\leq \chi \leq r-1$. By applying Formula \ref{dwaction} we get $F_{\chi,j}^j\subseteq U^0$ for all $0\leq \chi \leq r-1$ and $\chi+j < r$. Finally, applying Formula \ref{waction}
we get $F_{\chi,\ell}^j\subseteq U^0$ for all $0 \leq \chi \leq r-1, \chi +j < r$ and $\ell \leq j$.

Similarly, we argue for the containment of the other subsets in $U^1$.

The last statements follow easily from these containments. For example, let $\ell \leq 0$, then $F_{1,\ell}^j \subseteq U^1$  for all $j \geq 0$. Indeed, if $j\geq r-1$, then this comes from the fact that $F_{1,r-1}^{r-1}$ is contained as argued above. If $0 \leq j < r-1$, then in particular, $1+j = \chi + j < r$, so this follows from the fact that $F_{1,0}^0$ is contained as argued above.
\end{proof}

\ref{v2}: Obvious, from the lemma. Also, $F_{1,r-1}^{r-1} = (\xi - g) F_{0,r}^r$.

\ref{v3} Indeed, $\xi F_{0,0}^0 \subseteq E_{1,-1}$ and $\xi F_{0,r}^r \subseteq E_{1,r-1}$, so this follows from the lemma.

\ref{v4} Indeed, $\de F_{1,0}^0 = F_{0,1}^1 \subseteq E_{0,1}$, which is in $U^0$ by the lemma, and $\de F_{1,r-1}^{r-1} = F_{0,r}^r$, which is in $U^0$ by definition.

\ref{v5} Note that $\varphi_{0,0} \circ (\de \xi)^a = (\theta +r)^a \circ \varphi_{0,0}$, so since $\varphi_{0,0}$ has image in $\cM^0$, $(\theta+r)^a$ kills this for $a>>0$. Similarly, $\varphi_{0,r} \circ (\de\xi)^a = (\theta-r +r)^a \circ \varphi_{0,r}$, and $(\theta-r+r)^a$ kills $\cM^r$ for $a>>0$. Thus, we see that $(\de \xi)^a$ multiplies $F_{0,0}^0$ and $F_{0,r}^r$ into $\ker(\varphi_{0,0})$ and $\ker(\varphi_{0,r})$, respectively. Well, by Lemma \ref{kernel}, these are
\[ \sum_{i=1}^r \de_{z_i}(E_{1,0}) \quad \text{and} \quad \sum_{i=1}^r \de_{z_i}(E_{1,r})\]
respectively, and both of these are contained in $U^1$ by the lemma and the fact that $U^1$ is closed under $\de_{z_i}$ action.

This finishes the proof and shows that $U^\bullet = V^\bullet$ is the $V$-filtration along $\xi$.

\textit{Case 2}: $\lambda \in (0,1)$.

Define $U^0 = U^\lambda := V^0\cD_{\cE\times \A^1}\cdot F_{\lambda,0}^0$ and $U^1 : = V^0\cD_{\cE\times \A^1}\cdot F_{\lambda+1,0}^0$.

Exhaustive: As $F_{\lambda,0}^0 = \cM^\lambda \otimes 1$, the fact that the filtration is exhaustive is shown in exactly the same way as above (using the acyclicity of the Koszul-like complex).

\ref{v1}: By taking finitely many $\gr_V^0\cD_{E}$ generators of $\cM^\lambda$ and $\cM^{\lambda+1}$, we see that $U^\bullet$ are $V^0\cD_{\cE\times \A^1}$-coherent.

\ref{v2}: Obvious,, using the relation \ref{zaction} above.

In a similar way to the lemma above, we see that $F^j_{\lambda+b,\ell} \subseteq U^0$ and $F^j_{\lambda+1+b,\ell}\subseteq U^1$ for all $b\geq 0, j\geq 0$ and $\ell \leq j$.

\ref{v3}, \ref{v4}: Note that $\xi F_{\lambda,0}^0 \subseteq E_{\lambda+1,-1}$, which is contained in $U^1$ by the previous observation. Similarly, $\de F_{\lambda+1,0} \subseteq E_{\lambda,1}$ which is contained in $U^0$ by the previous observation.

\ref{v5}: Finally, we need only check $(\de \xi -\lambda)^a U^\lambda \subseteq U^1$ for some $a>>0$. Just as before, $(\de \xi -\lambda)^a$ multiplies $F_{\lambda,0}^0$ into $\ker(\varphi_{\lambda,0}) = \sum \de_{z_i}(E_{\lambda+1,0})$. By the above, this is contained in $U^1$, as desired.

This completes the proof that this is indeed the $V$-filtration along $\xi$.

\subsection{Showing the isomorphism}
Fix $\lambda \in [0,1)$ and $\ell \in \Z_{\geq 0}$. We show that
\[ \varphi: E^\lambda_{\lambda,\ell} \to \cM^{\lambda+\ell}\]
is surjective, with kernel equal to $\sum \de_{z_i}(E^\lambda_{\lambda+1,\ell}) + E^{>\lambda}_{\lambda,\ell}$.

Surjectivity is easy: we showed above that, either when $\lambda =0$ or when $\lambda \in (0,1)$, $F_{\lambda,\ell}^\ell = \cM^{\lambda+\ell} \otimes \de^\ell \subseteq V^\lambda$. Under $\varphi_{\lambda,\ell}$, this hits all of $\cM^{\lambda+\ell}$.

Using Lemma \ref{kernel}, we know that, before restricting to $E^\lambda_{\lambda,\ell}$, the kernel of $\varphi_{\lambda,\ell}$ (for $\ell \geq 0$, which we have reduced the problem to) is precisely
\[ \sum \de_{z_i}(E_{\lambda+1,\ell}).\]

Note that we have $E_{\lambda+1,\ell} \subseteq V^\lambda$ for all $\ell \geq 0$. Indeed, for any fixed $\ell \geq 0$, it suffices to show $F_{\lambda+1,\ell}^j \subseteq V^\lambda$ for all $j\geq \ell$. This was already noted in the computation of the $V$-filtration.

Hence, $E^{\lambda}_{\lambda+1,\ell} = E_{\lambda+1,\ell}$, and so all that remains is to show $E^{>\lambda}_{\lambda,\ell}\subseteq \ker(\varphi)$. To see this, note that, in either case, we have $E_{\lambda,\ell}^{>\lambda} = E_{\lambda,\ell}^1$.

Write an arbitrary element $P$ of $V^0 \cD_{\cE \times \A^1}$ as $P = \sum \de_z^\beta \de_w^\gamma z^\rho w^\alpha (\xi \de)^j \xi^k$. We see easily that
\[ P E_{\lambda + 1,\ell} \subseteq E_{\lambda + 1 + |\rho| + k - |\beta|, \ell +|\gamma|-|\alpha|-k}.\]

For our fixed $\lambda \in [0,1)$, recall the explicit generators of $V^1$ given above, they are all of the form $F_{\lambda+1,p}^j$ for some $p \geq 0, j\geq p$. We are interested in $E^1_{\lambda,\ell}$. The only way to get into this eigenspace from one of the form $E_{\lambda+1,p}$ is to involve $\de_z^\beta$ for some $|\beta| \geq 1$ (as $\lambda < \lambda+1$). As we know that elements of this form lie in the kernel, this completes the proof.

\begin{rmk} \label{rmk:comp} In \cite[Formula (10.3.31)]{KS}, the Fourier-Sato transform is defined in the following way: let $q: \cE \to E^\vee$ be the projection, and consider $(q\times id): \cE \times \A^1 \to E^\vee \times \A^1$ with coordinate $t$ on $\A^1$. Then the Fourier-Sato transformation in loc. cit. is defined by
\[ \phi_t \circ (q\times id)_* \circ \Gamma_* \circ p^!,\]
where $\Gamma: \cE \to \cE\times \A^1$ is the graph embedding along $g$. 

We trivially have $\phi_t \circ (q\times id)_* = q_* \circ \phi_t$, as $q$ does not see the $t$ coordinate. Hence, this definition is equivalent to
\[ q_* \circ \phi_t \circ \Gamma_* \circ p^!  = q_* \circ \phi_g  \circ p^!,\]
where, as $g$ does not define a smooth hypersurface, in order to define $\phi_g$ we first apply the graph embedding and then take $\phi_t$, i.e., $\phi_g := \phi_t \circ \Gamma_*$.

Finally, by Remark \ref{rmk:cvvcMHM} and the fact that $\phi_t \Gamma_+ p^!(\cM)$ is monodromic (see Lemma \ref{lem:bimono}), we know that $q_+ \phi_t \Gamma_+ p^!(\cM) = \sigma^* \phi_t \Gamma_+ p^!(\cM)$. Hence, our definition of Fourier-Sato transform is the same as that in \cite{KS}, up to a cohomological shift so that it sends modules to modules in degree $0$.
\end{rmk}

\subsection{Results Concerning the Hodge and Weight Filtrations}
We have just finished showing that the $\cD_{E^\vee}$-module $\cH^0 \sigma^*(\phi_g \cM[w])$ is isomorphic to $\mathsf{FL}(\cM)$. We will now prove the main result concerning the weight filtration and Hodge filtration. For this, we will make use of the results shown at the end of Section \ref{sec:applic}.

We begin by proving that the operator $N = \bigoplus_{\chi \in \Q} \theta - \chi + r$ decreases the weight filtration of any monodromic mixed Hodge module by two. This generalizes the result in the pure case saying that $N =0$.

\begin{thm} Let $(\cM,F,W)$ be a bifiltered $\cD_E$-module underlying a monodromic mixed Hodge module on $E$. Then the weight filtration $W$ on $\cM$ is its own relative monodromy filtration with respect to the nilpotent operator $N$, i.e., $N W_\bullet \cM \subseteq W_{\bullet -2}\cM$.
\end{thm}
\begin{proof} For this, we consider the specialization of $\cM$ along the zero section of the vector bundle $E$. It is easy to see that, as $\cD$-modules, $\Sp(\cM) \cong \cM$ (using the isomorphism $E\cong T_X E$ sending $t_i$ to $z_i$).

Recall that by Lemma \ref{lem:relmonext} the weight filtration on $\Sp(\cM)$ is the relative monodromy filtration with respect to the filtration $\bigoplus_{\chi \in \Q} \gr_V^\chi(W_\bullet \cM)$ and the nilpotent operator $N$. Note that $W_k \cM \subseteq \cM$ is a $\cD$-submodule, hence it is monodromic. Thus, $\gr_V^\chi(W_k \cM) = (W_k \cM)^{\chi}$ in the monodromic decomposition.

So we are looking at the relative monodromy filtration for $N$ and 
\[ L_i \Sp(\cM) : = \bigoplus_{\chi \in \Q} (W_i \cM)^\chi.\]

Note that $\gr_i^L \Sp(\cM) = \bigoplus_{\chi \in \Q} (\gr^W_i \cM)^\chi = \Sp(\gr^W_i \cM)$ is the specialization of the pure monodromic Hodge module $\gr^W_i \cM$, so by Proposition \ref{puremono}, we know that the monodromy filtration is trivial on $\gr_i^L \Sp(\cM)$.

Recall that the relative monodromy filtration induces the monodromy filtration on each graded piece $\gr_i^L \Sp(\cM)$, and so this implies that its restriction to each graded piece is trivial. Finally, we conclude that $W_\bullet \Sp(\cM) = L_\bullet \Sp(\cM)$ using Deligne's canonical splitting of the relative monodromy filtration.
\end{proof}

Using this and knowing that $\FL(\cM)$ is a monodromic mixed Hodge module, we can compute the weight filtration on $\FL(\cM)$.

To begin, we make an easy observation using the fact that the relative monodromy filtration is functorial. Specifically, if $\eta: (A,W,N) \to (A',W',N')$ is a morphism between filtered objects in an abelian category with nilpotent operators $N,N'$, and if the relative monodromy filtrations $M_\bullet A$ and $M'_\bullet A'$ exist, then
\begin{equation} \label{relmonfunct} \eta(M_\bullet A) \subseteq M'_\bullet A'\end{equation}

Recall the notation from the introduction that if $\cM$ is a monodromic $\cD$-module, then $\cM^{\lambda +\Z} = \bigoplus_{\ell \in \Z} \cM^{\lambda + \ell}$ for any $\lambda \in [0,1)$.

\begin{lem} \label{weight} Let $(\cM,F,W)$ be a bifiltered $\cD_E$-module underlying a monodromic mixed Hodge module on $E$. Then
\[ W_k \mathsf{FL}(\cM)^{\lambda+\Z} \subseteq \mathsf{FL}(W_{k+r + \lceil \lambda \rceil} \cM)^{\lambda +\Z}.\]
\end{lem}
\begin{proof} The difficult part is that, in the composition of functors defining $\mathsf{FL}(-)$, we must understand iterated relative monodromy filtrations.

Specifically, we consider $\phi_{g,\lambda} p^!(\cM)[-r]$, with $W^0_\bullet \phi_{g,\lambda} p^!(\cM)[-r] : = \phi_{g,\lambda} p^!(W_{\bullet+r+\lceil \lambda \rceil}\cM)$. Here the $\lceil \lambda \rceil$ term is explained in Formulas \ref{nearbyWeight} and \ref{vanWeight}. Then we take the relative monodromy filtration of this along $N_1: =\de \xi - \lambda$. This gives a filtration $W^1$, which is preserved by the operator $N_2 : = \theta_z +r = \sum_{i=1}^r \de_{z_i} z_i$. By definition, this is the weight filtration on $\phi_{g,\lambda} p^!(\cM)[-r]$, see ~\cite[Sect. 20]{Schnell-MHM}. Then, we take the relative monodromy filtration of $W^1$ and $N_2 = \theta_z +r$ on $\gr_V^0(\phi_{g,\lambda}\cM[w])$, which we call $W^2$.

By definition of the $\cD$-action on the graph embedding, we know that $N_1 + N_2 = \widetilde{N}$, where $\widetilde{N}$ acts on the $\cM$ part of $\Gamma_* \cM[w]$ by $N = \bigoplus_{\chi \in \Q} \theta - \chi +r$. We know that $(\cM,W)$ is its own relative monodromy filtration with respect to the action of $N$, or equivalently, that $N W_{\bullet} \cM \subseteq W_{\bullet -2} \cM$.

Hence, by Formula \ref{relmonfunct}, we know $\widetilde{N} W^1_\bullet \subseteq W^1_{\bullet -2}$, and by applying the formula again, this implies $\widetilde{N} W^2_\bullet \subseteq W^2_{\bullet -2}$.

But also $\widetilde{N} = N_2$ on $\gr^{W^1}_{k}\gr_V^0(\phi_{g,\lambda}\cM[w])$, because $N_1$ kills this associated graded. In particular, we see that $W^2$ is the relative monodromy filtration for $(\widetilde{N},W^1)$. But $\widetilde{N}$ decreases $W^1$ by 2, and so $W^1$ is its own relative monodromy filtration with respect to $\widetilde{N}$. In particular, we get that $W^2 = W^1$, the relative monodromy along $N_1$ and $W^0$.

Now, we have the quotient map $\gr_V^0(\phi_{g,\lambda} \cM[w]) \to \cH^0\sigma^*(\phi_{g,\lambda}\cM[w])$, which induces a filtration $W^1$ and $N_1$ induces a nilpotent operator on the quotient. By functoriality of the relative monodromy filtration in this setting, we get that the weight filtration on the mixed Hodge module $\cH^0 \sigma^*(\phi_{g,\lambda}\cM[w])$ (which is the filtration induced by this quotient) is contained in the relative monodromy filtration of $\cH^0 \sigma^*(\phi_{g,\lambda}\cM[w])$ with respect to the induced filtration by $W^0$ and the induced nilpotent operator from $N_1$. Under the isomorphism $\varphi$ to $\cM$, these map to $W_{\bullet+r} \cM$ and $\bigoplus_{\chi \in \Q} \theta - \chi +r$, and so by the previous theorem, we get the desired containment.
\end{proof}

We now give a computation of the Hodge filtration on the Fourier transform of a monodromic mixed Hodge module. The only difficulty is that there are shifts of the Hodge filtration for certain functors along the way. We refer back to formulas from Section \ref{subsec:conv} for the correct shifts of the Hodge filtration.

We have that $p^!M[-r]$ has underlying $\cD_{\cE}$-module given by $\cM[w]$. In any case, the Hodge filtration is given by Formula \ref{smooth!Hodge}
\[ F_p(\cM[w]) = (F_{p-r} \cM)[w].\]

Now, for the closed embedding $\Gamma_+$, the filtration is given by Formula \ref{closedHodge}:
\[ F_p \Gamma_+(\cM[w]) = \sum_{j\geq 0} F_{p-j-1}(\cM[w]) \otimes \de^j = \sum_{j\geq 0} (F_{p-j-1}\cM)[w] \otimes \de^j.\]

By Formula \ref{nearbyHodge}, there is no shift of the filtration for $\phi_{g,\neq 1}$, and by Formula \ref{vanHodge}, there is a shift of $+1$ for $\phi_{g,1}$.

Finally, the Hodge filtration on $\cH^0 q_+ = \cH^0\sigma^*$ is shifted by $[-r]$ as dictated by Theorem \ref{thm:main}. In summary, an arbitrary element of $F_p E^{\lambda}_{\lambda,\ell}$ is of the form
\[ \sum_{|\alpha|\geq -\ell} m_\alpha w^\alpha \otimes \de^{|\alpha|+\ell},\]
where \begin{equation} \label{Hodgefilt} m_\alpha \in \begin{cases} F_{p-\ell-|\alpha|} \cM^{\lambda+|\alpha|+\ell} & \lambda = 0\\ F_{p -1 - \ell -|\alpha|} \cM^{\lambda+|\alpha|+\ell-1} & \lambda \in (0,1)  \end{cases}\end{equation}

Recall the claim:
\fourier*

So fix $\ell \geq 0$, and we consider $\varphi(F_p E_{\lambda,\ell}^\lambda)$. Note that we have containment $\varphi(F_p E_{\lambda,\ell}^{\lambda}) \subseteq F_{p-\ell} \cM^{\lambda+\ell}$ if $\lambda = 0$ and $\subseteq F_{p-1-\ell}\cM^{\lambda+\ell}$ if $\lambda \in (0,1)$. Indeed, by applying $\varphi$ to an element of the form in Equation \ref{Hodgefilt}, we get
\[ (-1)^{\ell} \sum_\alpha \de_z^\alpha(m_\alpha) \in F_{p-\ell} \cM^{\lambda +\ell}.\]

For the other containment, recall that $\cM^{\lambda+\ell}\otimes \de^\ell \subseteq E_{\lambda,\ell}^\lambda$ in either case $\lambda = 0$ or $\lambda \in (0,1)$. Hence, $F_{p - \ell - \lceil\lambda \rceil} \cM^{\lambda+\ell} \otimes \de^\ell \subseteq F_p E_{\lambda,\ell}^{\lambda}$, and by applying $\varphi$, we get all of $F_{p-\ell -\lceil \lambda \rceil}\cM^{\lambda+\ell}$.

For $\ell <0$, use descending induction on $\ell$ and Theorem  \ref{thm:acyclic}, which tells us that for $\ell <0$,
\[
\sum w_i F_p \FL(\cM)^{r-(\lambda+\ell+1)} = F_p \FL(\cM)^{r-(\lambda+\ell)}.
\]

Now, we introduce the inverse Fourier-Laplace transform, which will help us prove that the containment in Lemma \ref{weight} is an equality. Note, that if we apply the Fourier transform twice, we get the actions
\[ z_i m = -z_i m, \quad \de_{z_i} m = -\de_{z_i} m,\]
and so if $a: E \to E$ is the antipodal map $(x,v)\mapsto (x,-v)$, we see that, in terms of underlying $\cD$-modules,
\[ a^*\FL_{E^\vee} \FL_E(\cM) \xrightarrow[]{\cong} \cM.\]

We recall the definition of the \emph{inverse Fourier transform} from the introduction: recall that $\FL$ is defined to be $\bigoplus_{\lambda \in (0,1]} \cH^0\sigma^* \phi_{g,\lambda} \Gamma_* p^![-r]$. Breaking this up into the case $\lambda = 1$ and $\lambda \in (0,1)$, we define 
\[ \overline{\FL}_E(M) : = \cH^0 \sigma^* \phi_{g,1} \Gamma_+ p^!(M(-r))[-r] \oplus \cH^0 \sigma^* \phi_{g,\neq 1} \Gamma_+ p^!(M(1-r))[-r]\]

Using this, we see that $a^* \overline{\FL}_{E^\vee}\FL_E(\cM) \to \cM$ preserves the Hodge filtration. Indeed, let us check it on monodromic pieces. By definition, for $\lambda = 0$, we have
\[ \overline{\FL} _{E^\vee}\FL_E(\cM)^{r- \ell} = \FL_{E^\vee} \FL_E(\cM(-r))^{r-\ell},\]
so by using our formula for the Hodge filtration, taking $F_p$ of this, we get
\[ F_p \FL_{E^\vee} \FL_E(\cM(-r))^{r-\ell} = F_{p+r-\ell} \FL_E(\cM)^{\ell} \] 
\[= F_{p+r-\ell} \FL_E(\cM)^{r-(r-\ell)} = F_{p+r - \ell -(r-\ell)} \cM^{r-\ell} = F_{p}\cM^{r-\ell}.\]

Now, for $\lambda \in (0,1)$, we have
\[ F_p a^* \overline{\FL}_{E^\vee}\FL_E(\cM)^{r-(\lambda+\ell)} = F_{p+r-1} \FL_{E^\vee} \FL_E(\cM)^{r-(\lambda+\ell)},\]
because of the Tate twist in the definition of $\overline{\FL}$. Hence, by the theorem, this is equal to
\[ F_{p+r-1-\ell-1} \FL_E(\cM)^{\lambda+\ell} = F_{p+r-1-\ell-1}\FL_E(\cM)^{r - (r-(\lambda+\ell)}.\]

If we want to apply the theorem again, we need to write $r-(\lambda+\ell)$ as a number in $(0,1)$ plus some integer. Well, this can be written as $(1-\lambda) + (r-1-\ell)$. Hence, applying the theorem once more, we get
\[ F_{p+r-1-\ell-1-(r-1-\ell)-1} \cM^{r-(\lambda+\ell)} = F_{p}\cM^{r-(\lambda+\ell)}.\]

Moreover, by Lemma \ref{weight}, this map is compatible with the weight filtration (again, using the Tate twist by $(1)$ on the $\lambda \neq 0$ part). Finally, in terms of the $\Q$-structure, it is shown in ~\cite[Theorem 3.7.12(i)]{KS} that this morphism is an isomorphism of $\Q$-structure (see also ~\cite[Prop. 6.13]{Bry}). Hence, it is an isomorphism of mixed Hodge modules. This gives the equality in Lemma \ref{weight} and in the statement of Theorem \ref{thm:fourier}.

Finally, we describe how the functor $\FL$ behaves when taking duals of mixed Hodge modules. Using the fact that
\[ \bD \circ \cH^0\sigma^* = \cH^0 \sigma^! \circ \bD,\]
\[ \bD \circ\phi_{g,1} = \phi_{g,1}\circ \bD\]
\[ \bD \circ \phi_{g,\neq 1} = \phi_{g,\neq 1}(1) \circ \bD\]
and
\[ \bD \circ p^![-r] = p^*[r]\circ \bD,\]
we have
\[ \bD_{E^{\vee}} \circ \FL_E = \bD_{E^{\vee}} \circ \cH^0 \sigma^* \phi_g \Gamma_+ p^! [-r] = \cH^0 \sigma^! \phi_{g,1}\oplus \phi_{g,\neq 1}(1) \Gamma_+ p^*[-r] \circ \bD_{E}.\]

Ignoring Tate twists, the functors $\phi_g \Gamma_+ p^![-r]$ and $\phi_g \Gamma_+ p^* [-r]$ agree. In particular, the computation of the $V$-filtration which we gave above can still be used for this composition of functors. The difficult part is that now we are considering $\cH^0 \sigma^!$. For this, we are interested in the $z$-monodromic pieces corresponding to $\chi = 0$ and $1$ as before. Using Theorem \ref{thm:main}, we have
\[
\begin{aligned} 
    \cH^0\sigma^!(\phi_g \Gamma_+\cM[w]) &= \bigoplus_{\lambda \in [0,1)} \cH^0\sigma^!(\phi_{g,\lambda} \Gamma_+ \cM[w])  \\
    &= \bigoplus_{\lambda \in [0,1)} \bigoplus_{\ell \in \Z}\text{ker}\left( E^\lambda_{\lambda,\ell}/E^{>\lambda}_{\lambda,\ell}\xrightarrow[]{z_i} \bigoplus_{ 1\leq i \leq r} E^\lambda_{\lambda-1,\ell}/E^{>\lambda}_{\lambda-1,\ell}\right).
\end{aligned}
\]

The $r-\chi$th monodromic piece of this is 
\[
\begin{aligned}
    \text{ker}\left( E^\lambda_{\lambda,\chi-\lambda}/E^{>\lambda}_{\lambda,\chi-\lambda}\xrightarrow[]{z_i} \bigoplus_{ 1\leq i \leq r} E^\lambda_{\lambda-1,\chi-\lambda}/E^{>\lambda}_{\lambda-1,\chi-\lambda}\right).
\end{aligned}
\]

As before, we want to give an isomorphism from this kernel to $\cM^{\chi}$. Similarly, using acyclicity of the Koszul-like complex, we can assume $\chi- \lambda \geq 0$ and show we have an isomorphism for such terms. We will use the same map $\varphi: E_{\lambda,\chi-\lambda}^\lambda \to \cM^{\chi}$ defined above.

\begin{lem} For $\chi - \lambda \geq 0$, the map $\varphi_{\lambda,\chi-\lambda}$ restricted to this kernel is an isomorphism.
\end{lem}
\begin{proof} Both surjectivity and injectivity make heavy use of the computation of the $V$-filtration given above. 

For surjectivity, when $\lambda = 0$, by Lemma \ref{VFiltContain}, we have that $F_{0,\ell}^{\ell} = \cM^{\ell} \otimes \de^{\ell} \subseteq E^0_{0,\ell}$ for any $\ell \geq 0$. Also, $\cM^{\ell+1}\otimes \de^{\ell} = F_{1,\ell}^{\ell} \subseteq E^1_{1,\ell}$, so $z_i \cM^{\ell}\otimes \de^{\ell} \subseteq \cM^{\ell+1}\otimes \de^{\ell} \subseteq E^1_{1,\ell}$ implies that this lies in the kernel. Hence, under $\varphi$, we get all of $\cM^{\ell}$, as desired. A similar proof works for $\lambda \in (0,1)$.

For injectivity, we use the fact that $\chi - \lambda \geq 0$ to apply Lemma \ref{kernel}. To see that our map $\varphi$ is injective, it suffices to show that if $m \in E^\lambda_{\lambda,\chi-\lambda}$ is such that $m \in \ker(\varphi) = \sum \text{im}(\de_{z_i})$ and $z_i m \in V^1$ for all $1\leq i \leq r$, then $m \in V^{1}$.

When $\lambda = 0$, using our explicit description of the $V$-filtration, we know that we can write
\[ m = \sum Q_{\ell,k,\alpha,\beta,\gamma,\eta} (\xi \de)^\ell \xi^k \de_z^\alpha w^\beta \de_w^\gamma z^{\eta} m_{\ell,k,\alpha,\beta,\gamma,\eta},\]
where $Q\in \cD_X$ and $m_{\ell,k,\alpha,\beta,\gamma,\eta} \in \cM^0\otimes 1$ or $\cM^r \otimes \de^r$. Moreover, since $z_i \cdot \cM^0 \otimes 1 \subseteq \cM^1 \otimes 1 \subseteq V^1$ and $z_i \cdot \cM^r \otimes \de^r \subseteq \cM^{r+1}\otimes \de^r = F_{1,r}^r \subseteq V^1$ (by Lemma \ref{VFiltContain}), we can assume $\eta =0$. Similarly, of course we can assume $k = 0$. So we really have
\[ m = \sum Q_{\ell,\alpha,\beta,\gamma} (\xi \de)^\ell \de_z^\alpha w^\beta \de_w^\gamma m_{\ell,\alpha,\beta,\gamma}.\]

As this is supposed to lie in $E_{\lambda,\chi-\lambda}$

By Lemma \ref{kernel}, we can assume $|\alpha| > 0$ for all $\alpha$, as $m$ is supposed to lie in the kernel of $\varphi$. Also, by taking homogeneous pieces with respect to the $E_{\beta,\ell}$ decomposition, we can assume $|\alpha|$ does not depend on $\alpha$. 

Now, the assumption is that $z_im \in V^1$ for all $1\leq i\leq r$. Writing this out, we have
\[ z_i m = \left(\sum Q_{\ell,\alpha,\beta,\gamma} (\xi\de)^\ell \alpha_i \de_z^{\alpha-e_i} w^\beta \de_w^\gamma m_{\ell,\alpha,\beta,\gamma} \right) + \left( \sum Q_{\ell,\alpha,\beta,\gamma} (\xi\de)^\ell \de_z^{\alpha} w^\beta \de_w^\gamma z_im_{\ell,\alpha,\beta,\gamma}\right),\]
and by the explicit description of the $V$-filtration (specifically, the fact that $\cM^1 \otimes 1, \cM^{r+1}\otimes \de^r \subseteq V^1$), we know $z_i m_{\ell,\alpha,\beta,\gamma} \in V^1$ already. Hence, the entire second sum lies in $V^1$, and so we see that actually the first sum lies in $V^1$, too. As $V^1$ is closed by the $\de_{z_i}$-action, we get that
\[\sum Q_{\ell,\alpha,\beta,\gamma} (\xi\de)^\ell \alpha_i \de_z^{\alpha} w^\beta \de_w^\gamma m_{\ell,\alpha,\beta,\gamma} \in V^1.\]

Finally, taking the sum over all $i$, we get $|\alpha| m \in V^1$, and so since $|\alpha|>0$, we conclude that $m \in V^1$, as desired.

The proof for $\lambda \in (0,1)$ is similar.
\end{proof}

This lemma shows that the $\cD$-module underlying $\cH^0 \sigma^! \phi_{g,1} \oplus \phi_{g,\neq 1}(1) \Gamma_+ p^*[-r]\cM$ is isomorphic to $\FL \cM$. Now, using the relation $p^!\cM[-r] \cong p^*(\cM(-r))[r]$, we get

\[ \cH^0 \sigma^! \phi_{g,1}\Gamma_+ p^*[r] \oplus \cH^0 \sigma^!\phi_{g,\neq 1}(1) \Gamma_+ p^*[r] = \overline{\FL},\]
and hence we have given a proof of Theorem \ref{FLDuality}.

\end{document}